\documentclass[11pt,a4paper,oneside]{amsart}
\usepackage{a4,a4wide}

\usepackage{geometry}
\usepackage{amsmath,amsfonts,amssymb,amsthm}
\usepackage{graphics}
\usepackage{dsfont}
\usepackage{epsfig}
\usepackage{esint}
\usepackage[colorlinks,linkcolor=blue]{hyperref}
\usepackage{comment}

\usepackage{tikz}
\usetikzlibrary{matrix}
\usetikzlibrary{quotes,angles,positioning}

\numberwithin{equation}{section}

\newcommand{\lessim}{\stackrel{<}{\sim}}

\newcommand{\ignore}[1]{}

\newtheorem{theorem}{Theorem}[section]

\newtheorem{proposition}[theorem]{Proposition}
\newtheorem{remark}[theorem]{Remark}
\newtheorem{lemma}[theorem]{Lemma}

\newtheorem{assumption}[theorem]{Assumption}

\renewcommand{\P}{{\mathbb P}}

\newcommand{\cc}{\text{C}}

\newcommand{\dd}{\mathrm{d}}

\newcommand{\bb}{\text{B}}
\newcommand{\LL}{\mathrm{L}}
\newcommand{\rhs}{r.~h.~s.\,}
\newcommand{\lhs}{l.~h.~s.\,}

\newcommand{\pp}{\mathrm{P}}

\newcommand{\m}{\mathrm{m}}

\usepackage{color}
\definecolor{darkred}{rgb}{0.9,0.1,0.1}
\definecolor{darkblue}{rgb}{0,0,0.7}
\definecolor{darkgreen}{rgb}{0,0.5,0}

\begin{document}
\title[Annealed quantitative estimates for the 2D-discrete random matching problem]{Annealed quantitative estimates for the quadratic 2D-discrete random matching problem}
\author{  Nicolas Clozeau \address[Nicolas Clozeau]{IST Austria, Austria} \email{nicolas.clozeau@ist.ac.at} \hspace*{0.5cm} Francesco Mattesini \address[Francesco Mattesini]{Universit\"at M\"unster \& MPI Leipzig,  Germany} \email{francesco.mattesini@uni-muenster.de}  } 
\thanks{NC has received funding from the European Research Council (ERC) under the Eu\-ropean Union’s Horizon 2020 research and innovation programme (grant agreement No 948819) \includegraphics[height=\fontcharht\font`\B]{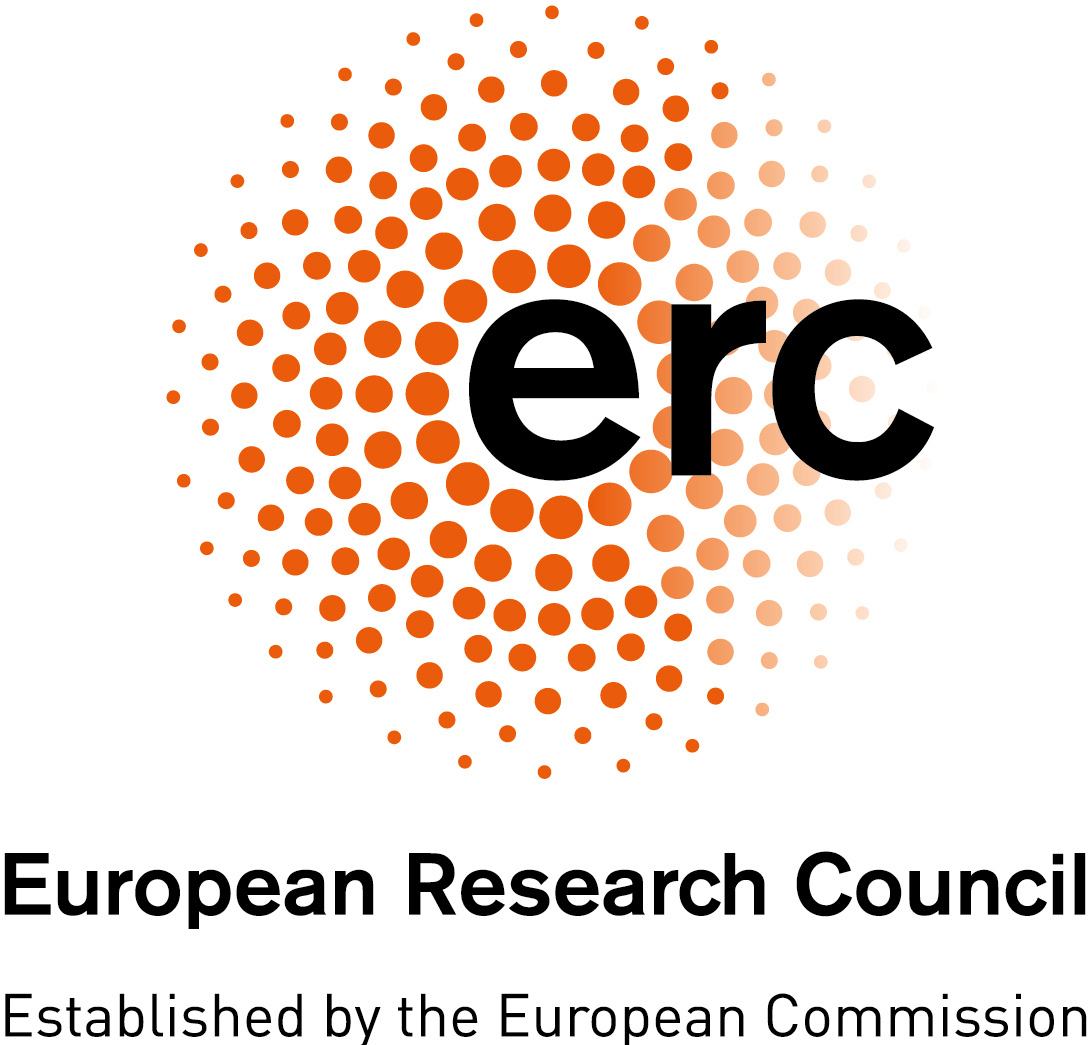}\,\includegraphics[height=\fontcharht\font`\B]{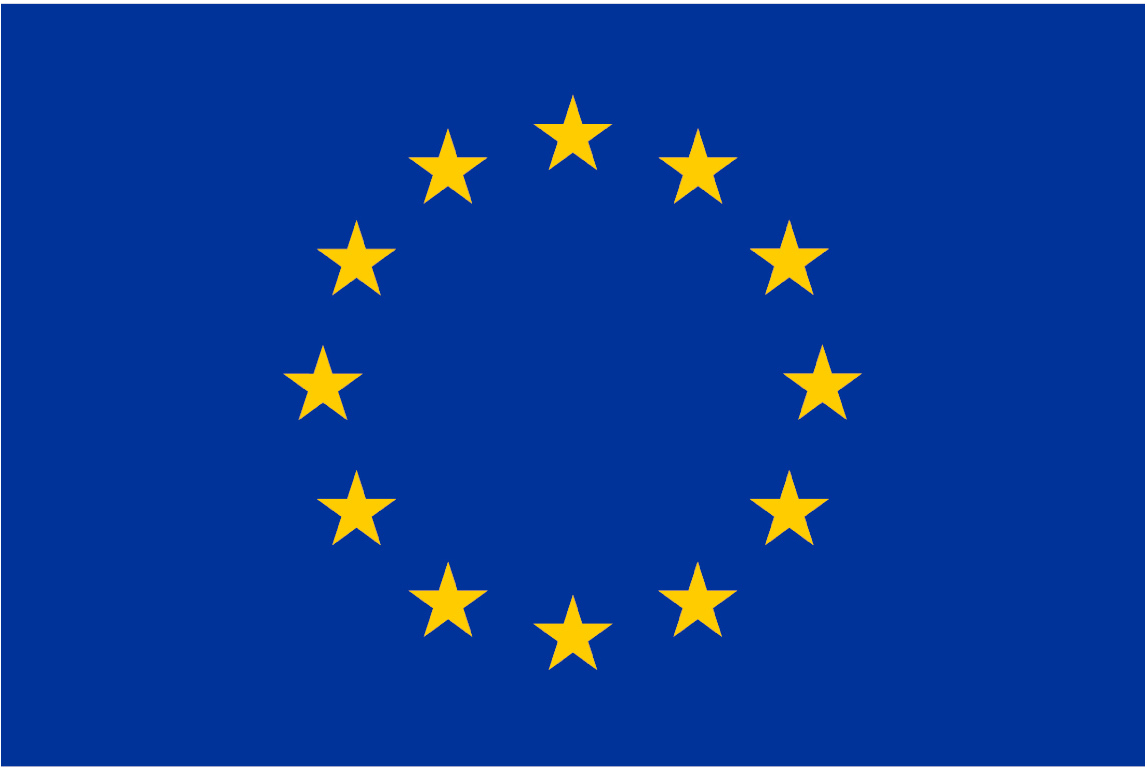}}
\thanks{FM is supported by the Deutsche Forschungsgemeinschaft (DFG, German Research Foundation) through the SPP 2265 {\it Random Geometric Systems}. FM has been funded by the Deutsche Forschungsgemeinschaft (DFG, German Research Foundation) under Germany's Excellence Strategy EXC 2044 -390685587, Mathematics M\"unster: Dynamics--Geometry--Structure. FM has been funded by the Max Planck Institute for Mathematics in the Sciences. }
%
%
\begin{abstract}
We study a 
random matching problem on closed compact $2$-dimensional Riemannian manifolds (with respect to the squared Riemannian distance), with samples of random points whose common law is absolutely continuous with respect to the volume measure with strictly positive and bounded density. We show that given two sequences of numbers $n$ and $m=m(n)$ of points, asymptotically equivalent as $n$ goes to infinity, the optimal transport plan between the two empirical measures $\mu^n$ and $\nu^{m}$ is quantitatively well-approximated by $\big(\mathrm{Id},\exp(\nabla h^{n})\big)_\#\mu^n$ where $h^{n}$ solves a linear elliptic PDE obtained by a regularized first-order linearization of the Monge-Ampère equation. This is obtained in the case of samples of correlated random points for which a stretched exponential decay of the $\alpha$-mixing coefficient holds and for a class of discrete-time sub-geometrically ergodic Markov chains having a unique absolutely continuous invariant measure with respect to the volume measure.

%
\end{abstract}
\maketitle
\begin{center}
\end{center}
\tableofcontents
\section{Introduction and statement of the main results}\label{introsec}
\subsection{The random matching problem and its asymptotic}\label{TMP}
%

The random matching problem is a popular optimization problem at the interface between analysis and probability with applications in many different fields such as statistical physics \cite{caracciolo2014scaling, MezPar}, computer science \cite{CSMatch} and economics \cite{EcMatch, GalShap}. Within the mathematical literature, it has been subject of intense studies due to its interactions with many areas, including for instance graph theory \cite{lovasz2009matching} and geometric probability \cite{steele1997probability}. In this paper we focus on one of its simple versions. Let $\{X_k\}_{1\leq k\leq n}$ and $\{Y_k\}_{1\leq k\leq m}$ (with possibly $m>n$) be two families of random points on a compact Riemannian manifold $\mathcal{M}$ (endowed with the Riemannian distance $\dd$). We are interested in the quadratic matching problem
\begin{equation}\label{eq:transintr}
\min_{\pi\in\Pi_{nm}} \sum_{i=1}^n \sum_{j=1}^m \pi_{ij}\,\mathrm{d}^2(X_i ,Y_j)
\end{equation}
where
$$\Pi_{nm}:=\bigg\{\pi\in [0,1]^{n\times m}\ \Big\vert\ \sum_{i=1}^n \pi_{ij} = \tfrac{1}{m}\quad\text{and}\quad\sum_{j=1}^{m} \pi_{ij} = \tfrac{1}{n}\bigg\}.$$
Classically, \eqref{eq:transintr} can be phrased in terms of a transport problem. Indeed, letting 
\begin{equation}\label{eq:empmeas}
\mu^n:=\frac1n\sum_{i=1}^n \delta_{X_i}\quad\text{and}\quad\nu^m:=\frac1m \sum_{j=1}^m \delta_{Y_j},
\end{equation}
be the empirical measures associated with the two point clouds, the linear programming problem \eqref{eq:transintr} amounts to determine the quadratic Wasserstein distance $\ W_2^2(\mu^n,\nu^m)$. 

\medskip
In the special case $n=m$ the Birkhoff-von Neumann  Theorem provides a correspondence between \eqref{eq:transintr} and the usual bipartite matching
\begin{equation}\label{eq:minweightgraph}
\min_{\sigma\in \mathcal{S}_{n}} \sum_{i=1}^n \mathrm{d}^2(X_i,Y_{\sigma(i)}),
\end{equation}
where $\mathcal{S}_{n}$ denotes the set of injective maps $\sigma: \{1, \dots, n\} \rightarrow \{1, \dots, n\}$. 
Indeed, since $\Pi_{nn}$ is a convex polytope, minimizers in \eqref{eq:transintr} have to be searched among extremal points. By the Birkhoff-von Neumann Theorem \cite[Lemma 2.1.3]{bapat}, the latter are nothing but permutation matrices (up to a factor $\tfrac{1}{n}$). 

%
%
%
%
%

%
\medskip
%

%
%
%
%

%
\medskip

A first natural question is to understand the asymptotics of \eqref{eq:transintr} as $n,m\uparrow\infty$. For the same number of samples $n=m$ and independently and identically distributed (i.i.d.) on the unit square $[0,1]^d$, the scaling of the cost \eqref{eq:transintr} has been well understood in the mathematical and statistical physics literature. A simple heuristic argument, see for instance \cite{MezPar}, suggests that given a point $X_i$, we can find a point $Y_j$ within a volume of order  $O(n^{-1})$ with high probability. For this reason, the typical inter-point distance is of order  $O(n^{-\frac1d})$ suggesting that the scaling of \eqref{eq:transintr} is of order  $O(n^{-\frac2d})$. Although attractive, this heuristic turns out to be unfortunately false in low dimension showing a critical behavior when $d=2$. This critical case is the one on which we focus on in this paper. Ajtai, Koml\'os and Tusn\'ady in \cite{AKT84} were the first to show that, for i.i.d. uniform samples, a logarithmic correction is needed, deriving\footnote{We use the notation $A \lesssim B$ if there exists a global constant $C>0$, which may only depend on $d$, such that $A\le CB$. We write $A\sim B$ if both $A \lesssim B$ and $B \lesssim A$ hold.}
\begin{equation}\label{eq:unifmatchingrates}
\mathbb{E}\big[W_2^2 (\mu^n, \nu^n)\big]\sim \frac{\log(n)}{n},
\end{equation}
extended later by Talagrand in \cite{talagrand1992ajtai} for clouds of i.i.d. points which are distributed accordingly to more general common law. A recent breakthrough was obtained within the physics community by Caracciolo, Lucibello, Parisi and Sicuro in \cite{caracciolo2014scaling}, and further developed by Caracciolo and Sicuro in \cite{caracciolo2015scaling} and by Sicuro in \cite{sicuro2017euclidean}, where the asymptotics of the cost are formally derived thanks to a novel PDE approach and optimal transport theory rather than combinatorics. A couple of years later, in general $2$-dimensional compact Riemannian manifolds without boundary, the first-order asymptotic has been rigorously justified by Ambrosio, Stra and Trevisan in \cite{ambrosio2019pde} for i.i.d. uniform samples and recently extended by Ambrosio, Goldman and Trevisan in \cite{ambrosio2021quadratic} for samples distributed accordingly to more general laws which are absolutely continuous (with H\"older continuous density) w.r.t. the volume measure $\dd \mathrm{m}$, leading to 
\begin{equation}\label{eq:AsymptoticCost}
\lim_{n\rightarrow\infty} \frac{n}{\log(n)} \mathbb{E}\big[W_2^2 (\mu^n,\nu^n)\big] = \frac{\vert\mathcal{M}\vert}{2\pi},
\end{equation}
where $|\mathcal{M}|$ denotes the Lebesgue measure of $\mathcal{M}$. The case $n\neq m$ with $n,m\uparrow\infty$ with similar rates is also covered, see \cite[Theorem 1.2]{ambrosio2021quadratic}. 

\medskip
The novel approach introduced in \cite{caracciolo2014scaling}, later revised in \cite{benedetto2020euclidean}, consists in a linearization of the Monge-Amp\`ere equation that allows for an explicit description of the cost thanks to the linearized proxies (see Section \ref{sec:lin} for more details). The aim of this work is to quantitatively justify the linearization ansatz in terms of convergence of the approximating minimizers of \eqref{eq:transintr} towards the optimal ones. In particular, we are interested in the case where the points are identically distributed with a common law $\rho\,\dd\m$ (where we recall that $\dd\m$ denotes the volume measure) and $\rho$ satisfies for some $\lambda,\Lambda>0$
%
\begin{equation}\label{Ellipticity}
\lambda\leq\rho\leq\Lambda.
\end{equation}
To the best of our knowledge, there are only few results on the asymptotic behavior of the transport map and they are so far limited to the case of i.i.d. uniform samples in the study of the semi-discrete matching problem (that is couplings between $\mu^n$ and $\dd\m$), see the work of Ambrosio, Glaudo and Trevisan in \cite{ambrosio2019optimal}. In connection with this work, quantitative estimates on the optimal map for the matching between the Lebesgue measure and Poisson clouds have been obtained by Goldman, Huesmann and Otto in \cite{GHO1} and Goldman and Huesmann in \cite{GH}. 

\medskip

Our extension in this paper is fourfold: First, we look at more general distribution of points and we consider the case of general densities $\rho$ satisfying \eqref{Ellipticity}. Second, we do not assume independence and we consider samples which may possess correlations. Third, we do not restrict the analysis to the semi-discrete matching problem and we also investigate the ansatz for the full matching problem \eqref{eq:transintr}. 
Finally, we investigate the case where the points are not identically distributed and we extend our result to a class of sub-geometrically ergodic Markov chains.

\medskip

We finally mention that the effectiveness of the linearization ansatz introduced in \cite{caracciolo2014scaling} is not only limited to the case of i.i.d. distributed points on bounded domains, but it can be employed in many different settings. See for instance \cite{CagliotiPieroni} for an interesting application to the matching on unbounded domains, \cite{Ledoux, Ledoux2, Ledoux3} for an application to Gaussian matching, \cite{Jonas} for an application in random matrix theory, \cite{Wang2, Wang3, Wang, HMTfBm} for an application to a continuous instance of the matching problem, i.e. when the empirical measure is replaced by the occupation measure of a stochastic process. It is worth to further mention that these techniques can also be employed when considering the matching problem with $p$-costs in higher dimension, see \cite{GoldTrev} and when considering a larger class of optimization problems, see \cite{GoldTrevOpt}.

\subsection{Linearization ansatz}\label{sec:lin}
We now briefly reproduce the linearization ansatz introduced in \cite{caracciolo2014scaling}. For simplicity, we consider the case $\mathcal{M} = \mathbb{T}^2$, $n=m$ and i.i.d. samples with common distribution $\rho \, \dd \mathrm{m}$. Let $T^n$ be an optimal transport map (whose existence is ensured by Brenier's Theorem \cite{Brenier}) between $\mu^n$ and $\nu^n$. Based on the transport relation $T^n_{\#}\mu^n=\nu^n$ and a change of variables, $T^n$ solves (formally) the Monge-Amp\`ere equation 
\begin{equation}\label{eq:MongAmp}
(\nu^n\circ T^n) \det(\nabla T^n) = \mu^n.
\end{equation}
Since the cost is quadratic, by \cite[Theorem 1.25]{Santambrogio}, there exists a function $h^n$ such that $T^n= \mathrm{Id}+\nabla h^n$. Applying the law of large numbers, we further have the weak convergences $\mu^n,\nu^n {\rightharpoonup}\rho\, \dd \mathrm{m}$ as $n\uparrow \infty$ so that we expect $T^n\approx \mathrm{Id}$ as $n\uparrow\infty$. Thus, this suggests that the correction $\nabla h^n$ is small as $n\uparrow\infty$, allowing to perform (formally) the Taylor expansions
\begin{equation}\label{eq:TaylExp}
\nu^n \circ T^n \approx \nu^n + \nabla \nu^n\cdot\nabla h^n\quad\text{and}\quad \det(\nabla T^n) \approx 1 + \Delta h^n.
\end{equation}
Plugging \eqref{eq:TaylExp} into \eqref{eq:MongAmp}, neglecting the higher order terms and replacing $\mu^n$ by $\rho$ yields  
\begin{equation}\label{eq:ApproxPDE}
\nabla \cdot \rho \nabla h^n = \mu^n-\nu^n.
\end{equation}
This formal linearization suggests the following two conjectures 
\begin{equation}\label{eq:approxcost}
\lim_{n\uparrow\infty}\bigg\vert\mathbb{E}\big[W_2^2 (\mu^n, \nu^n)\big]-\mathbb{E}\bigg[\int_{\mathbb{T}^2} |\nabla h^n|^2\rho\,\dd\m\bigg]\bigg\vert=0\quad\text{and}\quad\lim_{n\uparrow\infty}\int_{\mathbb{T}^2}\vert T^n-(\text{Id}+\nabla h^n)\vert^2=0\text{ a.s.}
\end{equation}
Unfortunately, \eqref{eq:approxcost} cannot hold as it is, since the solution of \eqref{eq:ApproxPDE} does not belong to $\mathrm{H}^1$ due to the roughness of the source term. To overcome this, following the strategy in \cite{ambrosio2019pde}, a regularization using the heat-semigroup at time $\sim \frac{1}{n}$ (up to logarithmic corrections) is made. Doing so, the first item of \eqref{eq:approxcost} turns out to be true, leading to the result \eqref{eq:AsymptoticCost} (see for instance \cite{ambrosio2019finer} for a convergence rate). 

\medskip
\subsection{Formulation of the main results}
%
For the remainder of the paper $\mathcal{M}$ denotes a $2$-dimensional connected and compact Riemannian manifold without boundary (or the square $[0,1]^2$) endowed with the Riemannian distance $\dd$. For $t>0$ we denote by $p_t$ the fundamental solution of the heat operator $\partial_t-\Delta$ on $\mathcal{M}$, where $\Delta$ denotes the Beltrami-Laplace operator. We define the heat semigroup $(\text{P}_t)_{t>0}$ via its action on probability measures $\mu \in \mathcal{P}(\mathcal{M})$ and square integrable functions $f\in\LL^2(\mathcal{M})$
\[
\text{P}_t \mu := \int_\mathcal{M} p_t(\cdot,y)\dd \mu(y) \quad\text{and}\quad \text{P}_t f:= \int_{\mathcal{M}} p_t(\cdot,y) f(y)\dd \m(y).
\]
%

\medskip
We first introduce the class of correlated point clouds that we consider for studying the matching problem \eqref{eq:transintr}. This class concerns point clouds $\{X_i\}_i$ for which the correlations between points decay at an exponential rate, where the correlations are measured in terms of the $\alpha$-mixing coefficient given by, for any $\ell\geq 1$
\begin{equation}\label{AlphaMixing}
\alpha_\ell:=\sup_{k\geq 1}\sup\bigg\{\frac{\text{cov}(f,g)}{\|f\|_{\LL^{\infty}}\|g\|_{\LL^{\infty}}},\quad f\in\LL^{\infty}(\sigma\{X_j,j\leq k\})\quad \text{and}\quad g\in\LL^{\infty}(\sigma\{X_j,j\geq \ell+k\})\bigg\},
\end{equation}
and the $\beta$-mixing coefficient given by\footnote{We denote by $\mathbb{P}_X$ the law of a random variable $X$.}
\begin{equation}\label{BetaMixing}
\beta_{ij}:=\sup_{\vert F\vert\leq 1}\bigg\vert\int_{\mathcal{M}\times\mathcal{M}}F\,\dd (\mathbb{P}_{(X_i,X_j)}-\mathbb{P}_{X_i}\otimes\mathbb{P}_{X_j})\bigg\vert\quad \text{for any $i,j\geq 1$}.
\end{equation}
\begin{assumption}[Correlated point clouds]\label{Assumptions}
We consider point clouds $\{X_i\}_{i}\subset \mathcal{M}$ which are identically distributed according to $\rho\,\dd\m$ where $\rho$ satisfies \eqref{Ellipticity}. We further assume decay of the correlations in the form of 
\begin{equation}\label{DecatBetaMixing}
\sup_{n\geq 1}\frac{1}{n}\sum_{1\leq i<j\leq n}\beta_{ij}<\infty,
\end{equation}
and there exist $a,b>0$ and $\eta\in (0,\infty]$ such that 
\begin{equation}\label{DecayAlphaMixing}
\alpha_\ell\leq a\exp\big(-b\ell^{\eta}\big)\quad\text{for any $\ell\geq 1$.}
\end{equation}
\end{assumption}
Assumption \ref{Assumptions} is made to ensure good concentration properties of the point clouds. On the one hand, under \eqref{DecatBetaMixing}, the cost $W^2_2(\mu^n,\nu^n)$ behaves as in the i.i.d. case \eqref{eq:unifmatchingrates} (cf. \cite[Theorem 2]{borda2021empirical} and Appendix \ref{sec:matchcost}). On the other hand, the sub-exponential decay \eqref{DecayAlphaMixing} of the $\alpha$-mixing coefficient ensures sub-exponential concentration properties (cf. \cite[Theorem 1]{merlevede2011bernstein} and Proposition \ref{BersteinCorrelated}), which is necessary to run our argument. We refer the reader to Section \ref{StrategyOfProofSection} for further technical details.

\medskip

Our first main result concerns the approximation of transport plans coupling $\{\mu^n\}_n$ and $\{\nu^m\}_m$ defined in \eqref{eq:empmeas}. We justify the formal linearization of the Monge-Amp\`ere equation achieved in Section \ref{sec:lin} in an annealed quantitative way (i.e. in expectation): We show that, a suitable regularization of, the plan $\big(\text{Id},\exp(\nabla h^n)\big)_\#\mu^n$, with $h^n$ defined in \eqref{eq:ApproxPDE}, provides a good approximation when measuring the error with respect to the $W_2$-Wasserstein distance in the product space $\mathcal{M}\times\mathcal{M}$ endowed with the metric 
\begin{equation}\label{DistanceTransportPlan}
\delta^2 \big((x,y),(z,w)\big):=\dd^2(x,z)+\dd^2(y,w).
\end{equation}
The density $\rho$ will need further regularity in form of fractional Sobolev spaces defined as, for some $\varepsilon>0$
\begin{equation}\label{DefHepsilon}
\mathrm{H}^{\varepsilon}:=\Big\{f: \mathcal{M}\rightarrow\mathbb{R}\ \Big\vert \|f\|^2_{\mathrm{H}^\varepsilon}:= \sum_{k\geq 1}\lambda^{2\varepsilon}_k\vert \hat{f}(k)\vert^2<\infty\Big\},
\end{equation}
where $\{\lambda_k,\phi_k\}_{k}$ denote the eigenvalues and eigenvectors of $-\Delta$ on $\mathcal{M}$ and $\hat{f}(k)=\int_{\mathcal{M}}f\,\phi_k$ denotes the Fourier modes of $f$.
Finally, we denote by $\dot{\mathrm{H}}^1$ the $\LL^2$-based Sobolev space 
\begin{equation}\label{eq:H1}
\dot{\mathrm{H}}^1:=\Big\{f: \mathcal{M}\rightarrow \mathbb{R}\ \Big\vert\ \int_{\mathcal{M}}\vert \nabla f\vert^2<\infty\quad\text{with } \int_{\mathcal{M}} f\dd\m=0\Big\}.
\end{equation}
\begin{theorem}[Approximation of the transport plan]\label{MainResultMatchingClouds}
Let $\rho\in \mathrm{H}^{\varepsilon}$ for some $\varepsilon>0$ satisfying \eqref{Ellipticity} and $\{\mu^n\}_n$ and $\{\nu^{m}\}_m$ be defined in \eqref{eq:empmeas} (for $m=m(n)$ with some given increasing map $m:\mathbb{N}\rightarrow\mathbb{N}$) with point clouds satisfying Assumption \ref{Assumptions} and such that there exists $q\in [1,\infty)$ for which $\frac{m(n)}{n}\underset{n\uparrow\infty}{\rightarrow} q$. We consider\footnote{where we impose additional Neumann boundary conditions in the case $\mathcal{M}=[0,1]^2$} $h^{n,t}\in \dot{\mathrm{H}}^1$ the weak solution of
\begin{equation}\label{eq:f^nmt}
\nabla\cdot \rho\nabla h^{n,t} =\mu^{n,t}-\nu^{m,t},
\end{equation}
for any $t\in(0,1)$ with $\mu^{n,t}:=\pp_t\mu^n$ and $\nu^{m,t}:=\pp_t\nu^m$.

\medskip

There exist an exponent $\kappa>0$, a deterministic constant $C$ and a random variable $\mathcal{C}_n$ both depending on $\lambda, \Lambda$ and $\mathcal{M}$ for which given $t=\frac{\log^{\kappa}(n)}{n}$ and
%
\begin{equation}\label{DefApproxPlan}
\gamma^{n,t}:=\big(\mathrm{Id},\exp(\nabla h^{n,t})\big)_\#\mu^{n,t},
\end{equation}
it holds
\begin{equation}\label{ApproximationTransportPlan}
\inf_{\pi}W^2_2(\pi,\gamma^{n,t})\leq \mathcal{C}_n\frac{\log(n)}{n}\sqrt{\frac{\log\log(n)}{\log(n)}}\quad\text{with \,$\sup_{n\geq 1}\mathbb{E}[\tfrac{1}{C}\mathcal{C}_n]\leq 1$,}
\end{equation}
where the $\inf$ runs over all optimal transport plans $\pi$ between $\mu^n$ and $\nu^m$.

\medskip

Furthermore, if \eqref{DecayAlphaMixing} holds with $\eta>2$, the assumption \eqref{DecatBetaMixing} can be dropped and it holds
\begin{equation}\label{ApproximationTransportPlanBis}
\inf_{\pi}W^2_2(\pi,\gamma^{n,t})\leq \mathcal{C}_n\frac{\log(n)}{n}\sqrt{\frac{\log\log(n)}{\log^{1-\frac{2}{\eta}}(n)}}\quad\text{with }\sup_{n\geq 1}\mathbb{E}[\exp(\tfrac{1}{C}\mathcal{C}^{\frac{1}{2}}_n))]\leq 2,
\end{equation}
where the $\inf$ runs over all optimal transport plans $\pi$ between $\mu^n$ and $\nu^m$.
%
%
\end{theorem}
Our second main result concerns the particular case of the semi-discrete matching problem, i.e. optimal coupling between the common law $\rho\,\dd \m$ and $\{\mu^n\}_n$. We know from McCann's theorem \cite{McCann2001} that there exists a unique optimal transport map $T^n$, that is the optimal transport plan $\pi^n$ can be written as 
$$\pi^n=\big(\text{Id},T^n\big)_\#\rho\,\dd m.$$
We show that $T^n$ can be approximated in an annealed quantitative way in $\LL^2$ by (a suitable regularized version of) the solution of \eqref{eq:ApproxPDE} with $\nu^m$ replaced by $\rho$.
%
%
%
%
%
%
%
%
%
%

%
%

%
%
%

%
%
\begin{theorem}\label{th1}
Let $\rho\in \mathrm{H}^{\varepsilon}$ for some $\varepsilon>0$ satisfying \eqref{Ellipticity} and $\{\mu^{n}\}_n$ be defined in \eqref{eq:empmeas} with a point cloud satisfying Assumption \ref{Assumptions}. We consider $f^{n,t}\in {\rm \dot{H}^1}$ the weak solution of
\begin{equation}\label{eq:f^ntBis}
\nabla\cdot \rho\nabla f^{n,t} =\mu^{n,t}-\rho_t,
\end{equation}
where, for all $t\in (0,1)$, we recall that $\mu^{n,t}=\pp_t\mu^n$ and $\rho_t=\pp_t\rho$. Finally, we denote by $T^n$ the optimal transport map from $\rho\,\dd \m$ to $\mu^n$. 
\medskip

There exist an exponent $\kappa>0$, a deterministic constant $C$ and a random variable $\mathcal{C}_n$ both depending on $\lambda,\Lambda$ and $\mathcal{M}$ for which given $t=\frac{\log^{\kappa}(n)}{n}$, it holds
\begin{equation}\label{ApproximationTransportMap}
\int_\mathcal{M} \mathrm{d}^2\big(T^n,\mathrm{exp}(\nabla f^{n,t})\big)\dd \mathrm{m}\leq \mathcal{C}_n\frac{\log(n)}{n}\sqrt{\frac{\log\log(n)}{\log(n)}}\quad\text{with \,$\sup_{n\geq 1}\mathbb{E}[\tfrac{1}{C}\mathcal{C}_n]\leq 1$.}
\end{equation}
%
%

%
\medskip

Furthermore, if \eqref{DecayAlphaMixing} holds with $\eta>2$, the assumption \eqref{DecatBetaMixing} can be dropped and it holds
\begin{equation}\label{ApproximationTransportMapBis}
\int_\mathcal{M} \mathrm{d}^2\big(T^n,\mathrm{exp}(\nabla f^{n,t})\big)\dd \mathrm{m}\leq \mathcal{C}_n\frac{\log(n)}{n}\sqrt{\frac{\log\log(n)}{\log^{1-\frac{2}{\eta}}(n)}}\quad\text{with }\sup_{n\geq 1}\mathbb{E}[\exp(\tfrac{1}{C}\mathcal{C}_n^\frac12))]\leq 2.
\end{equation}
%
%
\end{theorem}
We finally mention that in the case where the eigenfunctions $\{\phi_k\}_{k}$ admit a uniform bound, the conclusions \eqref{ApproximationTransportPlanBis} and \eqref{ApproximationTransportMapBis} can be improved. We comment on the proof at the end of Sub-section \ref{ProofTransport}.
\begin{remark}\label{FlatGeometry}
Let $\{\mu^n\}_n$ and $\{\nu^m\}_m$ be as in Theorem \ref{MainResultMatchingClouds}. We assume that the family of eigenfunctions $\{\phi_k\}_k$ satisfies the uniform bound
\begin{equation}\label{FlatenessAss}
\sup_{k\geq 1}\|\phi_k\|_{\LL^\infty}<\infty.
\end{equation}
Then, \eqref{ApproximationTransportPlanBis} and \eqref{ApproximationTransportMapBis} hold true for $\eta>1$ with a convergence rate $\frac{\log(n)}{n}\sqrt{\frac{\log\log(n)}{\log^{1-\frac{1}{\eta}}(n)}}$ and the same stochastic integrability. Note that \eqref{FlatenessAss} typically holds when the geometry of $\mathcal{M}$ is flat, see \cite{toth2002riemannian}.
\end{remark}

\medskip

Theorem \ref{MainResultMatchingClouds} and Theorem \ref{th1} are not restricted to the case of identically distributed point clouds and we present in the next section and in Appendix \ref{ClassMarkovAppendix} a possible extension, using the same techniques, to a class of sub-geometrically ergodic discrete-time Markov chains.

\subsection{Extension to a class of sub-geometrically ergodic Markov chains}\label{Examples}
%
We first recall some basic facts on discrete-time Markov chains on $\mathcal{M}$. Such a Markov process is described by its initial distribution $\mu_0 \in \mathcal{P}(\mathcal{M})$ and its transition kernel $K$, that is a measurable map from $\mathcal{M}$ to the space of probability measures $\mathcal{P}(\mathcal{M})$. We recall that $K$ acts on $\mathcal{P}(\mathcal{M})$ in form of 
\begin{equation}\label{eq:actionkernel}
(K\mu)(A) = \int_{\mathcal{M}} K(\cdot,A)\,\dd\mu\quad\text{for every Borel set $A\subset \mathcal{M}$ and $\mu\in \mathcal{P}(\mathcal{M})$,}
\end{equation}
and likewise on bounded measurable function $\psi$ in form of
\begin{equation}\label{eq:actionkernelBis}
(K\psi)(x) = \int_{\mathcal{M}} \psi(y) \,K(x, \dd y)\quad\text{for any $x\in\mathcal{M}$.}
\end{equation}
%
Given an initial distribution $\mu_0\in \mathcal{P}(\mathcal{M})$, we recall that the law of a Markov chain $\{X_n\}_{n\geq 0}$ can be expressed with the help of the transition kernel, namely 
\begin{equation}\label{LawOfTheChain}
X_n\sim K^n\mu_0\quad\text{for any $n\geq 0$,}
\end{equation}
where $K^n$ stands for the $n^{\mathrm{th}}$-iteration of the kernel $K$. 

\medskip

We now introduce the class of discrete-time Markov chains we consider.
\begin{assumption}\label{ClassMarkov}
Let $\mathcal{M}$ be a $2$-dimensional compact Riemannian manifold. Let $\mu_0\in \mathcal{P}(\mathcal{M})$ and $\{X_n\}_{n\geq 1}\subset \mathcal{M}$ be a Markov chain with initial law $\mu_0$. We assume that the chain satisfies the two following conditions:
\begin{itemize}
\item[(i)]We assume that there exists a measurable function $k : \mathcal{M}\times\mathcal{M}\rightarrow [0,\infty)$ and $\lambda,\Lambda>0$ such that for any Borel set $A\subset\mathcal{M}$
\begin{equation}\label{AbsoluteContinuity}
K(x,A)=\int_{A}k(x,\cdot)\dd \m\quad\text{with $\lambda\leq k\leq \Lambda$}.
\end{equation}
\item[(ii)]We assume that there exist an unique invariant measure $\mu_\infty$, i.e. $K\mu_\infty=\mu_\infty$, constants $a,b>0$ and $\eta\in (0,1]$ as well as a map $V : \mathcal{M}\rightarrow [1,\infty)$ such that $\int_\mathcal{M}V\dd\mu_\infty,\int_{\mathcal{M}}V\dd\mu_0<\infty$, for which for any $\ell\geq 1$ and any 
$\phi\in\LL^\infty(\mathcal{M})$
\begin{equation}\label{ConvergenceRate}
\|K^\ell\phi-\mu_\infty(\phi)\|_V\leq a\exp(-b\ell^{\eta})\|\phi\|_\infty\quad\text{for any $x\in\mathcal{M}$},
\end{equation}
where 
$$\|\phi\|_V:=\sup \frac{\vert \phi\vert}{1+V}\quad\text{and}\quad\mu_\infty(\phi)=\int_{\mathcal{M}}\phi\,\dd\mu_\infty.$$
\end{itemize}
\end{assumption}
We now comment on the consequences of the above assumptions. First, the condition \eqref{AbsoluteContinuity} ensures that the invariant measure $\mu_\infty$ is absolutely continuous with respect to the volume measure with bounded density, namely
\begin{equation}\label{AbsolutelyContinuousInvariantMeasure}
\mu_\infty=\rho\,\dd \m\quad\text{with $\lambda\leq \rho\leq \Lambda$.}
\end{equation}
We briefly give the argument for \eqref{AbsolutelyContinuousInvariantMeasure}. Using the second item of \eqref{AbsoluteContinuity}, we have that the operator 
$$\mathcal{I}: f\in\LL^1\mapsto \int_{\mathcal{M}}k(\cdot,y)f(y)\dd\m(y)\quad\text{is compact.}$$
Furthermore, we note that the closed convex set $\mathcal{C}:=\{f\in \LL^1\,\vert\, \lambda\leq f\leq \Lambda\text{ and }\int_{\mathcal{M}}f=1\}$ is invariant under the action of $\mathcal{I}$. Therefore Schauder's fixed point theorem implies that $\mathcal{I}$ admits a fixed point in $\mathcal{C}$. Given such a fixed point $\rho$, it is clear that $\mu_\infty$ defined in \eqref{AbsolutelyContinuousInvariantMeasure} is an invariant measure according to \eqref{AbsoluteContinuity}. 

\medskip

Second, for irreducible and aperiodic Markov chains, the condition \eqref{ConvergenceRate} is satisfied when the kernel satisfies the following geometric drift condition: There exist a function $V: \mathcal{M}\rightarrow [1,\infty)$, a petite set $\mathcal{P}$ and a constant $C>0$ such that $\sup_{\mathcal{P}} V<\infty$ and
$$KV+\phi\circ V\leq V+C\mathds{1}_\mathcal{P},$$
with for large $x\geq 1$, $\phi(x)=c\frac{x}{\log^\alpha(x)}$ for some $\alpha\geq 0$ and $c>0$ (\eqref{ConvergenceRate} is then satisfied with $\eta=\frac{1}{1+\alpha}$), see \cite[Theorem 2.8 \& Section 2.3]{douc2004practical} and the references therein. Moreover, the assumption \eqref{ConvergenceRate} implies the sub-exponential decay of the $\beta$-mixing coefficient \eqref{BetaMixing}, namely there exists a constant $C$ depending on $\lambda$, $\Lambda$ and $\m(\mathcal{M})$ such that
\begin{equation}\label{eq:markchainbetamixc}
\beta_{ij}\leq C\exp(-b\vert i-j\vert^\eta)\quad\text{for any $i,j$,}
\end{equation}
which ensures that \eqref{DecatBetaMixing} and \eqref{DecayAlphaMixing} hold and ensures good concentration property of the Markov chain, see Proposition \ref{BersteinCorrelated}. The estimate \eqref{eq:markchainbetamixc} can be seen as a direct consequence of the combination of the estimate on the $\beta$-mixing coefficient in \cite[Proposition 3]{liebscher2005towards} and the assumptions \eqref{AbsoluteContinuity} and \eqref{ConvergenceRate}.
%
%

%
\medskip

Finally, the condition \eqref{ConvergenceRate} quantifies the weak convergence of the law of the Markov chain to its stationary distribution, namely there exists a constant $C>0$ such that for any $f \in \LL^\infty(\mathcal{M})$
\begin{equation}\label{eq:lemMark}
|\mathbb{E}[f(X_n)] - \mu_\infty(f)| \leq C\exp(-b n^\eta) \| f\|_{\LL^\infty}.
\end{equation}
We shortly give the argument. We first notice that a direct inductive argument together with the semigroup property $K^{n_1+n_2} = K^{n_1}K^{n_2}$ for every $n_1,n_2>0$ and Fubini's theorem gives
\begin{equation}\label{eq:swapkern}
\int_{\mathcal{M}} f\,\dd K^n\mu_0=  \int_{\mathcal{M}} K^n f\, \dd\mu_0\quad\text{for any $n\geq 1$.}
\end{equation}
The combination of \eqref{LawOfTheChain}, \eqref{eq:swapkern} and \eqref{ConvergenceRate} gives
\begin{align*}
|\mathbb{E}[f(X_n)] - \mu_\infty(f)| & \stackrel{\eqref{LawOfTheChain},\eqref{eq:swapkern}}{=} \Big| \int_{\mathcal{M}} (K^n f - \mu_\infty(f))\,\dd\mu_0 \Big| \\ & \stackrel{\eqref{ConvergenceRate}}{\le}  \Big(\int_\mathcal{M} 1+V\dd\mu_0\Big)a\exp(-b n^\eta) \|f\|_\infty.
\end{align*}
A classical example of a Markov chain satisfying Assumption \ref{ClassMarkov} is given by iterated function systems with additive noise. For simplicity, let $\mathcal{M}=\mathbb{T}^2$. Let $\{\theta_n\}_{n\geq 1}$ be i.i.d. random variables with common law $h\,\dd\m$ for some $h$ satisfying $\lambda\leq h\leq \Lambda$. Let $F:\mathbb{T}^2\rightarrow\mathbb{T}^2$ be a contraction, i.e. there exists a constant $L<1$ such that
\begin{equation}\label{eq:lipcont}
\vert F(x)-F(y)\vert \le L \vert x-y\vert\quad\text{for any $x,y\in\mathbb{T}^d$.}
\end{equation}
We define the iterated function system $\{X_n\}_{n\geq 1}$ according to the induction
$$X_{n+1}=F(X_n)+\theta_n\quad\text{for any $n\geq 1$}.$$
The kernel is given by 
$$K(x,A)=\int_{\mathbb{T}^2}\mathds{1}_A(F(x)+\theta)h(\theta)\dd\m(\theta),$$
so that $K$ satisfies \eqref{AbsoluteContinuity} with 
$$k(x,\cdot)=h(\cdot-F(x)).$$
Moreover, the condition \eqref{eq:lipcont} ensures the validity of \eqref{ConvergenceRate}, see for instance \cite[Theorem 3.2]{alsm03}. Thus the Markov process $\{X_n\}_{n\ge1}$ satisfies Assumption \ref{ClassMarkov}.

\medskip
We show in Appendix \ref{ClassMarkovAppendix} that the conclusions of Theorem \ref{th1} and Theorem \ref{MainResultMatchingClouds} hold true for Markov chains satisfying Assumption \ref{ClassMarkov}.
\subsection{Open problems}
%
%
%
%

We conclude this section with open questions that arise in view of our results. Those concern optimality of our convergence rates, extensions to more general costs and different type of correlated point clouds. For the latter, we mention two possible directions concerning the Ginibre ensemble and Coulomb gases that we think are worth investigating.
\begin{enumerate}
\item \emph{Sharpness of the rates. }The convergence rate in \eqref{ApproximationTransportPlan} and \eqref{ApproximationTransportMap} match with the one obtained for the case of uniformly distributed samples in \cite{ambrosio2019optimal}. However, even in the latter case, it has not been shown whether the rate is optimal and we suspect the opposite. A possible way to track the optimal rate could be to perform a second-order linearization of the Monge-Ampère equation \eqref{eq:MongAmp}. Following the same type of computations leading to \eqref{eq:ApproxPDE} in the case $\rho=1$, a second-order linearization $q^n$ should solve 
$$-\Delta q^n=\text{det}(\nabla^2 h^n),$$
where we recall that $h^n$ solves \eqref{eq:ApproxPDE}, providing the conjecture 
$$\lim_{n\uparrow\infty}\bigg\vert\int_{\mathbb{T}^2}\vert T^n-(\text{Id}+\nabla h^n)\vert^2-\int_{\mathbb{T}^2}\vert\nabla q^n\vert^2\bigg\vert=0.$$
%


%
\medskip

\item \emph{Extension to $p$-costs. }A natural question is to investigate if our results hold for different cost functions as $p$-cost functions for $p>1$. The behavior of the cost has been optimally quantified in \cite{AKT84,boutet2002almost,Barthe2013,Fournier2015}. However, to the best of our knowledge, quantitative estimates on the transport plan in the setting of general $p$-costs are not known, even for uniformly distributed samples.
A possible approach would be to revise the linearization ansatz of \cite{caracciolo2014scaling} for general $p$-costs. Indeed, if the transport cost between two points $x,y$ is given by $\frac{1}{p}|x-y|^p$ on the torus, Gangbo-McCann's theorem \cite[Theorem 1.2]{Gangbo1996} ensures that there exists a map $h^n$ such that the optimal transport map $T^n$ takes the form $T^n= \text{Id} +  |\nabla h^n|^{p'-2}\nabla h^n$, where $p'$ denotes the conjugate exponent. Therefore, following the same type of computations leading to \eqref{eq:ApproxPDE}, a first-order linearization should solve the following degenerate $p'$-Laplace equation
$$\nabla \cdot \rho|\nabla h^n|^{p'-2}\nabla h^n=\mu^n- \nu^m,$$
and we may expect 
$$\lim_{n\uparrow\infty}\int_{\mathbb{T}^2}\vert T^n-(\text{Id}+|\nabla h^n|^{p'-2}\nabla h^n)\vert^{p}=0.$$
See also \cite{Lukas} for a justification of this linearisation ansatz down to mesoscopic scales based on a large-scale regularity theory for the Monge-Amp\`ere equation.

\medskip

\item \emph{Ginibre ensemble. }A (complex) Ginibre ensemble is a non-Hermitian random matrix with independent complex Gaussian entries. Given a $n\times n$ Ginibre ensemble $X$, we define its empirical spectral distribution as 
\[
\mu^n = \frac 1n \sum_{i=1}^n \delta_{\lambda_i },
\]
where $\{\lambda_i\}_{i=1}^n$ are the eigenvalues of the matrix $\frac{X}{\sqrt n}$. The so called Circular Law states that, almost-surely, $\mu^n$ weakly converges to the uniform distribution on the complex unit disk $\mu^\infty$ having Lebesgue density $\frac1\pi \mathds{1}_{\bb_1}$, see for instance \cite[Theorem 1.10]{TaoVu}. An interesting question is to quantitatively understand the weak convergence of $\mu^n$ towards $\mu^\infty$ measured Wasserstein distances. A possible approach to this problem would be to employ the linearization argument discussed in Section \ref{sec:lin}. Note that the Ginibre ensemble is not covered by our setting as it is posed on the whole space $\mathbb{R}^2$ and, in general, the eigenvalues $\{\lambda_i\}_{i=1}^n$ possess long-range correlations and therefore do not satisfy our Assumption \ref{Assumptions}. However, in \cite[Th\'eor\`eme 3.1.1]{maxime}, the linearization argument has been shown to be robust enough to derive an upper bound on $\mathbb{E}[W_2(\mu^n, \mu^\infty)]$ (see also \cite[Theorem 1.3]{Jonas} for similar results for $\mathbb{E}[W_1(\mu^n, \mu^\infty)]$). 
The techniques used in \cite{maxime, Jonas} avoid a quantification of the correlations between the eigenvalues: This is done in \cite{maxime} by a decomposition argument together with concentration estimates of the Wasserstein distance around the Moser's coupling and in \cite{Jonas} by using classical tools from non-Hermitian random matrix theory. A challenging question would be to investigate the exact asymptotics of the transport cost $\mathbb{E}[W_2(\mu^n, \mu^\infty)]$ in the case of the Ginibre ensemble complementing, \cite[Th\'eor\`eme 3.1.1]{maxime} and \cite[Theorem 1.3]{Jonas} with a lower bound and consequently quantify the convergence of the linearized proxies to the optimal transport map. 

%

\medskip
\item \emph{Planar Coulomb gases. }Planar Coulomb gases are many  particles systems, in which the particles $\{X_i\}_{i=1}^n$ have repulsively Coulomb interactions and are confined by a potential $V: \mathbb{T}^2 \rightarrow \mathbb{R}$. These are modelled by the Hamiltonian
\begin{equation}\label{eq:hamiltonian}
\mathcal{H}_n(\{X_i\}_{i=1}^n) = - \sum_{i \neq j} \log |X_i - X_j| + n \sum_{i=1}^n V(X_i)
\end{equation}
and the Gibbs measure
\begin{equation}\label{eq:gibbsmeasure}
\mathrm{d}\mathbb{P}_{n, \beta} = \frac 1{Z_{n,\beta}} \exp\big(- \beta \mathcal{H}_n  (\{X_i\}_{i=1}^n)\big)\, \mathrm{d}X_1 \dots \mathrm{d}X_n,
\end{equation}
where $Z_{n,\beta}$ is the normalizing constant and $\beta$ denotes the inverse temperature.  In analogy with \eqref{eq:empmeas}, we can define the empirical measure of a Coulomb gas by $\mu^n = \frac1n \sum_{i=1}^n \delta_{X_i}$. In this setting the convergence of the empirical measure exhibits a twofold behavior.  On the one hand, for small temperature $\beta \gg \frac1n$, the empirical measure $\mu^n$ weakly converges to the \emph{equilibrium measure} $\mu^\infty_{\text{low}}$ defined as the minimizer
$$\mu^\infty_{\text{low}} := \operatornamewithlimits{argmin}_{\mu \in \mathcal{P}(\mathbb{T}^2)} \bigg\{ - \int_{\mathbb{T}^2\times \mathbb{T}^2} \log (x-y) \,\mathrm{d}\mu \otimes \mu (x,y) + \int_{\mathbb{T}^2} V \,\mathrm{d}\mu\bigg\}. $$
On the other hand, for large temperature $\beta\ll \frac{1}{n}$, a correction term is required and the empirical measure $\mu^n$ weakly converges to the \emph{thermal equilibrium measure} $\mu^\infty_{\text{high}}$, that is the minimizer
\begin{equation*}
\mu^\infty_{\text{high}} :=\operatornamewithlimits{argmin}_{\mu \in \mathcal{P}(\mathbb{T}^2)} \bigg\{- \int_{\mathbb{T}^2\times \mathbb{T}^2} \log (x-y) \,\mathrm{d}\mu \otimes \mu (x,y) + \int_{\mathbb{T}^2} V \,\mathrm{d}\mu + \frac 1{n \beta} \int_{\mathbb{T}^2} \mu \log \mu\bigg\}.
\end{equation*}
We refer the reader to \cite{serfaty} for a complete exposition on $2$D-Coulomb gases. This setting can be seen as an extension of the previous Ginibre ensemble on the torus. Indeed the law of the spectrum of the Ginibre ensemble is given by the Gibbs measure \eqref{eq:gibbsmeasure} choosing $\beta=2$ and $V(x)=|x|^2$ in \eqref{eq:hamiltonian}. Motivated by this observation and the works \cite{maxime, Jonas} on the Ginibre ensemble, we expect that the linearization approach could be employed also in this setting. First, we could justify the linearization Ansatz in the spirit of Theorem \ref{MainResultMatchingClouds} in both temperature regimes, using available results in the literature. Indeed, concentration inequalities around the equilibrium measure have been derived in \cite{Zelada} (see also \cite{Chafai}), that would replace the Bernstein's type inequality of Proposition \ref{BersteinCorrelated} and the matching cost estimates in Proposition \ref{prop:matchcost}. Second, a more ambitious question would be to use the linearization argument to derive optimal rates of the convergence to the equilibrium measure in both temperature regimes.

\end{enumerate}

\section{Structure of the proof}\label{StrategyOfProofSection}
This section is devoted to describe the main ideas and how are organized the proofs of Theorem \ref{MainResultMatchingClouds} and \ref{th1}. We mainly focus on the proof of Theorem \ref{MainResultMatchingClouds} since the proof of Theorem \ref{th1} follows by the same strategy.

\medskip

\subsubsection*{General strategy.}The proof of Theorem \ref{MainResultMatchingClouds} follows the strategy employed in \cite{ambrosio2019optimal} to deal with independent and uniformly distributed random points. 
The main idea is to use the quantitative stability result for transport maps in \cite[Theorem 3.2]{ambrosio2019optimal}, stating that two transport maps are close in the $\LL^2$-topology if the target measures are close in the $W_2$-topology. We restate below the result for reader's convenience.
\begin{theorem}[Stability of transport maps]\label{StabilityResultMap}
Let $\nu,\mu_1,\mu_2\in \mathcal{P}(\mathcal{M})$ such that $\nu\ll \m$ and let $T,S: \mathcal{M}\rightarrow\mathcal{M}$ be the optimal transport maps respectively for the pairs of measures $(\nu,\mu_1)$ and $(\nu,\mu_2)$. We assume that $S=\exp(\nabla f)$ for some $f:\mathcal{M}\rightarrow\mathbb{R}$ with $\mathrm{C}^{1,1}$-regularity. 

\medskip

\noindent There exists a constant $c>0$ depending on $\mathcal{M}$ such that, provided 
\begin{equation}\label{RegularityCondition}
\|\nabla f\|_{\LL^\infty}+\|\nabla^2 f\|_{\LL^\infty}\leq c,
\end{equation}
we have 
$$\int_{\mathcal{M}}\dd^2(S,T)\,\dd\nu\lesssim W^2_2(\mu_1,\mu_2)+W_2(\mu_1,\mu_2)W_2(\nu,\mu_1).$$
\end{theorem}
The first step consists of using Theorem \ref{StabilityResultMap} to deduce a stability estimate (in terms of the quadratic Wasserstein distance) of transport plans in the special case where $\mu_1=\nu^m$, $\mu_2=\exp(\nabla h^{n,t})_\#\mu^n$ and $\nu=\mu^n$. In this step, we immediately face the issue of the lack of regularity of $h^{n,t}$ necessary to ensure the condition \eqref{RegularityCondition}: Indeed, recalling that $h^{n,t}$ solves \eqref{eq:f^nmt}, it does not have the $\mathrm{C}^{1,1}$-regularity condition for non-smooth densities $\rho$ that we consider here. We overcome this issue introducing an additional regularization step: We smooth  the operator $-\nabla \cdot \rho \nabla$ and implicitly $\gamma^{n,t}$ in form of 
\begin{equation}\label{RegularizationIntro}
\gamma^{n,t}_{\delta}:=\big(\text{Id},\exp(\nabla h^{n,t}_\delta)\big)_\#\mu^{n,t}\quad \text{with }\nabla\cdot\rho_\delta\nabla h^{n,t}_\delta=\mu^{n,t}-\nu^{m,t},
\end{equation}
for a regularization parameter $\delta$ to be optimized and $\rho_\delta:=\text{P}_\delta \rho$. Classical Schauder's theory ensures that $h^{n,t}_\delta$ owns $\cc^\infty$-regularity. 
Doing so, we can use Theorem \ref{StabilityResultMap} to deduce, provided that 
\begin{equation}\label{SmallnessH}
\|\nabla h^{n,t}_\delta\|_{\LL^{\infty}}+\|\nabla^2 h^{n,t}_\delta\|_{\LL^\infty}\ll 1,
\end{equation}
a stability error estimate which reads
\begin{equation}\label{QuantitativeEstiPlanIntro}
\begin{aligned}
\inf_\pi W^2_2(\pi,\gamma^{n,t}_\delta)\lesssim&\, W^2_2\big(\nu^{m,t},\exp(\nabla h^{n,t}_\delta)_\#\mu^{n,t}\big)+W_2\big(\nu^{m,t},\exp(\nabla h^{n,t}_\delta)_\#\mu^{n,t}\big)W_2(\mu^{n},\nu^{m})\\
&+W^2_2(\nu^{m,t},\nu^{m})+W^2_2(\mu^{n,t},\mu^{n})+\big(W_2(\nu^{m,t},\nu^m)+W_2(\mu^{n,t},\mu^n)\big)W_2(\mu^n,\nu^m),
\end{aligned}
\end{equation}
where we refer to \eqref{QuantitativeEstiPlan} for further details. Our argument then differs from \cite{ambrosio2019optimal} by splitting the error in two parts
\begin{equation}\label{SplittingIntro}
\inf_\pi W^2_2(\pi,\gamma^{n,t})\leq \underbrace{W^2_2(\gamma^{n,t},\gamma^{n,t}_\delta)}_{\text{regularization error}}+\underbrace{\inf_{\pi} W^2_2(\pi,\gamma^{n,t}_\delta)}_{\text{stability error}}.
\end{equation}
Our proof then proceeds in two steps, controlling separately the two terms in \eqref{SplittingIntro}.
\subsubsection*{Control of the regularization error}To deal with the regularization error, we look at the difference $e^{n,t}_\delta:=h^{n,t}_\delta-h^{n,t}$ which solves, according to \eqref{eq:f^nmt} and \eqref{RegularizationIntro},
\begin{equation}\label{ErrorEquation}
-\nabla\cdot\rho_\delta\nabla e^{n,t}_\delta=\nabla\cdot(\rho_\delta-\rho)\nabla h^{n,t}.
\end{equation}
Using an energy estimate, we get
\begin{equation}\label{ErrorIntroRegul}
\int_{\mathcal{M}}\vert\nabla e^{n,t}_\delta\vert^2\lesssim \int_{\mathcal{M}}\vert\rho-\rho_\delta\vert^2\vert\nabla h^{n,t}\vert^2.
\end{equation}
On the one hand, since $\rho\in \LL^\infty$, we have $\rho_\delta\rightarrow\rho$ as $\delta\downarrow 0$ in every $\LL^q$ with $q<\infty$. On the other hand, we learn from Meyers' estimate (recalled in Proposition \ref{Meyers}) that there exists $\bar q>2$ such that $\nabla h^{n,t}\in \LL^{\bar q}$. Consequently, we can treat \eqref{ErrorIntroRegul} using H\"older's inequality which provides
\begin{equation}\label{LqControlIntro}
\int_{\mathcal{M}}\vert\nabla e^{n,t}_\delta\vert^2\lesssim\|\rho_\delta-\rho\|^2_{\LL^{2(\frac{\bar{q}}{2})'}}\|\nabla h^{n,t}\|^2_{\LL^{\bar{q}}}.
\end{equation}
The next step is to control the averaged Meyers' norm $\mathbb{E}\big[\|\nabla h^{n,t}\|^2_{\LL^{\bar{q}}}\big]$, that we show in Proposition \ref{prop1} to be of order of 
\begin{equation}\label{LqEstimatesIntro}
\mathbb{E}\big[\|\nabla h^{n,t}\|^2_{\LL^{\bar{q}}}\big]\lesssim \frac{\vert\log(t)\vert+\log^{\frac{1}{\eta}}(n)}{n},
\end{equation} 
where we recall that $\eta$ denotes the correlation length, see Assumption \ref{Assumptions}. 

\medskip

The combination of \eqref{LqControlIntro}, \eqref{LqEstimatesIntro} and local Lipschitzianity of the exponential map yields
\begin{equation}\label{regularizedErrorIntro}
\mathbb{E}\big[W^2_2(\gamma^{n,t}_\delta,\gamma^{n,t})\big]\lesssim \|\rho_\delta-\rho\|^2_{\LL^{2(\frac{\bar{q}}{2})'}}\frac{\vert \log(t)\vert+\log^{\frac{1}{\eta}}(n)}{n},
\end{equation}
 where we refer to \eqref{eq:thm1explipW2} for further details. We emphasize here that the stretched exponential decay assumption \eqref{DecayAlphaMixing} of the $\alpha$-mixing coefficient plays a crucial role in the estimate \eqref{LqEstimatesIntro}. Indeed, the additional contribution on the numerator of the r.h.s. of \eqref{LqEstimatesIntro} is due to the correlations and is only of logarithmic type $\log^{\frac{1}{\eta}}(n)$ thanks to the stretched exponential decay \eqref{DecayAlphaMixing}. Finally, the latter appears in the estimate \eqref{regularizedErrorIntro} and can be compensated with the choice of $\delta$ in \eqref{ChoicesIntro} and the regularity assumption on $\rho$ in form of \eqref{ApproximationRhoIntro}.
\subsubsection*{Control of the stability error}For the stability error, we first need to ensure \eqref{SmallnessH}. Our strategy follows the idea in \cite{ambrosio2019finer} which consists of showing that \eqref{SmallnessH} is satisfied with very high probability. In our case, a new difficulty comes from our regularization of $\rho$ and the regularization parameter $\delta$ has to be carefully optimized. We show that if $\delta$ is taken as an inverse power of $\log(n)$, \eqref{SmallnessH} becomes very likely as $n\uparrow\infty$. More precisely, we show in Proposition \ref{Fluctuation} that given two exponents $\kappa_1$ and $\upsilon\gg \kappa_1$, there exists $\kappa_2$ such that given the choices
\begin{equation}\label{ChoicesIntro}
\delta=\frac{1}{\log^{\kappa_1}(n)}\quad \text{and}\quad t=\frac{\log^{\kappa_2}(n)}{n},
\end{equation}
we have
\begin{equation}\label{FluctuationEstimates}
\mathbb{P}\Big(\|\nabla h^{n,t}_\delta\|_{\LL^{\infty}}+\|\nabla^2 h^{n,t}_\delta\|_{\LL^\infty}>\tfrac{1}{\log^{\upsilon}(n)}\Big)=o\Big(\frac{1}{n^\ell}\Big)\quad\text{for any $\ell\in \mathbb{N}$}.
\end{equation}
This result is in the spirit of \cite[Theorem 3.3]{ambrosio2019finer} but, in our setting, the proof relies on Schauder's theory rather than an explicit formula for $h^{n,t}_\delta$ as well as concentration inequalities in form of Proposition \ref{BersteinCorrelated} to treat the correlations. For further details on the strategy, we refer to Section \ref{FluctuationSection}. We emphasize here that, to obtain the super-polynomial behaviour \eqref{FluctuationEstimates}, we crucially use the fact that the concentration inequalities in Proposition \ref{BersteinCorrelated} are of sub-exponential type which is itself ensured by the sub-exponential decay assumption on the $\alpha$-mixing coefficient \eqref{DecayAlphaMixing}. The reason lies in the choices of $\delta$ and $t$ in \eqref{ChoicesIntro}. Indeed, the only room we have when optimizing $t$ is in the logarithmic growth $\log^{\kappa_2}(n)$ (as already mention in Section \ref{sec:lin}, the natural regularization time is $t\sim \frac{1}{n}$ up to logarithmic corrections). Furthermore, the quantity $\|(\nabla h^{n,t}_\delta,\nabla^2 h^{n,t}_\delta)\|_{\LL^{\infty}}$ can be heuristically estimated by an inverse power of $\delta$ and $t^{-1}$ as it involves powers of $\|(\nabla\rho_\delta,\nabla^2\rho_\delta)\|_{\LL^\infty}$ and the norm $\|\mu^{n,t}-1\|_{\LL^{\infty}}$ by Schauder's theory. Hence, in the best case scenario where exponential concentration holds, we expect
\begin{align*}
\mathbb{P}\Big(\|\nabla h^{n,t}_\delta\|_{\LL^{\infty}}+\|\nabla^2 h^{n,t}_\delta\|_{\LL^\infty}>\lambda\Big)&\lesssim \exp\big(-\lambda n\|(\nabla h^{n,t}_\delta,\nabla^2 h^{n,t}_\delta)\|^{-1}_{\LL^{\infty}}\big)\\
&\leq \exp\big(-\lambda nt\delta^{\tilde{\upsilon}}\big)=\exp\big(-\lambda\log^{\kappa_2}(n)\delta^{\tilde{\upsilon}}\big),
\end{align*}
for some $\tilde{\upsilon}>0$, which gives a super-polynomial behavior for the choices $\delta=\frac{1}{\log^{\kappa_1}(n)}$ and $\lambda=\frac{1}{\log^{\upsilon}(n)}$ for large $\kappa_2$. Weaker properties, as for instance polynomial concentration, would only lead to a decay given by an inverse power of $\log(n)$ which is not enough for our purpose.
\medskip

With \eqref{FluctuationEstimates} in hands, we can restrict the analysis to realizations satisfying $\|\nabla h^{n,t}_\delta\|_{\LL^{\infty}}+\|\nabla^2 h^{n,t}_\delta\|_{\LL^\infty}\leq\tfrac{1}{\log^{\upsilon}(n)}$ where, for $n\gg 1$, \eqref{SmallnessH} is satisfied which puts us in the validity of \eqref{QuantitativeEstiPlanIntro}. We then treat each terms appearing in \eqref{QuantitativeEstiPlanIntro} separately:
\begin{itemize}
\item \textbf{The optimal control of the cost} $W^2_2(\mu^n,\nu^m)$ has been already well studied and is optimally estimated by 
\begin{equation}\label{CostIntro}
\mathbb{E}\big[W^2_2(\mu^n,\nu^m)\big]\lesssim \frac{\log(n)}{n}.
\end{equation}
We refer to Appendix \ref{sec:matchcost} for a detailed statement, references and extensions to the cases of Assumption \ref{Assumptions} and Assumption \ref{ClassMarkov}.

\medskip

\item \textbf{The smoothing errors} $W^2_2(\mu^n,\mu^{n,t})$ and $W^2_2(\nu^m,\nu^{m,t})$. Classical contractivity estimates are known and are usually applied to deal with these errors, see for instance  \cite[Theorem 3]{erbar2015equivalence}, which bound the errors by $t$. However, due to the choice of $t$ in \eqref{ChoicesIntro}, this result is of no use since $t$ is much larger than the magnitude of the cost, namely $t\gg \frac{\log(n)}{n}$. Instead, we follow the approach in \cite{ambrosio2019finer}, where the authors showed that in the particular case of empirical measures in dimension $2$, we can improve the rate and obtain the bound $\frac{\log\log(n)}{n}\ll \frac{\log(n)}{n}$. We extend this result to our setting of non-constant densities and correlated points. In Proposition \ref{Contractivity}, we derive
\begin{equation}\label{ContractivityEstimatesIntro}
\mathbb{E}\big[W^2_2(\mu^{n},\mu^{n,t})\big]+\mathbb{E}\big[W^2_2(\nu^{m},\nu^{m,t})\big]\lesssim \frac{\log\log(n)}{n}+t\|\rho_{t+\frac{1}{n}}-\rho_{\frac{1}{n}}\|_{\LL^1}.
\end{equation}
As opposed to \cite{ambrosio2019finer}, our approach uses Fourier analysis together with additional cares to handle the correlations and non-constant densities.

\medskip

\item \textbf{The error in the Moser coupling} $W^2_2\big(\nu^{m,t},\exp(\nabla h^{n,t}_\delta)_\#\mu^{n,t}\big)$. We follow the strategy in \cite{ambrosio2019finer}. This error can be related to a Moser coupling between $\mu^{n,t}$ and $\nu^{m,t}$ (see for instance \cite[Appendix p. 16]{OldNewVillani}): The equation \eqref{eq:f^nmt} gives a natural coupling between $\mu^{n,t}$ and $\nu^{m,t}$ which can be formulated using Benamou-Brenier's theorem \cite{benamou2000computational},
$$\nu^{m,t}=\phi(1,\cdot)_\#\mu^{n,t}\quad\text{with $\phi$ being the flow induced by $s\mapsto \frac{\rho_\delta\nabla h^{n,t}_\delta}{(1-s)\mu^{n,t}+s\nu^{m,t}}$}.$$
Then, using the transport plan $\big(\phi(1,\cdot),\exp(\nabla h^{n,t}_\delta)\big)_\#\mu^{n,t}$ as a competitor, we get
\begin{align*}
W^2_2\big(\nu^{m,t},\exp(\nabla h^{n,t}_\delta)_\#\mu^{n,t}\big)&=W^2_2\big(\phi(1,\cdot)_\#\mu^{n,t},\exp(\nabla h^{n,t}_\delta)_\#\mu^{n,t}\big)\\
&\leq \int_{\mathcal{M}}\dd^2\big(\phi(1,\cdot),\exp(\nabla h^{n,t}_\delta)\big),
\end{align*}
that we combine with a quantitative stability result for flows of vector fields, proved in \cite[Proposition A.1]{ambrosio2019finer}, leading to
$$W^2_2\big(\nu^{m,t},\exp(\nabla h^{n,t}_\delta)_\#\mu^{n,t}\big)\lesssim\big(\|\rho_{\delta}-\rho\|^{2}_{\LL^{2(\frac{\bar q}{2})'}}+\|\rho_{t}-\rho\|^{2}_{\LL^{2(\frac{\bar q}{2})'}}+\tfrac{1}{\log^{\upsilon}(n)}\big)\|\nabla h^{n,t}_\delta\|^2_{\LL^{\bar{q}}},$$
%
%
where we recall that $\bar{q}$ denotes the Meyers' exponent introduced in \eqref{ErrorIntroRegul}. For further details, we refer to \eqref{StabilityFlowProof}. It then remains to appeal to Meyers' estimate, see Proposition \ref{Meyers}, to \eqref{ErrorEquation} together with \eqref{LqEstimatesIntro} in form of 
$$\mathbb{E}\big[\|\nabla h^{n,t}_\delta\|^2_{\LL^{\bar{q}}}\big]\lesssim  \mathbb{E}\big[\|\nabla h^{n,t}\|^2_{\LL^{\bar{q}}}\big]\stackrel{\eqref{LqEstimatesIntro}}{\lesssim}\frac{\vert\log(t)\vert+\log^{\frac{1}{\eta}}(n)}{n},$$
which finally yields
\begin{equation}\label{EstimateMoserCoupling}
\mathbb{E}\big[W^2_2\big(\nu^{m,t},\exp(\nabla h^{n,t}_\delta)_\#\mu^{n,t}\big)\big]\lesssim \big(\|\rho_{\delta}-\rho\|^{2}_{\LL^{2(\frac{\bar q}{2})'}}+\|\rho_{t}-\rho\|^{2}_{\LL^{2(\frac{\bar q}{2})'}}+\tfrac{1}{\log^{\upsilon}(n)}\big)\frac{\vert\log(t)\vert+\log^{\frac{1}{\eta}}(n)}{n}.
\end{equation}
\end{itemize}
\medskip

To conclude, we see that all bounds involve errors in terms of the approximation of $\rho$ using the heat-semigroup, that we need to quantify. This is where the assumption $\rho\in \mathrm{H}^\varepsilon$ plays a role, in form of the quantitative estimate 
\begin{equation}\label{ApproximationRhoIntro}
\|\rho_s-\rho\|_{\LL^2}\lesssim \|\rho\|_{\mathrm{H}^\varepsilon} \,s^{\varepsilon}\quad\text{for any $s>0$},
\end{equation}
see \eqref{ApproximationRho} for a proof. Combining \eqref{QuantitativeEstiPlanIntro}, \eqref{SplittingIntro}, \eqref{regularizedErrorIntro}, \eqref{CostIntro}, \eqref{ContractivityEstimatesIntro}, \eqref{EstimateMoserCoupling}, \eqref{ApproximationRhoIntro} with the choices of $\delta$ and $t$ in \eqref{ChoicesIntro}, we obtain Theorem \ref{MainResultMatchingClouds}. The proof of Theorem \ref{th1} is obtained using the same strategy where the first step is simpler, since we apply directly Theorem \ref{StabilityResultMap} with $\mu_1=\rho$, $\nu=\mu^n$ and $\mu_2=\exp(\nabla f^{n,t}_\delta)_\#\mu^n$ where $f^{n,t}$ solves $-\nabla\cdot \rho_\delta\nabla f^{n,t}=\mu^{n,t}-\rho_t$.

\section{Proofs}
\subsection{Notations and preliminary results}\label{sec:preliminary}
We provide in this section some notations and preliminary results needed in the proofs of Theorem \ref{MainResultMatchingClouds} and Theorem \ref{th1}. We recall that throughout the paper, we denote by $\mathcal{M}$ a $2$-dimensional compact connected Riemannian manifold (or the square $[0,1]^2$) endowed with the Riemannian distance $\dd$.

\medskip
\noindent
\textbf{Wasserstein distance. }Given $\mu, \nu\in\mathcal{P}(\mathcal{M})$, we define the quadratic Wasserstein distance as
\begin{equation}\label{Wasserstein}
W_2^2(\mu, \nu):= \min_{\pi \in \Pi(\mu, \nu)} \int_{\mathcal{M}\times \mathcal{M}} \dd^2(x,y) \,\dd \pi (x,y),
\end{equation}
where $\Pi(\mu, \nu)$ is the set of couplings between $\mu$ and $\nu$, that is the set of $\pi\in\mathcal{P}(\mathcal{M}\times \mathcal{M})$ having $\mu$ and $\nu$ as first and second marginal, respectively. We refer the reader to the monographs \cite{Viltop, OldNewVillani} for a detailed overview on the subject of optimal transport. 
%
%
%
%
%
%
%
%
%
We recall the following simple, but useful Lipschitz contraction property of the Wasserstein distance. 
\begin{lemma}[Lipschitz property of the Wasserstein metric] \label{lem:transportcontraction}
Let $(D, \dd_D)$ be a complete and separable metric space, let $\mu, \nu\in\mathcal{P}(\mathcal{M})$ and let $T: \mathcal{M} \rightarrow D$ be a $L$-Lipschitz map. It holds
\begin{equation}\label{Eq:LipEstimateWasser}
W_2^2 (T_{\#}\mu,T_{\#} \nu) \le L^2 W_2^2 (\mu, \nu).
\end{equation}
\end{lemma}
\begin{proof}
For any coupling $\pi \in \Pi(\mu, \nu)$, the push-forward $(T,T)_\# \pi$ is a coupling between $T_\# \mu$ and $T_\# \nu$. Moreover, it holds
\[
\begin{split}
\int_{D\times D} \dd^2(x,y) \,\dd ((T,T)_\#\pi)(x,y) & = \int_{\mathcal{M}\times\mathcal{M}} \dd^2(T(x),T(y))\, \dd \pi(x,y)\\
&\le L^2 \int_{\mathcal{M}\times\mathcal{M}} \dd^2(x,y) \, \dd \pi(x,y).
\end{split}
\]
Taking the infimum among all possible couplings $\pi \in \Pi(\mu,\nu)$ leads to \eqref{Eq:LipEstimateWasser}.
\end{proof}
%
%
%
%
%
%
%
%
%

\medskip
\noindent
\textbf{Heat semigroup and heat kernel. }We recall some basic facts on the heat semigroup and its generator, we refer the reader to \cite[Chapter VI]{chavel1984eigenvalues} for a more detailed overview. For $t>0$, we denote by $p_t$ the fundamental solution of the heat operator $\partial_t-\Delta$ on $\mathcal{M}$, where $\Delta$ denotes the Beltrami-Laplace operator. Classical Schauder's theory ensures that $p_t$ is smooth and it is known, see for instance \cite{stroock1998upper} and \cite[Appendix B]{ambrosio2019finer}, that $p_t$ and its derivatives satisfy for some $C>0$ depending on $\mathcal{M}$
\begin{equation}\label{Eq28}
\vert\nabla^N p_{t}(x,y)\vert\lesssim t^{-1-\frac{N}{2}}\exp\Big(-\tfrac{1}{C}\tfrac{\dd^2(x,y)}{t}\Big)\quad\text{for any $t>0$, $N\geq 1$ and $x,y\in\mathcal{M}$.}
\end{equation}
%
The kernel $p_t$ admits the spectral decomposition
\begin{equation}\label{SpectralDecompoHeatKernel}
p_{t}(x,y):=\sum_{k\geq 1}e^{-t\lambda_k}\phi_k(x)\phi_k(y)\quad \text{for any $t>0$ and $x,y\in\mathcal{M}$},
\end{equation}
converging in $\LL^2(\mathcal{M}\times\mathcal{M})$, where we recall that $\{\lambda_k,\phi_k\}$ denotes the eigenvalues and eigenvectors of $-\Delta$ on $\mathcal{M}$. Specifying \eqref{SpectralDecompoHeatKernel} on the diagonal and using $\|\phi_k\|_{\mathrm{L}^2(\mathcal{M})}=1$, we obtain the trace formula
\begin{equation}\label{eq:traceformula}
\sum_{k\geq 1}e^{-t\lambda_k}=\int_{\mathcal{M}}p_t(x,x)\dd\m(x)\quad\text{for any $t>0$.}
\end{equation}
We recall that $\{\phi_k\}_{k\geq 1}$ can be estimated in terms of the eigenvalues,
\begin{equation}\label{eq:eigvaleighfunc}
 \|\phi_k\|_{\LL^{\infty}}\lesssim \lambda^{\frac{1}{2}}_k.
\end{equation}
We briefly recall the argument. Applying the Gagliardo-Nirenberg's interpolation inequality \cite[Theorem 3.70]{aubin1998some}, it holds
$$\|\phi_k\|_{\LL^{\infty}}\lesssim \|\nabla^2 \phi_k\|^{\frac{1}{2}}_{\LL^{2}}\|\phi_k\|^{\frac{1}{2}}_{\LL^2}=\|\nabla^2\phi_k\|^{\frac{1}{2}}_{\LL^2}.$$
In combination with $-\Delta\phi_k=\lambda_k\phi_k$ and elliptic regularity, in form of
$$\|\nabla^2\phi_k\|_{\LL^2}\lesssim \lambda_k\|\phi_k\|_{\LL^2}=\lambda_k,$$
we obtain \eqref{eq:eigvaleighfunc}. 

\medskip

We recall that $(\text{P}_t)_{t>0}$ admits the spectral gap property, that is there exists a constant $C_{\text{sg}}>0$ such that 
\begin{equation}\label{SpectralGap}
\|\pp_t f\|_{\mathrm{L}^2(\mathcal{M})} \le e^{-C_{\text{sg}} t}\|f\|_{\mathrm{L}^2(\mathcal{M})}\quad\text{for any $f \in \mathrm{L}^2$ with $\int_{\mathcal{M}} f\, \dd \m =0$.}
\end{equation}
Note that equivalently, \eqref{SpectralGap} can be formulated in terms of the eigenvalues $\{\lambda_k\}_{k\geq 1}$ in form of 
\begin{equation}\label{SpectralGapEigenvalues}
\inf_{k\geq 1}\lambda_k\geq C_{\text{sg}},
\end{equation}
simply by specifying \eqref{SpectralGap} on $\{\phi_k\}_{k\geq 1}$. Finally, we recall the Riesz-transform bound
\begin{equation}\label{eq:Riesz}
\int_{\mathcal{M}} |\nabla f|^4\,\dd\m  \lesssim \int_{\mathcal{M}}\big|(-\Delta)^\frac12 f\big|^4\,\dd\m\quad\text{for any $f$ such that $\int_{\mathcal{M}}\vert\nabla f\vert^4\,\dd\m<\infty$},
\end{equation}
where $(-\Delta)^{\frac{1}{2}}$ can be defined via its spectral representation
\begin{equation}\label{Laplacian12}
(-\Delta)^{\frac{1}{2}}f=\sum_{k\geq 1} \sqrt{\lambda_k}\,\big(\phi_k,f\big)_{\LL^2}\phi_k\quad\text{for any $f\in \mathrm{H}^{\frac{1}{2}}$,}
\end{equation}
with $(\cdot,\cdot)_{\LL^2}$ the inner product in $\LL^2$. We refer to the monograph \cite{wangbook} for a discussion of the inequalities \eqref{SpectralGap} and \eqref{eq:Riesz}, see Chapter 1 for the case of a Riemannian manifold without boundary and Chapter 2 for the case of a Riemannian manifold with (convex) boundary. In connection with the Wasserstein metric, the heat semigroup satisfies the following classical contraction property.
\begin{lemma}[Semigroup contraction for absolutely continuous measures]\label{lem:HeatContractionrho}
Let $\rho \in\LL^\infty$ satisfying \eqref{Ellipticity}. Given $\rho_t:=\mathrm{P}_t\rho$, it holds
\begin{equation}\label{Eq80}
W_2^2 (\rho_t\,\dd\m, \rho\,\dd\m) \lesssim t \| \rho_t - \rho\|_{\LL^1}\quad\text{for any $t>0$.}
\end{equation}
\end{lemma}
\begin{proof}
Using $g$ defined via $-\Delta g=\rho_t-\rho$ together with \eqref{Ellipticity}, \cite[Corollary 4.4]{ambrosio2019finer} yields
$$W^2_2(\rho_t\,\dd\m,\rho\,\dd\m)\lesssim \int_{\mathcal{M}}\vert\nabla g\vert^2\,\dd\m=\int_{\mathcal{M}}(\rho_t-\rho)g\,\dd\m.$$
Writing $g=\int_{0}^{\infty}\text{P}_\tau(\rho_t-\rho)\,\dd\tau=-\int_{0}^t \rho_\tau\,\dd\tau$ together with $\vert\int_{0}^t \rho_\tau\,\dd\tau\vert\leq \|\rho\|_{\LL^{\infty}} t$ gives \eqref{Eq80}.
\end{proof}

\subsection{\texorpdfstring{$\LL^q$}{Lq}-type estimates}
As we have seen in \eqref{LqControlIntro}, we need a sharp control of the averaged Meyers' norm $\mathbb{E}\big[\|\nabla h^{n,t}\|^2_{\LL^{\bar q}}\big]$. This will be obtained as an immediate corollary of the following proposition, for more details see \eqref{LqEstihnt}.
\begin{proposition}[$\LL^{q}$-estimates]\label{prop1}
Let $\{\mu^n\}_n$ be defined in \eqref{eq:empmeas} with point clouds satisfying Assumption \ref{Assumptions}. Let $\bar q$ be the Meyers exponent given in Theorem \ref{Meyers} for the operator $-\nabla\cdot\rho\nabla$. The solution\footnote{with Neumann boundary conditions in case $\mathcal{M}$ has a boundary} $f^{n,t}\in {\rm \dot{H}^1}$ of 
\begin{equation}\label{eq:f^nt}
-\nabla\cdot \rho\nabla f^{n,t} =\mu^{n,t}-\rho_t,
\end{equation}
satisfies:
\begin{equation}\label{Eq22Bis}
\Big(\int_\mathcal{M} \vert\nabla f^{n,t}\vert^{q}\,\dd\m\Big)^\frac{2}{q}\leq \mathcal{C}_{n,t}\frac{\vert\log(t)\vert+\log^{\frac{1}{\eta}}(n)}{n}\quad \text{for any $q\in [2,\min\{\bar q,4\}]$},
\end{equation}
and a random variable $\mathcal{C}_{n,t}$ satisfying for some generic constant $C_q$ depending on $q$,
\begin{equation}\label{MomentBoundCn}
\sup_{n,t}\mathbb{E}[\tfrac{1}{C_{q}}\mathcal{C}_{n,t}]\leq 1.
\end{equation}
\medskip

Furthermore, if \eqref{DecayAlphaMixing} holds with $\eta\geq 1$ then the assumption \eqref{DecatBetaMixing} can be dropped and the stochastic integrability can be improved up to losing a $\log(n)$ factor, namely
\begin{equation}\label{Eq22}
\begin{split}
\lefteqn{\Big(\int_\mathcal{M} \vert\nabla f^{n,t}\vert^{q}\,\dd\m\Big)^\frac{2}{q}}\\ &\leq \mathcal{D}_{n,t}\Big(\frac{\log^{\frac
{1}{\eta}}(n)\vert\log(t)\vert}{n}+\frac{t^{-1}(1+\log^2(n)\mathds{1}_{\eta\neq \infty})}{n^2}\Big)\quad \text{for any $q\in [2,\min\{\bar q,4\}]$},
\end{split}
\end{equation}
and a random variable $\mathcal{D}_{n,t}$ satisfying for some generic constant $D_q$ depending on $q$,
$$\sup_{n,t}\mathbb{E}[\exp(\tfrac{1}{D_{q}}\mathcal{D}^{\frac{1}{2}}_{n,t})]\leq 2.$$
\begin{proof}
We proceed in four steps. In the first step, we prove a representation formula for $(-\Delta)^{\frac{1}{2}}$ that we will use as the core tool in the next steps. In the second step, we compare the two operators $-\nabla\cdot \rho\nabla$ and $-\Delta$, with help of Meyers' estimate recalled in Theorem \ref{Meyers}. Doing so, we then have to bound the $\LL^{q}$-norms ($2\leq q\leq\bar q$) of the gradient of the solution\footnote{which belongs to any $\LL^{q}$ for any $q<\infty$ from Calder\'on-Zygmund' theory, see for instance \cite{giaquinta2013introduction}} to the Poisson equation with r.h.s. $\mu^{n,t}-\rho_t$ and Neumann boundary conditions. We control all the norms by the $\LL^{4}$-norm that, in turn, we estimate using the Riesz-transform bound \eqref{eq:Riesz} and following some ideas from \cite[Lemma 3.17]{ambrosio2019pde}. In the third and fourth steps, we control the bound previously obtained in expectation where our main tool is Assumption \ref{Assumptions} and the concentration inequalities in Proposition \ref{BersteinCorrelated}.

\medskip

{\sc Step 1. A representation formula for $(-\Delta)^{\frac{1}{2}}$. }
We show that given $f\in \cc^2$ such that $n_{\mathcal{M}}\cdot \nabla f=0$ on $\partial \mathcal{M}$ we have
\begin{equation}\label{Eq45}
(-\Delta)^{\frac{1}{2}} f=\frac{1}{\sqrt{\pi}}\int_{0}^\infty\tau^{-\frac{1}{2}}\Delta \text{P}_{\tau}f\,\dd\tau.
\end{equation}
%
%
%
%
%
Note that $(-\Delta)^{\frac{1}{2}}f$, defined in \eqref{Laplacian12}, is well defined in $\LL^2$. Indeed, using the fact that $(\phi_k,f)_{\LL^2}=\frac{1}{\lambda_k}(\phi_k,\Delta f)_{\LL^2}$ and \eqref{SpectralGapEigenvalues}, we have for any $N\leq M<\infty$
\begin{align*}
\bigg\|\sum_{N\leq n\leq M}\sqrt{\lambda_n}\,\big(\phi_n,f)_{\LL^2}\phi_n\bigg\|^2_{\LL^2}=\sum_{N\leq n\leq M}\lambda_n\vert (\phi_n,f)_{\LL^2}\vert^2=&\sum_{N\leq n\leq M}\tfrac{1}{\lambda_n}\vert (\phi_n,\Delta f)_{\LL^2}\vert^2\\
\leq & \tfrac{1}{C_{\text{sg}}}\sum_{n\geq N}\vert (\phi_n,\Delta f)_{\LL^2}\vert^2,
\end{align*}
which vanishes as $N\uparrow\infty$ uniformly in $M$.

\medskip

We now justify \eqref{Eq45}. Observe that since $n_\mathcal{M}\cdot \nabla f=0$ on $\partial\mathcal{M}$,
\begin{equation}\label{Eq44}
\Delta \text{P}_s f=\text{P}_s\Delta f\quad \text{for any $s\in (0,\infty)$,}
\end{equation}
which is a direct consequence of two integration by parts using the heat-kernel representation $\text{P}_s f=\int_{\mathcal{M}}p_s(\cdot,y)f(y)\,\dd \m(y)$. Therefore 
\begin{equation}\label{Laplacian12:Eq1}
\int_{0}^{\infty}\tau^{-\frac{1}{2}}\Delta {\rm P}_{\tau} f\,\dd\tau=\int_{0}^{\infty}\tau^{-\frac{1}{2}}{\rm P}_{\tau} \Delta f\,\dd\tau,
\end{equation}
where the last integral is well-defined in $\LL^2$ since from \eqref{SpectralGap} 
$$\int_0^\infty t^{-\frac{1}{2}}\|\pp_\tau \Delta f\|_{\LL^2}\,\dd\tau\leq\int_0^\infty t^{-\frac{1}{2}}e^{-C_{sg}\tau}\|\Delta f\|_{\LL^2}\,\dd\tau<\infty.$$
We then use the spectral decomposition of the heat semigroup \eqref{SpectralDecompoHeatKernel} to get
\begin{equation}\label{Laplacian12:Eq2}
\pp_\tau \Delta f=\sum_{n} e^{-\lambda_n\tau}\big(\phi_n,\Delta f\big)_{\LL^2}\phi_n\quad\text{in $\LL^2$.}
\end{equation}
%
%
%
The combination of \eqref{Laplacian12:Eq1} and \eqref{Laplacian12:Eq2} yields for any $\eta\in \LL^2$
\begin{equation}\label{Laplacian12:Eq3}
\begin{aligned}
\bigg(\int_{0}^{\infty}\tau^{-\frac{1}{2}}\Delta {\rm P}_{\tau} f,\eta\bigg)_{\LL^{2}}\,\dd\tau=&\int_{0}^\infty \tau^{-\frac{1}{2}}\big(\pp_\tau\Delta f,\eta\big)_{\LL^2}\,\dd\tau\\
\stackrel{\eqref{Laplacian12:Eq2}}{=}&\int_{0}^{\infty}\tau^{-\frac{1}{2}}\sum_{n} e^{-\lambda_n\tau}\big(\phi_n,\Delta f\big)_{\LL^2}\big(\phi_n,\eta\big)_{\LL^{2}}\,\dd\tau.
\end{aligned}
\end{equation}
Using \eqref{SpectralGapEigenvalues}, we have
\begin{align*}
&\int_{0}^\infty \tau^{-\frac{1}{2}}\sum_n e^{-\lambda_n\tau}\vert\big(\phi_n,\Delta f\big)_{\LL^2}\vert\vert \big(\phi_n,\eta\big)_{\LL^2}\vert\,\dd\tau\\
&\leq \Big(\int_{0}^\infty  \tau^{-\frac{1}{2}}e^{-C_{sg}\tau}\,\dd\tau\Big)\sum_{n}\vert\big(\phi_n,\Delta f\big)_{\LL^2}\vert\vert \big(\phi_n,\eta\big)_{\LL^2}\vert<\infty,
\end{align*}
so that we can exchange integration and summation in \eqref{Laplacian12:Eq3} to obtain
\begin{align*}
\bigg(\int_{0}^{\infty}\tau^{-\frac{1}{2}}\Delta {\rm P}_{\tau} f,\eta\bigg)_{\LL^{2}}\,\dd\tau&=\sum_n\Big(\int_{0}^{\infty}\tau^{-\frac{1}{2}}e^{-\lambda_n\tau}\,\dd\tau\Big)\big(\phi_n,\Delta f\big)_{\LL^2}\big(\phi_n,\eta\big)_{\LL^2}\\
&=\sqrt{\pi}\sum_n \tfrac{1}{\sqrt{\lambda_n}}\big(\phi_n,\Delta f\big)_{\LL^2}\big(\phi_n,\eta\big)_{\LL^2}\\
&=\sqrt{\pi}\sum_n \sqrt{\lambda_n}\,\big(\phi_n, f\big)_{\LL^2}\big(\phi_n,\eta\big)_{\LL^2}\stackrel{\eqref{Laplacian12}}{=}\sqrt{\pi}\big((-\Delta f)^{\frac{1}{2}},\eta\big)_{\LL^2},
\end{align*}
which gives \eqref{Eq45} by arbitrariness of $\eta$. Finally, note that the \rhs integral in \eqref{Eq45} is absolutely convergent thanks to the integration by parts $\Delta \pp_\tau f=\pp_{\tau}\Delta f$ and the bounds on the heat kernel \eqref{Eq28}, so that it defines a function in $\cc^0$.
\medskip

{\sc Step 2. Comparison with the solution of the Poisson equation. }We claim that for any $q\in [2,\min\{\bar q,4\}]$ and $p<\infty$
\begin{equation}\label{Eq51}
\begin{aligned}
\mathbb{E}\bigg[\Big(\int_{\mathcal{M}}\vert \nabla f^{n,t}\vert^q\,\dd\m\Big)^{\frac{2p}{q}}\bigg]^{\frac{1}{p}}\lesssim_q \bigg(\int_{\mathcal{M}}\dd\m\bigg(\int_{0}^{\infty}\dd s\,\mathbb{E}\Big[\Big((-s\Delta)^{\frac{1}{2}}\text{P}_{s+t}(\mu^n-\rho)\Big)^{2p}\Big]^{\frac{1}{p}}\bigg)^2\bigg)^{\frac{1}{2}}.
\end{aligned}
\end{equation}
Let $g^{n,t}\in {\rm \dot{H}^1}$ be the solution of the following Poisson equation 
\begin{equation}\label{Eq47}
        -\Delta g^{n,t} =\mu^{n,t}-\rho_t.
\end{equation}
Re-expressing the right-hand side of \eqref{eq:f^nt} as $\nabla\cdot \nabla g^{n,t}$, we apply Meyers' estimate recalled in Theorem \ref{Meyers} and Hölder's inequality to obtain: 
\begin{equation}\label{Eq41}
\Big(\int_\mathcal{M} \vert\nabla f^{n,t}\vert^{q}\,\dd\m\Big)^\frac{2}{q}\lesssim \Big(\int_\mathcal{M} \vert\nabla g^{n,t}\vert^{q}\,\dd\m\Big)^\frac{2}{q}\lesssim_q\Big(\int_\mathcal{M} \vert\nabla g^{n,t}\vert^{4}\,\dd\m\Big)^\frac{1}{2}\quad\text{for any $q\in [2,\min\{\bar q,4\}]$,}
\end{equation}
We now introduce the Paley-Littlewood functional
$$\mathcal{L}(g):=\Big(\int_{0}^{\infty}s(\partial_s\text{P}_s g)^2\,\dd s\Big)^{\frac{1}{2}}\quad \text{for any $g\in \LL^{4}$ and $\int_{\mathcal{M}} g=0$.}$$
We recall that the inverse of $\mathcal{L}$ is continuous, see \cite{stein2016topics}, namely
\begin{equation}\label{Eq40}
\|g\|_{\LL^{4}}\lesssim \|\mathcal{L}(g)\|_{\LL^{4}}.
\end{equation}
Combining the Riesz transform bound \eqref{eq:Riesz} with \eqref{Eq40} yields
\begin{equation}\label{Eq50}
\begin{aligned}
\int_{\mathcal{M}}\vert\nabla g^{n,t}\vert^4\,\dd\m\lesssim\int_{\mathcal{M}}\Big(\mathcal{L}\Big((-\Delta)^{\frac{1}{2}}g^{n,t}\Big)\Big)^4\dd\m
=\int_{\mathcal{M}}\dd\m\Big(\int_{0}^{\infty}\dd s\, s(\partial_s\text{P}_s(-\Delta)^{\frac{1}{2}} g^{n,t})^2\Big)^{2}.
\end{aligned}
\end{equation}
We now claim that
\begin{equation}\label{Eq43}
\partial_s\text{P}_s(-\Delta)^{\frac{1}{2}}g^{n,t}=(-\Delta)^{\frac{1}{2}}\text{P}_{s+t}(\mu^n-\rho)\quad\text{for any $s\geq 0$,}
\end{equation}
which requires a special care when $\mathcal{M}$ has a boundary. We use the definition of $\text{P}_s$ in form of $\partial_s\text{P}_s=\Delta\text{P}_s$  to get
\begin{equation}\label{Eq46}
\partial_s\text{P}_s(-\Delta)^{\frac{1}{2}}g^{n,t}=\Delta\text{P}_s(-\Delta)^{\frac{1}{2}}g^{n,t}.
\end{equation}
Recalling that $n_{\mathcal{M}}\cdot \nabla g^{n,t}=0$, \eqref{Eq44} implies that $\Delta\text{P}_{\tau}g^{n,t}=\text{P}_{\tau}\Delta g^{n,t}$ which, combined with \eqref{Eq47} and \eqref{Eq45}, gives
\begin{equation}\label{Eq48}
(-\Delta)^{\frac{1}{2}} g^{n,t}=\frac{1}{\sqrt{\pi}}\int_{0}^\infty \tau^{-\frac{1}{2}}\text{P}_{\tau}(\mu^{n,t}-\rho_t)\,\dd\tau.
\end{equation}
In particular, it implies that $n_{\mathcal{M}}\cdot \nabla(-\Delta)^{\frac{1}{2}}g^{n,t}=0$. Therefore, one can use once more \eqref{Eq44} and, together with \eqref{Eq48}, \eqref{Eq46} turns into
\begin{equation}\label{Laplace12:Eq5}
\partial_s\text{P}_s(-\Delta)^{\frac{1}{2}}g^{n,t}=\frac{1}{\sqrt{\pi}}\int_{0}^\infty\tau^{-\frac{1}{2}}\text{P}_s\Delta\text{P}_{\tau}(\mu^{n,t}-\rho_t)\,\dd\tau.
\end{equation}
Using a last time \eqref{Eq44} in form of $\text{P}_s\Delta\text{P}_{\tau}(\mu^{n,t}-\rho_t)=\Delta\text{P}_s\text{P}_{\tau}(\mu^{n,t}-\rho_t)$ that we combine with the semigroup property $\text{P}_{t}\text{P}_{t'}=\text{P}_{t+t'}$ yields
$$\text{P}_s\Delta\text{P}_{\tau}(\mu^{n,t}-\rho_t)=\Delta\text{P}_{\tau}\text{P}_{s+t}(\mu^{n}-\rho),$$
which, together with \eqref{Laplace12:Eq5} and \eqref{Eq45} leads to \eqref{Eq43}.

\medskip

The combination of \eqref{Eq41}, \eqref{Eq50} and \eqref{Eq43} together with Minkowski's inequality gives
\begin{equation*}
\begin{aligned}
\mathbb{E}\bigg[\Big(\int_{\mathcal{M}}\vert \nabla f^{n,t}\vert^q\,\dd\m\Big)^{\frac{2p}{q}}\bigg]^{\frac{1}{p}}\stackrel{\eqref{Eq41}}{\lesssim}&\mathbb{E}\bigg[\bigg(\int_{\mathcal{M}}\dd\m\bigg(\int_{0}^{\infty}\dd s\,\Big((-s\Delta)^{\frac{1}{2}}\text{P}_{s+t}(\mu^n-\rho)\Big)^2\bigg)^2\bigg)^{\frac{p}{2}}\bigg]^{\frac{1}{p}}\\
\leq& \bigg(\int_{\mathcal{M}}\dd\m\bigg(\int_{0}^{\infty}\dd s\,\mathbb{E}\Big[\Big((-s\Delta)^{\frac{1}{2}}\text{P}_{s+t}(\mu^n-\rho)\Big)^{2p}\Big]^{\frac{1}{p}}\bigg)^2\bigg)^{\frac{1}{2}},
\end{aligned}
\end{equation*}
which is \eqref{Eq51}.

\medskip

{\sc Step 3. Proof of \eqref{Eq22Bis}. }The estimate \eqref{Eq22Bis} is a consequence of \eqref{Eq51} applied with $p=1$ and 
\begin{equation}\label{Eq54Bis}
\mathbb{E}\Big[\Big((-s\Delta)^{\frac{1}{2}}\text{P}_{s+t}(\mu^n-\rho)(x)\Big)^2\Big]\lesssim \frac{\zeta(s,t)}{n}(1+\log^{\frac{1}{\eta}}(n)s^{\frac{1}{\log(n)}}\zeta(s,t)^{1-\frac{1}{\log(n)}}) \quad\text{for any  $s\in (0,\infty)$,}
\end{equation}
with 
\begin{equation}\label{DefZeta}
\zeta(s,t):=\min\big\{(s+t)^{-1},(s+t)^{-2}\big\}.
\end{equation}
Indeed, plugging \eqref{Eq54Bis} in \eqref{Eq51} gives
\begin{equation}\label{Eq57Bis}
\mathbb{E}\bigg[\Big(\int_{\mathcal{M}}\vert\nabla f\vert^q\,\dd\m\Big)^{\frac{2}{q}}\bigg]\lesssim \frac{1}{n}\int_{0}^{\infty}\zeta(s,t)(1+\log^{\frac{1}{\eta}}(n)s^{\frac{1}{\log(n)}}\zeta(s,t)^{1-\frac{1}{\log(n)}})\,\dd s.
\end{equation}
Using that 
$$\int_{0}^{\infty}\zeta(s,t)\,\dd s\leq t^{-1}\int_{0}^t\dd s+\int_{t}^1s^{-1}\,\dd s+\int_{1}^{\infty}s^{-2}\,\dd s\lesssim \vert\log(t)\vert,$$
and analogously
$$\int_{0}^{\infty}s^{\frac{1}{\log(n)}}\zeta(s,t)^{1-\frac{1}{\log(n)}}\,\dd s\lesssim 1,$$
\eqref{Eq57Bis} implies \eqref{Eq22Bis}.

\medskip

We now show \eqref{Eq54Bis}. First, using the definition \eqref{eq:empmeas} of $\mu^n$, we have
\begin{equation}\label{ExpandSquareLaplace}
(-s\Delta)^{\frac{1}{2}}\text{P}_{s+t}(\mu^n-\rho)=\frac{1}{n}\sum_{k=1}^n\omega_{s+t}(\cdot,X_k)\quad\text{with}\quad\omega_{s+t}(\cdot,y):=(-s\Delta)^{\frac{1}{2}}\text{P}_{s+t}(\delta_y-\rho),
\end{equation}
such that expanding the square gives
\begin{align*}
\Big((-s\Delta)^{\frac{1}{2}}\text{P}_{s+t}(\mu^n-\rho)(x)\Big)^2=\frac{1}{n^2}\sum_{k=1}^n \big(\omega_{s+t}(x,X_k)\big)^2+\frac{2}{n^2}\sum_{1\leq \ell<\ell'\leq n}\omega_{s+t}(x,X_\ell)\omega_{s+t}(x,X_{\ell'}).
\end{align*}
The estimate \eqref{Eq54Bis} is then a consequence of
\begin{equation}\label{Eq53}
\mathbb{E}[\vert\omega_{s+t}(x,X_1)\vert^2]\lesssim \min\Big\{(s+t)^{-1},(s+t)^{-2}\Big\}\quad\text{and}\quad\sup_{y\in\mathcal{M}}\vert\omega_{s+t}(x,y)\vert\lesssim\sqrt{s} \min\Big\{(s+t)^{-\frac{3}{2}},(s+t)^{-2}\Big\}.
\end{equation}
Indeed, the first item of \eqref{Eq53} immediately implies
$$\mathbb{E}\bigg[\sum_{k=1}^n \big(\omega_{s+t}(\cdot,X_k)\big)^2\bigg]\lesssim n\zeta(s,t),$$
while, using the assumption \eqref{DecayAlphaMixing} and that $t^{-\frac{3}{\log(n)}}\lesssim 1$,
\begin{equation}\label{Eq53Bis}
\begin{aligned}
&\sum_{1\leq \ell<\ell'\leq n}\mathbb{E}[\omega_{s+t}(\cdot,X_\ell)\omega_{s+t}(\cdot,X_{\ell'})]\\
&=\sum_{1\leq \ell<\ell'\leq n}\mathbb{E}[\omega_{s+t}(\cdot,X_\ell)\omega_{s+t}(\cdot,X_{\ell'})]^{\frac{1}{\log(n)}}\mathbb{E}[\omega_{s+t}(\cdot,X_\ell)\omega_{s+t}(\cdot,X_{\ell'})]^{1-\frac{1}{\log(n)}}\\
&\leq \sum_{1\leq \ell<\ell'\leq n}\alpha^{\frac{1}{\log(n)}}_{\ell-\ell'}\big(\sup_{y\in\mathcal{M}}\vert \omega_{s+t}(x,y)\vert\big)^{\frac{2}{\log(n)}}\mathbb{E}[\vert \omega_{s+t}(x,X_1)\vert^2]^{1-\frac{1}{\log(n)}}\\
&\stackrel{\eqref{Eq53},\eqref{DecayAlphaMixing}}\lesssim n\log^{\frac{1}{\eta}}(n)s^{\frac{1}{\log(n)}}t^{-\frac{3}{\log(n)}}\zeta(s,t)^{1-\frac{1}{\log(n)}}\lesssim\log^{\frac{1}{\eta}}(n)s^{\frac{1}{\log(n)}}\zeta(s,t)^{1-\frac{1}{\log(n)}}.
\end{aligned}
\end{equation}

It remains to prove \eqref{Eq53} and we start with the first item. Here, we use the fact that  
\begin{equation}\label{UniformContinuityL2}
\|(-s\Delta)^{\frac{1}{2}}\text{P}_{\frac{s}{2}}f\|_{\LL^2}\lesssim \|f\|_{\LL^2}\quad\text{for all $f\in\cc^2$ and uniformly in $s$.}
\end{equation}
This can be seen from \eqref{Laplacian12}: Using that $\pp_{\frac{s}{2}}$ is an auto-adjoint operator in $\LL^2$ and that from \eqref{Laplacian12:Eq2} we learn $\pp_{\frac{s}{2}}\phi_n =e^{-\lambda_n \frac{s}{2}}\phi_n$, we have 
\begin{align*}
\|(-s\Delta)^{\frac{1}{2}}\text{P}_{\frac{s}{2}}f\|^2_{\LL^2}=\sum_n (s\lambda_n)\vert \big(\phi_n,\pp_{\frac{s}{2}} f\big)_{\LL^2}\vert^2=\sum_n(s\lambda_n)e^{-s\lambda_n }\vert \big(\phi_n,f\big)_{\LL^2}\vert^2\lesssim \|f\|^2_{\LL^2}.
\end{align*}
Hence, using the semi-group property $\text{P}_{s+t}=\text{P}_{\frac{s}{2}}\text{P}_{\frac{s}{2}+t}$, we deduce
\begin{align*}
\mathbb{E}[\vert\omega_{s+t}(x,X_1)\vert^2]\lesssim \|(-s\Delta)^{\frac{1}{2}}\text{P}_{s+t}(\delta_y-\rho)\|^2_{\LL^{2}}&=\|(-s\Delta)^{\frac{1}{2}}\text{P}_{\frac{s}{2}}\text{P}_{\frac{s}{2}+t}(\delta_y-\rho)\|^2_{\LL^{2}}\\
&\stackrel{\eqref{UniformContinuityL2}}{\lesssim} \|\text{P}_{\frac{s}{2}+t}(\delta_y-\rho)\|^2_{\LL^{2}}.
\end{align*}
We now bound $\|\text{P}_{\frac{s}{2}+t}(\delta_y-\rho)\|^2_{\LL^{2}}$ in two different ways. First, using the bounds \eqref{Eq28} of the heat-kernel, we have by the triangle inequality
$$\|\text{P}_{\frac{s}{2}+t}(\delta_y-\rho)\|^2_{\LL^{2}}\lesssim \|p_{\frac{s}{2}+t}(\cdot,y)\|^2_{\LL^{2}}\stackrel{\eqref{Eq28}}{\lesssim} (s+t)^{-1},$$
yielding to the first alternative in the first item of \eqref{Eq53}. Second, applying Poincaré's inequality yields
$$\|\text{P}_{\frac{s}{2}+t}(\delta_y-\rho)\|^2_{\LL^{2}}\lesssim \|\nabla p_{\frac{s}{2}+t}(\cdot,y)\|^2_{\LL^{2}}\stackrel{\eqref{Eq28}}{\lesssim} (s+t)^{-2},$$
yielding to the second alternative in the first item of \eqref{Eq53}.

\medskip

We now turn to the second item of \eqref{Eq53}. Here, we make use of the representation formula \eqref{Eq45} applied to $f=\pp_{s+t}(\delta_y-\rho)$. Using the semi-group property $\text{P}_{\tau}\text{P}_{s+t}=\text{P}_{\tau+s+t}$, this takes the form
$$\omega_{s+t}(x,y)=\sqrt{\frac{s}{\pi}}\int_{0}^{\infty}\tau^{-\frac{1}{2}}\Delta\text{P}_{\tau+s+t}(\delta_y-\rho)(x)\,\dd\tau.$$
A direct application of the heat-kernel bounds \eqref{Eq28} leads to
\begin{align*}
\sup_{y\in\mathcal{M}}\vert\omega_{s+t}(x,y)\vert&\lesssim\sqrt{s}\int_{0}^{\infty}\tau^{-\frac{1}{2}}(\tau+s+t)^{-2}\,\dd\tau\lesssim \sqrt{s}\,(s+t)^{-\frac{3}{2}},
\end{align*}
which is the first alternative in the second item of \eqref{Eq53}. For the second alternative, we write 
$$\Delta \text{P}_{\tau+s+t}(\delta_y-\rho)(x)=\int_{\mathcal{M}}\rho\big(\Delta p_{\tau+s+t}(x,y)-\Delta p_{\tau+s+t}(x,\cdot)\big)\,\dd\m,$$
so that, using \eqref{Eq28},
\begin{align*}
\sup_{y\in\mathcal{M}}\vert\omega_{s+t}(x,y)\vert&\lesssim \sqrt{s}\int_{0}^{\infty}\tau^{-\frac{1}{2}}\sup_{y\in\mathcal{M}}\vert\nabla \Delta p_{\tau+s+t}(x,y)\vert\,\dd\tau\\
&\stackrel{\eqref{Eq28}}{\lesssim}\int_{0}^{\infty}\tau^{-\frac{1}{2}}(\tau+s+t)^{-\frac{5}{2}}\,\dd\tau\lesssim \sqrt{s}\,(s+t)^{-2}.
\end{align*}
{\sc Step 4. Proof of \eqref{Eq22}. }Following the same computations done in the previous step, the estimate \eqref{Eq22} is a consequence of \eqref{Eq51} and the moment bounds
\begin{equation}\label{Eq54}
\begin{aligned}
&\mathbb{E}\Big[\Big((-s\Delta)^{\frac{1}{2}}\text{P}_{s+t}(\mu^n-\rho)(x)\Big)^{2p}\Big]^{\frac{1}{p}}\\
&\lesssim p^2\bigg(\frac{\zeta(s,t)}{n}(1+\log^{\frac{1}{\eta}}(n)s^{\frac{1}{2\log(n)}}\zeta(s,t)^{1-\frac{1}{\log(n)}})+\frac{\sup_{y\in\mathcal{M}}\vert \omega_{s+t}(x,y)\vert+\log^2(n)}{n^2}\bigg) \quad\text{for any  $s\in (0,\infty)$,}
\end{aligned}
\end{equation}
together with the second item of \eqref{Eq53} and Lemma \ref{momentexp}, where we recall that $\zeta$ is defined in \eqref{DefZeta}.

\medskip

We now show \eqref{Eq54}. Using \eqref{ExpandSquareLaplace} and the assumption $\eta\geq 1$, we apply the concentration inequality in Proposition \ref{BersteinCorrelated} to the effect of: for any $\lambda\geq 0$
\begin{equation}\label{BoundProbaTailLqEsti}
\begin{aligned}
&\mathbb{P}\bigg(\Big\vert(-s\Delta)^{\frac{1}{2}}\text{P}_{s+t}(\mu^n-\rho)(x)\Big\vert\geq \lambda\bigg)\\
&\lesssim \exp\bigg(-\frac{1}{C}\frac{n^2\lambda^2}{nv^2+\big(\sup_{y\in\mathcal{M}}\vert\omega_{s+t}(x,y)\vert\big)^2+n\lambda\big(\sup_{y\in\mathcal{M}}\vert\omega_{s+t}(x,y)\vert\big)\log^2(n)}\bigg),
\end{aligned}
\end{equation}

with 
$$v^2:=\mathbb{E}[\vert\omega_{s+t}(x,X_1)\vert^2]+2\sum_{1\leq \ell<\ell'\leq n}\big\vert\mathbb{E}[\omega_{s+t}(x,X_\ell)\omega_{s+t}(x,X_{\ell'})]\big\vert.$$
We then obtain \eqref{Eq54} combining \eqref{BoundProbaTailLqEsti} with \eqref{Eq53} and \eqref{Eq53Bis}, together with an application of the Layer-cake formula.
\end{proof}

%
%
%

%
\end{proposition}
\subsection{Fluctuation estimates}\label{FluctuationSection}
This section is devoted to justify \eqref{FluctuationEstimates} needed to ensure the condition \eqref{SmallnessH} with very high probability. Our result is in the spirit of \cite[Theorem 3.3]{ambrosio2019finer}. However, our strategy differs from \cite[Theorem 3.3]{ambrosio2019finer} and is based on Schauder's theory, with an additional special care on the dependences on $\delta$.
We briefly sketch the main ingredients of the fluctuation estimates \eqref{FluctuationEstimates}. By linearity, it is enough to show \eqref{FluctuationEstimates} for $f^{n,t}_\delta\in \dot{\mathrm{H}}^1$ being the solution of 
\begin{equation}\label{Eq2}
-\nabla\cdot\rho_\delta\nabla f^{n,t}_\delta=\mu^{n,t}-\rho_t.
\end{equation}
We then make use of the chain rule to expand the equation in
\begin{equation}\label{Eq3}
-\Delta f^{n,t}_\delta=\tfrac{1}{\rho_\delta}\nabla\rho_\delta\cdot\nabla f^{n,t}_\delta+\tfrac{1}{\rho_\delta}(\mu^{n,t}-\rho_t),
\end{equation}
and we define the auxiliary problem
\begin{equation}\label{AuxiliaryProblemFluctu}
-\Delta u^{n,t}_\delta=\tfrac{1}{\rho_\delta}(\mu^{n,t}-\rho_t),
\end{equation}
such that the difference $v^{n,t}_\delta:=f^{n,t}_\delta-u^{n,t}_\delta$ solves $-\Delta v^{n,t}_\delta=\tfrac{1}{\rho_\delta}\nabla\rho_\delta\cdot\nabla f^{n,t}_\delta$. Doing so, on the one hand, we can control $u^{n,t}_\delta$ which can be handled using the explicit formula in terms of the heat-kernel, the explicit bounds on the latter, cf \eqref{Eq28}, and the regularity of $\frac{1}{\rho_\delta}$. On the other hand, we use Schauder's estimates to control $v^{n,t}_\delta$ and $f^{n,t}_\delta$. Using the fact that those estimates depend polynomially on $\|\rho_\delta\|_{\cc^{0,\alpha}}$, we can keep track on the dependences on $\delta$ that we can optimize later on. We finally mention that in the case where $\mathcal{M}$ has a boundary, $u^{n,t}_\delta$ cannot be directly defined by \eqref{AuxiliaryProblemFluctu} since the \rhs does not have zero mean. In order to also include this case, we add a zero order term in the equation, see \eqref{Eq3Bis}.
\begin{proposition}[Fluctuation estimates]\label{Fluctuation}
Let $\{\mu^n\}_n$ be defined in \eqref{eq:empmeas} with point clouds satisfying Assumption \ref{Assumptions}. For any parameter $\delta\in (0,1)$ and $\upsilon>0$, we define\footnote{with Neumann boundary conditions in case $\mathcal{M}$ has a boundary} $u^{n,t}_{\delta}\in \mathrm{H}^1$ weak solution of
\begin{equation}\label{Eq3Bis}
u^{n,t}_{\delta} -\Delta u^{n,t}_{\delta}=\tfrac{1}{\rho_\delta}(\mu^{n,t}-\rho_t),
\end{equation}
and the two events
\begin{equation}\label{Events}
\mathcal{A}_n:=\Big\{\|\mu^{n,t}-\rho_t\|_{\LL^{\infty}}\leq \tfrac{1}{\log^{\upsilon}(n)}\Big\}\quad \text{and}\quad\mathcal{B}_{\delta,n}:=\Big\{\big\|(\nabla u^{n,t}_\delta,\nabla^2 u^{n,t}_{\delta})\big\|_{\LL^{\infty}}\leq \tfrac{1}{\log^{\upsilon}(n)}\Big\}.
\end{equation}
There exists $\kappa>0$ such that for any $\kappa_1>0$ and the choice 
\begin{equation}\label{Eq:ChoicesPara}
\delta=\delta_n:=\tfrac{1}{\log^{\kappa_1}(n)},
\end{equation}
the solution $f^{n,t}_\delta\in \dot{\mathrm{H}}^1$ of \eqref{Eq2} satisfies
\begin{equation}\label{Eq6}
\mathds{1}_{\mathcal{A}_n\cap\mathcal{B}_{\delta,n}}\big\|(\nabla f^{n,t}_\delta,\nabla^2 f^{n,t}_{\delta})\big\|_{\LL^{\infty}}\leq \frac{1}{\log^{\upsilon-(\kappa+2)\kappa_1}(n)}.
\end{equation}
Furthermore, there exists $\kappa_2>0$ depending on $\upsilon$ such that for the choice $t=t_n:=\frac{\log^{\kappa_2}(n)}{n}$, we have
\begin{equation}\label{Eq13}
\mathbb{P}(\mathcal{A}^{c}_n\cup\mathcal{B}^c_{\delta,n})=o(\tfrac{1}{n^\ell})\quad\text{for any $\ell\in\mathbb{N}$}.
\end{equation}
%

%
%
%
%
%
%
%
%
%
%
\end{proposition}
\begin{proof}
The proof of Proposition \ref{Fluctuation} is split into two steps. In the first step, we prove \eqref{Eq6} where our main tool is Schauder's theory and an explicit formula for $u^{n,t}_{\delta}$ defined in \eqref{Eq3}. In the second step, we show \eqref{Eq13}, where our main tool is the concentration inequalities in Proposition \ref{BersteinCorrelated} and the explicit formula for $u^{n,t}_{\delta}$ used in the first step.

\medskip

{\sc Step 1. Proof of \eqref{Eq6}. }We define the difference $v^{n,t}_{\delta}:=f^{n,t}_\delta-\Big(u^{n,t}_{\delta}-\int_{\mathcal{M}} u^{n,t}_\delta\Big)\in \dot{\text{H}}^1$ and note that from \eqref{Eq3} and \eqref{Eq3Bis}, it solves
\begin{equation}\label{Eq5}
 -\Delta v^{n,t}_{\delta}=\frac{1}{\rho_{\delta}}\nabla\rho_\delta\cdot\nabla f^{n,t}_\delta-\Big(u^{n,t}_{\delta}-\int_{\mathcal{M}}u^{n,t}_\delta\Big).
\end{equation}
We claim that there exists $\kappa>0$ such that with the choice of $\delta$ in \eqref{Eq:ChoicesPara}, we have
\begin{equation}\label{Eq24}
\mathds{1}_{\mathcal{A}_n\cap\mathcal{B}_{\delta,n}}\big\|(\nabla v^{n,t}_{\delta},\nabla^2 v^{n,t}_{\delta})\big\|_{\LL^{\infty}}\lesssim \frac{1}{\log^{\upsilon-(\kappa+2)\kappa_1}(n)}.
\end{equation}
The estimate \eqref{Eq6} then follows from \eqref{Eq24} and the triangle inequality. 

\medskip

We now prove \eqref{Eq24}. We apply Schauder's estimate (see for instance \cite{gilbarg2015elliptic} and \cite[Chapter 10]{nicolaescu2020lectures}) to \eqref{Eq5} to obtain
\begin{equation}\label{Eq11}
\big\|(\nabla v^{n,t}_{\delta},\nabla^2 v^{n,t}_{\delta})\big\|_{\LL^{\infty}}\lesssim \|\tfrac{1}{\rho_{\delta}}\nabla\rho_\delta\cdot\nabla f^{n,t}_\delta\|_{\cc^{0,\alpha}}+\Big\|u^{n,t}_{\delta}-\int_{\mathcal{M}}u^{n,t}_\delta\Big\|_{\cc^{0,\alpha}}.
\end{equation}
While the second \rhs term is directly of order of $\frac{1}{\log^{\upsilon}(n)}$ in $\mathcal{B}_{\delta,n}$, the first \rhs term requires additional attention. Using \eqref{Ellipticity} and \eqref{Eq28}, $\rho_\delta$ satisfies
\begin{equation}\label{Eq25}
\rho_{\delta}\geq \lambda\quad \text{and}\quad \delta^{-1}\|\nabla\rho_{\delta}\|_{\LL^{\infty}}+\|\nabla^2\rho_{\delta}\|_{\LL^{\infty}}\lesssim \delta^{-2},
\end{equation}
and together with the algebraic property of $\|\cdot\|_{\cc^{0,\alpha}}$, we have
$$\|\tfrac{1}{\rho_{\delta}}\nabla\rho_\delta\cdot\nabla f^{n,t}_\delta\|_{\cc^{0,\alpha}}\leq \|\tfrac{1}{\rho_{\delta}}\nabla\rho_\delta\|_{\cc^{0,\alpha}}\|\nabla f^{n,t}_\delta\|_{\cc^{0,\alpha}}\lesssim \delta^{-2}\|\nabla f^{n,t}_\delta\|_{\cc^{0,\alpha}}.$$
The latter is bounded using Schauder's estimate applied this time on \eqref{Eq2} (knowing that the dependence on $\|\rho_{\delta}\|_{\cc^{0,\alpha}}$ is at most polynomial): there exists $\kappa>0$ such that
\begin{equation*}
\begin{aligned}
\|\nabla f^{n,t}_\delta\|_{\cc^{0,\alpha}}&\lesssim \|\rho_{\delta}\|^{\kappa}_{\cc^{0,\alpha}}\|\mu^{n,t}-\rho_t\|_{\LL^{\infty}}\stackrel{\eqref{Eq25}}{\lesssim} \delta^{-\kappa}\|\mu^{n,t}-\rho_t\|_{\LL^{\infty}}\stackrel{\eqref{Eq:ChoicesPara}}{\lesssim} \log^{\kappa\kappa_1}(n)\|\mu^{n,t}-\rho_t\|_{\LL^{\infty}},
\end{aligned}
\end{equation*}
which yields the following control of the first \rhs term of \eqref{Eq11}
\begin{equation}\label{Eq7}
\|\tfrac{1}{\rho_{\delta}}\nabla\rho_\delta\cdot\nabla f^{n,t}_\delta\|_{\cc^{0,\alpha}}\lesssim \log^{(\kappa+2)\kappa_1}(n)\|\mu^{n,t}-\rho_t\|_{\LL^{\infty}},
\end{equation}
which is of order of $\frac{1}{\log^{\upsilon-(\kappa+2)\kappa_1}(n)}$ in $\mathcal{A}_n$.

\medskip

{\sc Step 2. Proof of \eqref{Eq13}. }We provide the arguments for \eqref{Eq13} in case that \eqref{DecayAlphaMixing} holds for $\eta<1$ and the event $\mathcal{B}^c_{\delta,n}$ that we reduce to $\{\|\nabla^2 u^{n,t}_\delta\|_{\LL^{\infty}}\leq \tfrac{1}{\log^{\upsilon}(n)}\}$, the other cases as well as the case of $\nabla u^{n,t}_\delta$ follow by a straightforward adaptation. The estimate \eqref{Eq13} is a consequence of 
\begin{equation}\label{Eq59}
\begin{aligned}
\sup_{x\in\mathcal{M}}\mathbb{P}\Big(\vert\partial^2_{ij} u^{n,t}_{\delta}(x)\vert\geq& \tfrac{1}{2\log^{\upsilon}(n)}\Big)\leq n\exp\bigg(-\frac{1}{C_1}\Big(\frac{nt}{\log^{\upsilon}(n)}\Big)^{\eta}\bigg)+\exp\bigg(-\frac{1}{C_2}\frac{n^2t\log^{-2\upsilon}(n)}{t^{-1}+n\delta^{-4}\log^{\frac{1}{\eta}}(n)}\bigg)\\
&+\exp\bigg(-\frac{1}{C_3}\frac{nt^2}{\log^{\upsilon}(n)}\exp\Big(\frac{1}{C_4}\Big(\frac{nt}{\log^{\upsilon}(n)}\Big)^{\eta(1-\eta)}\log^{-1}(\tfrac{nt}{\log^{\upsilon}(n)})\Big)\bigg).
\end{aligned}
\end{equation}
for some constants $C_1,C_2,C_3,C_4>0$. To see this, let $\eta_n$ be defined by
\begin{equation}\label{Eq72}
\eta_n=\frac{1}{2\log^{\upsilon}(n)\|\nabla\partial^2_{ij}u^{n,t}_{\delta}\|_{\LL^{\infty}}}.
\end{equation}
By compactness of $\mathcal{M}$, we can find a $\eta_n$-net $\{x_k\}_{1\leq k\leq N}\subset \mathcal{M}$ with $N\lesssim \eta^{-2}_n$. We note that
\begin{equation}\label{Eq77}
\Big\{\|\partial^2_{ij}u^{n,t}_{\delta}\|_{\LL^{\infty}}>\tfrac{1}{\log^{\upsilon}(n)}\Big\}\subset\bigcup_{k=1}^{N}\Big\{\vert\partial^2_{ij}u^{n,t}_{\delta}(x_k)\vert>\tfrac{1}{2\log^{\upsilon}(n)}\Big\}.
\end{equation}
Indeed, if for any $k\in\{1,\cdots,N\}$, $\vert\partial^2_{ij}u^{n,t}_{\delta}(x_k)\vert\leq\tfrac{1}{2\log^{\upsilon}(n)}$ then for any $x\in \mathcal{M}$ there exists $j\in\{1,\cdots,N\}$ such that
$$\vert\partial^2_{ij}u^{n,t}_{\delta}(x)\vert\leq \eta_n\|\nabla\partial^2_{ij}u^{n,t}_{\delta}\|_{\LL^{\infty}}+\vert\partial^2_{ij}u^{n,t}_{\delta}(x_j)\vert\stackrel{\eqref{Eq72}}{\lesssim} \frac{1}{\log^{\upsilon}(n)}.$$
Applying $\mathbb{P}$ on \eqref{Eq77} yields 
\begin{align*}
\mathbb{P}\Big(\|\partial^2_{ij} u^{n,t}_{\delta}\|_{\LL^{\infty}}\geq \tfrac{1}{\log^{\upsilon}(n)}\Big)&\leq \sum_{k=1}^N\mathbb{P}\Big(\vert\partial^2_{ij} u^{n,t}_{\delta}(x_k)\vert\geq \tfrac{1}{2\log^{\upsilon}(n)}\Big)\\
&\lesssim \eta^{-2}_{n}\sup_{x\in\mathcal{M}}\mathbb{P}\Big(\vert\partial^2_{ij} u^{n,t}_{\delta}(x)\vert\geq \tfrac{1}{2\log^{\upsilon}(n)}\Big).
\end{align*}
Using \eqref{Eq59} and 
$$\eta^{-2}_n\lesssim \log^{2\upsilon}(n)t^{-3},$$
which can be proven following the arguments leading to the second item of \eqref{Eq58} below, this yields $\P(\{\|\nabla^2 u^{n,t}_\delta\|_{\LL^{\infty}}\leq \tfrac{1}{\log^{\upsilon}(n)}\})=o(\tfrac{1}{n^m})$ for any $m\in\mathbb{N}$ for the choice $\kappa_2>4\kappa_1+\frac{1}{\eta}+2\upsilon$.
\medskip

We now prove \eqref{Eq59}. We first exploit the following explicit representation formula\footnote{a simple change of variable gives that the kernel $\tilde{p}_s$ associated to $1-\Delta$ is given by $\tilde{p}_s=e^{-s}p_s$} for $u^{n,t}_{\delta}$, 
\begin{equation}\label{eq:stella3}
u^{n,t}_{\delta}=\int_0^{\infty}e^{-s}\,\text{P}_s\Big(\tfrac{1}{\rho_\delta}(\mu^{n,t}-\rho_t)\Big)\,\dd s,
\end{equation}
that we expand, using the definition \eqref{eq:empmeas} of $\mu^{n}$, in form of
\begin{equation}\label{eq:stella}
u^{n,t}_{\delta}=\frac{1}{n}\sum_{k=1}^n\omega(\cdot,X_k)\quad\text{with } \omega(\cdot,y):=\int_{0}^{\infty}e^{-s}\,\text{P}_s\Big(\tfrac{1}{\rho_{\delta}}(p_t(\cdot,y)-\rho_t)\Big)\,\dd s.
\end{equation}
Then applying the concentration inequalities in Proposition \ref{BersteinCorrelated}, we get 
\begin{equation}\label{eq:stella2}
\begin{split}
\mathbb{P}\Big(\vert\partial^2_{ij} u^{n,t}_{\delta}(x)\vert &\geq \tfrac{1}{\log^{\upsilon}(n)}\Big)\leq n\exp\bigg(-\frac{1}{C_1}\Big(\frac{n}{M\log^{2}(n)}\Big)^{\beta}\bigg)+\exp\bigg(-\frac{1}{C_2}\frac{n^2\log^{-2\upsilon}(n)}{M^2+n v^2}\bigg)\\
&+\exp\bigg(-\frac{1}{C_3}\frac{n\lambda}{M^2\log^{\upsilon}(n)}\exp\Big(\frac{1}{C_4}\Big(\frac{n}{M\log^{\upsilon}(n)}\Big)^{\beta(1-\beta)}\log^{-1}(\tfrac{n}{M\log^{\upsilon}(n)})\Big)\bigg),
\end{split}
\end{equation}
where 
\begin{equation}\label{Eq79}
M:=\sup_{y\in\mathcal{M}}\vert\partial^2_{ij}\omega(x,y)\vert\quad\text{and}\quad v^2:=\mathbb{E}[\vert\partial^2_{ij}\omega(x,X_1)\vert^2]+2\sup_{\ell\geq 1}\sum_{k>\ell}\vert\mathbb{E}[\partial^2_{ij}\omega(x,X_\ell)\partial^2_{ij}\omega(x,X_k)]\vert.
\end{equation}
The estimate \eqref{Eq59} then follows from the three following estimates
\begin{equation}\label{Eq58}
\mathbb{E}[\vert\partial^2_{ij} \omega(x,X_1)\vert^2]\lesssim \delta^{-4}t^{-1},\quad \sup_{y\in\mathcal{M}}\vert\partial^2_{ij} \omega(x,y)\vert\lesssim t^{-1}\quad\text{and}\quad v^2\lesssim \log^{\frac{1}{\eta}}(n)\delta^{-4}t^{-1} ,
\end{equation}
that we prove separately in the next three sub-steps.

\medskip

{\sc Sub-step 2.1. Proof of the first item of \eqref{Eq58}. }Splitting the time integral into $\int_{0}^t+\int_{t}^\infty$ and subtracting and adding back $\tfrac{1}{\rho_\delta(x)}$ in the first integral as well as using the semigroup property of $\{\pp_s\}_{s}$ in form of 
$$\int_0^t\dd s\, e^{-s}\,\pp_s(p_t(\cdot,y)-\rho_t)=\int_{t}^{2t}\dd s\, e^{-(s-t)}(p_s(\cdot,y)-\rho_s),$$
we decompose $\omega$ into a regular-part $\mathcal{J}_1$ and a singular-part $\mathcal{J}_2$: 
\begin{equation}\label{Eq68}
\begin{aligned}
\omega(x,y)=&\underbrace{\frac{1}{\rho_{\delta}(x)}\int_t^{2t} \dd s\, e^{-s}(p_s(\cdot,y)-\rho_{s})(x)
+\int_{t}^{\infty}\dd s\, e^{-s}\text{P}_s\Big(\tfrac{1}{\rho_\delta}(p_t(\cdot,y)-\rho_t)\Big)(x)}_{=:\mathcal{J}_1(x,y)}\\
&+\underbrace{\int_{0}^{t}\dd s\, e^{-s}\text{P}_s\Big((\tfrac{1}{\rho_\delta}-\tfrac{1}{\rho_\delta(x)})(p_t(\cdot,y)-\rho_t)\Big)(x)}_{=:\mathcal{J}_2(x,y)}.
\end{aligned}
\end{equation}
For the regular-part $\mathcal{J}_1$, we apply directly the heat-kernel bounds \eqref{Eq28} and the first item of \eqref{Eq25}, to obtain from the triangle inequality and Minkowski's inequality
\begin{equation}\label{Eq61}
\begin{aligned}
\mathbb{E}[\vert\partial^2_{ij}\mathcal{J}_1(x,X_1)\vert^2]\lesssim& \|\partial^2_{ij}\mathcal{J}_1(x,\cdot)\|^2_{\LL^{2}}\\
\lesssim &(\|\nabla\rho_{\delta}\|^4_{\LL^{\infty}}+\|\nabla^2\rho_{\delta}\|^2_{\LL^{\infty}})\Big(\int_{t}^{2t}\|p_s(x,\cdot)\|_{\LL^{2}}\,\dd s\Big)^{2}\\
&\underbrace{+\|\nabla\rho_{\delta}\|^2_{\LL^{\infty}}\Big(\int_{t}^{2t}\|\nabla p_s(x,\cdot)\|_{\LL^{2}}\,\dd s\Big)^{2}+\Big(\int_{t}^{2t}\|\nabla^2 p_s(x,\cdot)\|_{\LL^{2}}\,\dd s\Big)^{2}}_{=:\mathcal{R}_1(x)}\\
&+\underbrace{\bigg(\int_{t}^{\infty}\dd s\bigg(\int_{\mathcal{M}}\dd \m(y)\Big(\partial^2_{ij}\text{P}_s\Big(\tfrac{1}{\rho_{\delta}}(p_t(\cdot,y)-\rho_t)\Big)(x)\Big)^2\bigg)^{\frac{1}{2}}\bigg)^2}_{=:\mathcal{R}_2(x)}.
\end{aligned}
\end{equation}
The first \rhs term $\mathcal{R}_1(x)$ is dominated using directly the heat-kernel bounds \eqref{Eq28} and \eqref{Eq25}
\begin{equation}\label{Eq62}
\mathcal{R}_1(x)\lesssim \delta^{-4}t^{\frac{1}{2}}+\delta^{-2}+t^{-1}\lesssim t^{-1}.
\end{equation}
For the second \rhs side term $\mathcal{R}_2(x)$, we first simplify the $y$-integral. Using that 
\begin{equation}\label{Eq63}
\int_{\mathcal{M}}\vert p_t(z,w)-\rho_t(z)\vert\,\dd \m(w)\lesssim 1\quad\text{and}\quad\int_{\mathcal{M}}\vert p_t(w,y)-\rho_t(w)\vert\,\dd \m(w)\lesssim 1\quad \text{for any $z,y\in \mathcal{M}$,}
\end{equation}
we have by Jensen's inequality and the heat-kernel bounds \eqref{Eq28}
\begin{align*}
\lefteqn{\int_{\mathcal{M}}\Big(\partial^2_{ij}\text{P}_s\Big(\tfrac{1}{\rho_{\delta}}(p_t(\cdot,y)-\rho_t)\Big)(x)\Big)^2\,\dd \m(y)}\\ &=\int_{\mathcal{M}}\bigg(\int_{\mathcal{M}}\dd w\,\partial^2_{ij}p_{s}(x,w)\tfrac{1}{\rho_{\delta}(w)}(p_t(w,y)-\rho_t(w))\bigg)^{2}\,\dd \m(y)\\
&\stackrel{\eqref{Eq63}}{\lesssim}\int_{\mathcal{M}}\vert\partial^2_{ij}p_s(x,w)\vert^2\,\dd \m(w)\lesssim s^{-3}.
\end{align*}
Thus, 
\begin{equation}\label{Eq64}
\mathcal{R}_2(x)\lesssim t^{-1}.
\end{equation}
The combination of \eqref{Eq61}, \eqref{Eq62} and \eqref{Eq63} yields
\begin{equation}\label{Eq65}
\mathbb{E}[\vert\partial^2_{ij}\mathcal{J}_1(x,X_1)\vert^2]\lesssim t^{-1}.
\end{equation}
We now turn to the singular-part $\mathcal{J}_2$. We first apply Minkowski's inequality in form of
\begin{equation}\label{Eq66}
\begin{split}
\mathbb{E}[\vert\mathcal{J}_2(x,X_1)\vert^2]&\lesssim \|\partial^2_{ij}\mathcal{J}_2(x,\cdot)\|^2_{\LL^{2}}\\& \lesssim \bigg(\int_{0}^t\dd s\bigg(\int_{\mathcal{M}}\dd \m(y)\Big(\partial^2_{ij}\text{P}_s\Big((\tfrac{1}{\rho_{\delta}}-\tfrac{1}{\rho_{\delta}(x)})(p_t(\cdot,y)-\rho_t)\Big)(x)\Big)^2\bigg)^{\frac{1}{2}}\bigg)^{2}.
\end{split}
\end{equation}
We then simplify the $y$-integral. To this aim, we bound the integrand in $\LL^{\infty}$ using the heat-kernel bounds \eqref{Eq28}, \eqref{Eq25} and 
\begin{equation}\label{Eq29}
\Big\vert \tfrac{1}{\rho_\delta(y)}-\tfrac{1}{\rho_\delta(x)}\Big\vert\stackrel{\eqref{Eq25}}{\lesssim} \delta^{-1}\dd(x,y)\quad \text{for any $x, y\in\mathcal{M}$},
\end{equation}
in form of
\begin{equation}\label{Eq69}
\begin{aligned}
& \Big\vert\partial^2_{ij}\text{P}_s\Big((\tfrac{1}{\rho_{\delta}}-\tfrac{1}{\rho_{\delta}(x)})(p_t(\cdot,y)-\rho_t)\Big)(x)\Big\vert\lesssim t^{-1}\int_{\mathcal{M}}\bigg\vert \partial^2_{ij}\Big(p_s(x,w)(\tfrac{1}{\rho_{\delta}(w)}-\tfrac{1}{\rho_{\delta}(x)})\Big)\bigg\vert\,\dd \m(w)\\
\stackrel{\eqref{Eq25},\eqref{Eq29}}{\lesssim}&t^{-1}\delta^{-1}\bigg(\int_{\mathcal{M}}\vert\nabla^2 p_s(x,w)\vert \dd(x,w)\,\dd \m(w)+\int_{\mathcal{M}}\vert\nabla p_s(x,w)\vert\,\dd \m(w)\bigg)\\
& +t^{-1}\delta^{-2}\int_{\mathcal{M}}\vert p_s(x,w)\vert\,\dd \m(w)\\
\stackrel{\eqref{Eq28}}{\lesssim} &t^{-1}\delta^{-1}s^{-\frac{1}{2}}+t^{-1}\delta^{-2}.
\end{aligned}
\end{equation}
This yields together with \eqref{Eq66}
\begin{equation}\label{Eq67}
\mathbb{E}[\vert\mathcal{J}_2(x,X_1)\vert^2]\lesssim \delta^{-4}t^{-1}.
\end{equation}
To conclude, the combination of \eqref{Eq68}, \eqref{Eq65} and \eqref{Eq67} shows the first item of \eqref{Eq58}.

\medskip

{\sc Sub-step 2.2. Proof of the second item of \eqref{Eq58}. }We use the decomposition \eqref{Eq68}. For the regular-part $\mathcal{J}_1$, we argue as in \eqref{Eq61} for the first term whereas the second-term is estimated using the heat-kernel bounds \eqref{Eq28} in form of
\begin{align*}
\bigg\vert\partial^2_{ij}\text{P}_s\Big(\tfrac{1}{\rho_\delta}(p_t(\cdot,y)-\rho_t)\Big)(x)\bigg\vert&=\bigg\vert\int_{\mathcal{M}}\dd \m(w)\,\partial^2_{ij}p_{s}(x,w)\tfrac{1}{\rho_{\delta}(w)}(p_t(w,y)-\rho_t(w))\bigg\vert\\
&\stackrel{\eqref{Eq28},\eqref{Eq25}}{\lesssim} s^{-2}\int_{\mathcal{M}}\dd \m(w)\, \vert p_t(w,y)-\rho_t(w)\vert\\
&\lesssim s^{-2},
\end{align*}
so that
\begin{align*}
\bigg\vert\int_{t}^{\infty}\dd s\, e^{-s}\partial^2_{ij}\text{P}_s\Big(\tfrac{1}{\rho_\delta}(p_t(\cdot,y)-\rho_t)\Big)(x)\bigg\vert\lesssim t^{-1}.
\end{align*}
Hence,
\begin{equation}\label{Eq71}
\sup_{y\in\mathcal{M}}\vert\mathcal{J}_1(x,y)\vert\lesssim t^{-1}.
\end{equation}
\medskip

For the singular-part $\mathcal{J}_2$, we use the bound \eqref{Eq69} which directly yields
\begin{equation}\label{Eq70}
\sup_{y\in\mathcal{M}}\vert\mathcal{J}_2(x,y)\vert\lesssim \delta^{-1}t^{-\frac{1}{2}}.
\end{equation}
The combination of \eqref{Eq68}, \eqref{Eq71} and \eqref{Eq70} gives the second item of \eqref{Eq58}.

\medskip

{\sc Sub-Step 2.3. Proof of the third item of \eqref{Eq58}. }According to the first item of \eqref{Eq58}, it suffices to give the argument for the second term in the definition \eqref{Eq79} of $v^2$. We use the assumption \eqref{DecayAlphaMixing} together with the two first items of \eqref{Eq58} in form of
\begin{align*}
\lefteqn{\sum_{k>\ell}\vert\mathbb{E}[\partial^2_{ij}\omega(x,X_i)\partial^2_{ij}\omega(x,X_j)]\vert}\\ &=\sum_{k>\ell}\vert\mathbb{E}[\partial^2_{ij}\omega(x,X_\ell)\partial^2_{ij}\omega(x,X_k)]\vert^{\frac{1}{\log(n)}}\vert\mathbb{E}[\partial^2_{ij}\omega(x,X_\ell)\partial^2_{ij}\omega(x,X_k)]\vert^{1-\frac{1}{\log(n)}}\\
&\lesssim \sum_{k>\ell}\alpha^{\frac{1}{\log(n)}}_{k-\ell}(\sup_{y\in\mathcal{M}}\vert \partial^2_{ij}\omega(x,y)\vert)^{\frac{2}{\log(n)}}(\mathbb{E}[\vert\partial^2_{ij}\omega(x,X_1)\vert^2])^{1-\frac{1}{\log(n)}}\\
&\stackrel{\eqref{Eq58},\eqref{DecayAlphaMixing}}{\lesssim} \delta^{-4(1-\frac{1}{\log(n)})} t^{-1-\frac{1}{\log(n)}}\sum_{k=0}^{\infty}\exp(-b\tfrac{k^{\eta}}{\log(n)})\\
&\lesssim \delta^{\frac{4}{\log(n)}}t^{-\frac{1}{\log(n)}}\log^{\frac{1}{\eta}}(n)\delta^{-4}t^{-1},
\end{align*}
which concludes since $\limsup_{n\uparrow\infty}\delta^{\frac{4}{\log(n)}}t^{-\frac{1}{\log(n)}}\lesssim 1$. 
\end{proof}
\subsection{Contractivity estimates}
This section is devoted to the control of the smoothing errors $W^2_2(\mu^{n,t},\mu^{n})$ and $W^2_2(\nu^{m,t},\nu^{m})$ for the particular choice of $t$ given in Proposition \ref{Fluctuation}. The first result is in the spirit of \cite[Theorem 5.2]{ambrosio2019finer} that we extend in the case of non-uniformly distributed and correlated points. This extension requires a finer analysis of the error and the proof relies on Berry-Esseen type inequalities in the spirit of \cite[Theorem 5]{borda2021empirical}.
\begin{proposition}[Semigroup contraction for empirical measures]\label{Contractivity}
Let $\{\mu^n\}_n$ be defined in \eqref{eq:empmeas} with point clouds satisfying Assumption \ref{Assumptions}. Given $t$ such that Proposition \ref{Fluctuation} holds, we have
\begin{equation}\label{ContractivityEstiExpectation}
W_2^2(\mu^{n,t}, \mu^{n}) \le \mathcal{C}_n \frac{\log\log(n)}{n}+t\big\|\rho_{t+\frac{1}{n}}-\rho_{\frac{1}{n}}\big\|_{\LL^1},
\end{equation}
for some random variable $\mathcal{C}_n$ satisfying for $C<\infty$
$$\sup_{n \ge 1} \mathbb{E} [\tfrac1{C}\mathcal{C}_n]\leq 1.$$
Furthermore, if \eqref{DecayAlphaMixing} holds with $\eta\geq 1$ then the assumption \eqref{DecatBetaMixing} can be dropped and the stochastic integrability can be improved up to losing a $\log(n)$ factor, namely
\begin{equation}\label{ContractivityEstiExpMoment}
W_2^2(\mu^{n,t}, \mu^{n}) \le \mathcal{D}_n \frac{\log^{\frac{1}{\eta}}(n)\log\log(n)}{n}+t\big\|\rho_{t+\frac{1}{n}}-\rho_{\frac{1}{n}}\big\|_{\LL^1},
\end{equation}
for some random variable $\mathcal{D}_n$ satisfying for $D<\infty$
$$\sup_{n \ge 1} \mathbb{E} \big[\exp(\tfrac1{D}\mathcal{D}^{\frac{1}{2}}_n)\big]\leq 2.$$
\end{proposition}
 \begin{proof}
 According to the fluctuation estimates in Proposition \ref{Fluctuation} together with $W^2_2(\mu^{n,t},\mu^n)\leq (\text{diam}(\mathcal{M}))^2$, we can restrict the analysis in $\mathcal{A}_n$ defined in \eqref{Events}. Note that for $n$ large enough, \eqref{Ellipticity} yields
\begin{equation}\label{EllipticityAn}
\frac{\lambda}{2}\leq \mu^{n,t}\leq \Lambda+1\quad\text{in $\mathcal{A}_n$.}
\end{equation}
\medskip

We split the proof into three steps. In the first step, we prove a Berry–Esseen type smoothing inequality for $W^2_2(\mu^{n,t},\mu^n)$ which decomposes the error in a deterministic part involving $\rho$ and a random part involving the Fourier coefficients $\{\widehat{\mu^n}(k)\}_k$ of $\mu^n$. In the second step, we prove \eqref{ContractivityEstiExpectation}. In the third step, we control the fluctuations of $\{\widehat{\mu^n}(k)\}_k$ using the concentration inequalities in Proposition \ref{BersteinCorrelated} and deduce \eqref{ContractivityEstiExpMoment}.

\medskip

{\sc Step 1. Berry-Esseen type inequality. }Recalling that we denote by $\{\lambda_k,\phi_k\}_k$ the eigenvalues and eigenfunctions of $-\Delta$ respectively, we prove that 
\begin{equation}\label{SmoothingIneg}
W^2_2(\mu^{n,t},\mu^n)\lesssim \frac{1}{n}+\sum_{k\geq 1} \frac{e^{-\frac{2}{n}\lambda_k}}{\lambda_k}\big(e^{-t\lambda_k}-1\big)^2\vert\widehat{\mu^n}(k)-\widehat{\rho}(k)\vert^2+t\|\rho_{t+\frac{1}{n}}-\rho_{\frac{1}{n}}\|_{\LL^1},
\end{equation}
where 
$$\widehat{\mu^n}(k):=\int_{\mathcal{M}}\phi_k\,\dd\mu^n \quad \text{and}\quad \widehat{\rho}(k):=\int_{\mathcal{M}}\rho\,\phi_k\,\dd\m.$$
We first apply the triangle inequality and use the classical contractivity estimate in \cite[Theorem 3]{erbar2015equivalence} to get
\begin{equation}\label{EstiPeyre:Eq4}
W^2_2(\mu^{n,t},\mu^n)\lesssim W^2_2(\mu^{n,\frac{1}{n}},\mu^{n})+W^2_2(\mu^{n,t+\frac{1}{n}},\mu^{n,t})+W^2_2(\mu^{n,t+\frac{1}{n}},\mu^{n,\frac{1}{n}})\lesssim \frac{1}{n}+W^2_2(\mu^{n,t+\frac{1}{n}},\mu^{n,\frac{1}{n}}).
\end{equation}
We then apply Peyre's estimate \cite{peyre2018comparison} to the second \rhs, which takes the form
\begin{equation}\label{PeyreLemma}
W^2_2(\mu^{n,t+\frac{1}{n}},\mu^{n,\frac{1}{n}})\leq 4\sup\bigg\{\Big\vert\int_{\mathcal{M}}(\mu^{n,t+\frac{1}{n}}-\mu^{n,\frac{1}{n}})\,f\,\dd\m\Big\vert^2\quad\text{with}\quad \int_{\mathcal{M}}\mu^{n,t+\frac{1}{n}}\,\vert\nabla f\vert^2\,\dd\m\leq 1\bigg\}.
\end{equation}
Now, given an arbitrary $f$ such that 
\begin{equation}\label{Boundtestfunctionf}
\int_{\mathcal{M}}\mu^{n,t+\frac{1}{n}}\,\vert\nabla f\vert^2\,\dd\m\leq 1,
\end{equation}
we split
\begin{equation}\label{EstiPeyre:Eq1}
\int_{\mathcal{M}}(\mu^{n,t+\frac{1}{n}}-\mu^{n,\frac{1}{n}})\, f\,\dd\m=\int_{\mathcal{M}}\big(\mu^{n,t+\frac{1}{n}}-\mu^{n,\frac{1}{n}}-(\rho_{t+\frac{1}{n}}-\rho_{\frac{1}{n}})\big)\, f\,\dd\m+\int_{\mathcal{M}}f(\rho_{t+\frac{1}{n}}-\rho_{\frac{1}{n}})\,\dd\m.
\end{equation}
For the first \rhs term of \eqref{EstiPeyre:Eq1}, we expand the integral using \eqref{SpectralDecompoHeatKernel}. Thus, together with the semigroup property of $\{\mathrm{P}_t\}_{t>0}$ and Cauchy-Schwarz's inequality, we get
\begin{align*}
&\int_{\mathcal{M}}\big(\mu^{n,t+\frac{1}{n}}-\mu^{n,\frac{1}{n}}-(\rho_{t+\frac{1}{n}}-\rho_{\frac{1}{n}})\big)\, f\,\dd\m\\
&=\int_{\mathcal{M}}\dd\big(\mu^{n,t}-\mu^{n}-(\rho_{t}-\rho)\big)(y)\int_{\mathcal{M}}\dd\m(x)\,f(x)p_{\frac{1}{n}}(x,y)\\
&\stackrel{\eqref{SpectralDecompoHeatKernel}}{=}\sum_{k\geq 1}e^{-\frac{1}{n}\lambda_k}\int_{\mathcal{M}}\dd\big(\mu^{n,t}-\mu^{n}-(\rho_{t}-\rho)\big)(y)\int_{\mathcal{M}}\dd\m(x)\,f(x)\phi_k(x)\phi_k(y)\\
&=\sum_{k\geq 1}e^{-\frac{1}{n}\lambda_k}\widehat{f}(k)\big(\widehat{\mu^{n,t}}(k)-\widehat{\mu^{n}}(k)-(\widehat{\rho_t}(k)-\widehat{\rho}(k))\big)\\
&\leq \Big(\sum_{k\geq 1}\lambda_k\vert\widehat{f}(k)\vert^2\Big)^{\frac{1}{2}}\Big(\sum_{k\geq 1}\frac{e^{-\frac{2}{n}\lambda_k}}{\lambda_k}\vert \widehat{\mu^{n,t}}(k)-\widehat{\mu^{n}}(k)-(\widehat{\rho_t}(k)-\widehat{\rho}(k))\vert^2\Big)^{\frac{1}{2}}
\end{align*}
Using that from \eqref{EllipticityAn} we have $\mu^{n,t+\frac{1}{n}}\geq \frac{\lambda}{2}$ and recalling \eqref{Boundtestfunctionf}, we get
$$ \Big(\sum_{k\geq 1}\lambda_k\vert\widehat{f}(k)\vert^2\Big)^{\frac{1}{2}}\leq \Big(\int_{\mathcal{M}}\vert \nabla f\vert^2\,\dd\m\Big)^{\frac{1}{2}}\leq \frac{2}{\lambda}\Big(\int_{\mathcal{M}}\mu^{n,t+\frac{1}{n}}\, \vert \nabla f\vert^2\,\dd\m\Big)^{\frac{1}{2}}\leq \frac{2}{\lambda}.$$
Furthermore, noticing that $\pp_t\phi_k=e^{-t\lambda_k}\phi_k$ which implies that, since $\pp_t$ is self-adjoint
$$\widehat{\mu^{n,t}}(k)=e^{-t\lambda_k}\widehat{\mu^n}(k)\quad\text{and}\quad \widehat{\rho_t}(k)=e^{-t\lambda_k}\widehat{\rho}(k),$$
we obtain
$$\sum_{k\geq 1}\frac{e^{-\frac{2}{n}\lambda_k}}{\lambda_k}\vert \widehat{\mu^{n,t}}(k)-\widehat{\mu^{n}}(k)-(\widehat{\rho_t}(k)-\widehat{\rho}(k))\vert^2=\sum_{k\geq 1}\frac{e^{-\frac{2}{n}\lambda_k}}{\lambda_k}\big(e^{-t\lambda_k}-1\big)^2\vert\widehat{\mu^n}(k)-\widehat{\rho}(k)\vert^2.$$
This leads to 
\begin{equation}\label{EstiPeyre:Eq2}
\int_{\mathcal{M}}\dd\big(\mu^{n,t+\frac{1}{n}}-\mu^{n,\frac{1}{n}}-(\rho_{t+\frac{1}{n}}-\rho_{\frac{1}{n}})\big)\, f\lesssim \Big(\sum_{k\geq 1}\frac{e^{-\frac{2}{n}\lambda_k}}{\lambda_k}\big(e^{-t\lambda_k}-1\big)^2\vert\widehat{\mu^n}(k)-\widehat{\rho}(k)\vert^2\Big)^{\frac{1}{2}}.
\end{equation}
For the second \rhs of \eqref{EstiPeyre:Eq1}, we introduce $w\in \dot{{\rm H}}^1$ satisfying 
$$-\Delta w=\rho_{t+\frac{1}{n}}-\rho_{\frac{1}{n}},$$
so that from an integration by parts, Cauchy-Schwarz' inequality and the combination of \eqref{EllipticityAn} and \eqref{Boundtestfunctionf}, we obtain
$$\int_{\mathcal{M}}f(\rho_{t+\frac{1}{n}}-\rho_{\frac{1}{n}})\,\dd\m=\int_{\mathcal{M}}\nabla f\cdot\nabla w\,\dd\m\lesssim \Big(\int_{\mathcal{M}}\vert\nabla w\vert^2\,\dd\m\Big)^{\frac{1}{2}}.$$
Using then the explicit formula $w=-\int_{\frac{1}{n}}^{t+\frac{1}{n}}\rho_{\tau}\,\dd\tau$ together with $\vert \int_{\frac{1}{n}}^{t+\frac{1}{n}}\rho_{\tau}\,\dd\tau\vert\stackrel{\eqref{Ellipticity}}{\leq} \Lambda t$, we get 
$$\int_{\mathcal{M}}\vert\nabla w\vert^2\,\dd\m=\int (\rho_{t+\frac{1}{n}}-\rho_{\frac{1}{n}})w\,\dd\m\lesssim t\|\rho_{t+\frac{1}{n}}-\rho_{\frac{1}{n}}\|_{\LL^1},$$
so that 
\begin{equation}\label{EstiPeyre:Eq3}
\Big\vert\int_{\mathcal{M}}f(\rho_{t+\frac{1}{n}}-\rho_{\frac{1}{n}})\,\dd\m\Big\vert\lesssim \sqrt{t}\|\rho_{t+\frac{1}{n}}-\rho_{\frac{1}{n}}\|^{\frac{1}{2}}_{\LL^1}.
\end{equation}
The combination of \eqref{EstiPeyre:Eq1}, \eqref{EstiPeyre:Eq2}, \eqref{EstiPeyre:Eq3} and \eqref{EstiPeyre:Eq4} leads to \eqref{SmoothingIneg}.

\medskip

{\sc Step 2. Proof of \eqref{ContractivityEstiExpectation}. }According to \eqref{SmoothingIneg}, it remains to show that 
\begin{equation}\label{ContractExpectation}
\mathbb{E}\bigg[\sum_{k\geq 1} \frac{e^{-\frac{2}{n}\lambda_k}}{\lambda_k}\big(e^{-t\lambda_k}-1\big)^2\vert\widehat{\mu^n}(k)-\widehat{\rho}(k)\vert^2\bigg]\lesssim \frac{\log\log(n)}{n}.
\end{equation}
Writing 
\begin{equation}\label{DecomposeFourier}
\widehat{\mu^n}(k)-\widehat{\rho}(k)=\frac{1}{n}\sum_{\ell=1}^n (\phi_k(X_\ell)-\mathbb{E}[\phi_k(X_\ell)]),
\end{equation}
we first expand the square in form of
\begin{equation}\label{SquareExpended}
\begin{aligned}
&\sum_{k\geq 1} \frac{e^{-\frac{2}{n}\lambda_k}}{\lambda_k}\big(e^{-t\lambda_k}-1\big)^2\vert\widehat{\mu^n}(k)-\widehat{\rho}(k)\vert^2\\
=&\frac{1}{n^2}\sum_{\ell=1}^n\sum_{k\geq 1} \frac{e^{-\frac{2}{n}\lambda_k}}{\lambda_k}\big(e^{-t\lambda_k}-1\big)^2\big(\phi_k(X_{\ell})-\mathbb{E}[\phi_k(X_\ell)]\big)^2\\
&+\frac{2}{n^2}\sum_{1\leq \ell<\ell'\leq n}\sum_{k\geq 1} \frac{e^{-\frac{2}{n}\lambda_k}}{\lambda_k}\big(e^{-t\lambda_k}-1\big)^2\big(\phi_k(X_{\ell})-\mathbb{E}[\phi_k(X_\ell)]\big)\big(\phi_k(X_{\ell'})-\mathbb{E}[\phi_k(X_{\ell'})]\big).
\end{aligned}
\end{equation}

For the first \rhs term of \eqref{SquareExpended}, we use the normalisation $\|\phi_k\|_{\LL^2}=1$ together with \eqref{Ellipticity} to the effect of 
\begin{equation}\label{BoundsMomentVectors}
\mathbb{E}\big[\vert \phi_k(X_1)-\mathbb{E}[\phi_k(X_1)]\vert^2\big]\leq \Lambda,
\end{equation}
and get
\begin{equation}\label{ExpendedSqaure:Est1}
\mathbb{E}\bigg[\sum_{\ell=1}^n\sum_{k\geq 1} \frac{e^{-\frac{2}{n}\lambda_k}}{\lambda_k}\big(e^{-t\lambda_k}-1\big)^2(\phi_k(X_{\ell})-\mathbb{E}[\phi_k(X_\ell)])^2\bigg]\leq \Lambda n\sum_{k\geq 1} \frac{e^{-\frac{2}{n}\lambda_k}}{\lambda_k}\big(e^{-t\lambda_k}-1\big)^2.
\end{equation}
For the second \rhs term of \eqref{SquareExpended}, we use the definition of the $\beta$-mixing coefficient \eqref{BetaMixing} together with the assumption \eqref{DecatBetaMixing} in form of
\begin{equation}\label{ExpendedSquare:Est2}
\begin{aligned}
&\mathbb{E}\bigg[\sum_{1\leq \ell<\ell'\leq n}\sum_{k\geq 1} \frac{e^{-\frac{2}{n}\lambda_k}}{\lambda_k}\big(e^{-t\lambda_k}-1\big)^2\big(\phi_k(X_{\ell})-\mathbb{E}[\phi_k(X_\ell)]\big)\big(\phi_k(X_{\ell'})-\mathbb{E}[\phi_k(X_{\ell'})]\big)\bigg]\\
&\leq \sum_{1\leq \ell<\ell'\leq n}\beta_{\ell\ell'}\sup_{x,y\in\mathcal{M}}\bigg\vert\sum_{k\geq 1} \frac{e^{-\frac{2}{n}\lambda_k}}{\lambda_k}\big(e^{-t\lambda_k}-1\big)^2\phi_k(x)\phi_k(y)\bigg\vert\\
&\stackrel{\eqref{DecatBetaMixing}}{\lesssim}n\sup_{x,y\in\mathcal{M}}\bigg\vert\sum_{k\geq 1} \frac{e^{-\frac{2}{n}\lambda_k}}{\lambda_k}\big(e^{-t\lambda_k}-1\big)^2\phi_k(x)\phi_k(y)\bigg\vert.
\end{aligned}
\end{equation}
The combination of \eqref{SquareExpended}, \eqref{ExpendedSqaure:Est1} and \eqref{ExpendedSquare:Est2} yields
\begin{align*}
&\mathbb{E}\bigg[\sum_{k\geq 1} \frac{e^{-\frac{2}{n}\lambda_k}}{\lambda_k}\big(e^{-t\lambda_k}-1\big)^2\vert\widehat{\mu^n}(k)-\widehat{\rho}(k)\vert^2\bigg]\\
&\lesssim \frac{1}{n}\bigg(\sum_{k\geq 1} \frac{e^{-\frac{2}{n}\lambda_k}}{\lambda_k}\big(e^{-t\lambda_k}-1\big)^2+\sup_{x,y\in\mathcal{M}}\bigg\vert\sum_{k\geq 1} \frac{e^{-\frac{2}{n}\lambda_k}}{\lambda_k}\big(e^{-t\lambda_k}-1\big)^2\phi_k(x)\phi_k(y)\bigg\vert\bigg).
\end{align*}
It remains to show that 
\begin{equation}\label{ControlExpectationLast}
\sum_{k\geq 1} \frac{e^{-\frac{2}{n}\lambda_k}}{\lambda_k}\big(e^{-t\lambda_k}-1\big)^2+\sup_{x,y\in\mathcal{M}}\bigg\vert\sum_{k\geq 1} \frac{e^{-\frac{2}{n}\lambda_k}}{\lambda_k}\big(e^{-t\lambda_k}-1\big)^2\phi_k(x)\phi_k(y)\bigg\vert\lesssim \log\log(n).
\end{equation}
We only treat the second \lhs term of \eqref{ControlExpectationLast}, the first term is controlled the same way. For any $x,y\in\mathcal{M}$, we expand
\begin{align*}
\lefteqn{\sum_{k\geq 1}\frac{e^{-\frac{2}{n}\lambda_k}}{\lambda_k}\big(e^{-t\lambda_k}-1\big)^2\phi_k(x)\phi_k(y)}\\ &\quad=\sum_{k\geq 1}\frac{1}{\lambda_k}e^{-2(\frac{1}{n}+t)\lambda_k}\phi_k(x)\phi_k(y)-2\sum_{k\geq 1}\frac{1}{\lambda_k}e^{-(\frac{2}{n}+t)\lambda_k}\phi_k(x)\phi_k(y)+\sum_{k\geq 1}\frac{1}{\lambda_k}e^{-\frac{2}{n}\lambda_k}\phi_k(x)\phi_k(y).
\end{align*}
We then disintegrate using the spectral decomposition of the heat kernel \eqref{SpectralDecompoHeatKernel} in form of  
$$\sum_{k\geq 1}\frac{1}{\lambda_k}e^{-s\lambda_k}\phi_k(x)\phi_k(y)=\int_{s}^1\dd s\, p_s(x,y)-\sum_{k\geq1}\frac{1}{\lambda_k}e^{-\lambda_k}\phi_k(x)\phi_k(y)\quad\text{for any $s>0$, }$$
so that we obtain
\begin{align*}
\sup_{x,y\in\mathcal{M}}\bigg\vert\sum_{k\geq 1}\frac{e^{-\frac{2}{n}\lambda_k}}{\lambda_k}\big(e^{-t\lambda_k}-1\big)^2\phi_k(x)\phi_k(y)\bigg\vert&=\sup_{x,y\in\mathcal{M}}\bigg\vert\int_{2(\frac{1}{n}+t)}^{\frac{2}{n}+t}p_s(x,y)\,\dd s+\int_{\frac{2}{n}}^{\frac{2}{n}+t}p_s(x,y)\,\dd s\bigg\vert\\
&\lesssim \int_{2(\frac{1}{n}+t)}^{\frac{2}{n}+t}s^{-1}\,\dd s+\int_{\frac{2}{n}}^{\frac{2}{n}+t}s^{-1}\,\dd s\\
&\lesssim \log\log(n).
\end{align*}

\medskip

{\sc Step 3. Proof of \eqref{ContractivityEstiExpMoment}. }It is a consequence of the following fluctuation estimates
\begin{equation}\label{ContractivityCOM:Eq1}
\mathbb{E}\big[\vert\widehat{\mu^n}(k)-\widehat{\rho}(k)\vert^{2p}\big]^{\frac{1}{p}}\lesssim p^2\,\frac{1}{n}\bigg(\log^{\frac{1}{\eta}}(n)\lambda^{\frac{1}{\log(n)}}_k+\frac{\lambda_k(1+\log^2(n))}{n}\bigg)\quad\text{for any $p<\infty$},
\end{equation}
together with Lemma \ref{momentexp}. Indeed, applying Minkowski's inequality followed by \eqref{ContractivityCOM:Eq1} and $\lambda^{\frac{1}{\log(n)}}_k e^{-\frac{2}{n}\lambda_k}\lesssim e^{-\frac{1}{n}\lambda_k}$ yield
\begin{align*}
&\mathbb{E}\bigg[\Big(\sum_{k\geq 1}\frac{e^{-\frac{2}{n}\lambda_k}}{\lambda_k}\big(e^{-t\lambda_k}-1\big)^2\vert\widehat{\mu^n}(k)-\widehat{\rho}(k)\vert^2\Big)^p\bigg]^{\frac{1}{p}}\\
&\lesssim p^2\frac{1}{n}\sum_{k\geq 1}\frac{e^{-\frac{1}{n}\lambda_k}}{\lambda_k}\big(e^{-t\lambda_k}-1\big)^2\bigg(\log^{\frac{1}{\eta}}(n)+\frac{\lambda_k(1+\log^2(n))}{n}\bigg),
\end{align*}
and \eqref{ContractivityEstiExpMoment} follows from
\begin{equation}\label{ContractivityConclusion:Eq2}
\sum_{k\geq 1}\frac{e^{-\frac{1}{n}\lambda_k}}{\lambda_k}\big(e^{-t\lambda_k}-1\big)^2\lesssim \log\log(n)\quad\text{and}\quad\sum_{k\geq 1}e^{-\frac{1}{n}\lambda_k}\big(e^{-t\lambda_k}-1\big)^2\lesssim t^{-1},
\end{equation}
which is obtained the same way as in \eqref{ControlExpectationLast} using additionally the trace formula \eqref{eq:traceformula}.

\medskip

We now prove \eqref{ContractivityCOM:Eq1}. It follows from the estimate on the probability tails
\begin{equation}\label{ContractivityCOM:Eq2}
\mathbb{P}\big(\vert\widehat{\mu^n}(k)-\widehat{\rho}(k)\vert>\lambda\big)\lesssim \exp\bigg(-\frac{1}{C}\frac{n^2\lambda^2}{n\log^{\frac{1}{\eta}}(n)\lambda^{\frac{1}{\log(n)}}_k+\lambda_k+n\lambda\lambda^{\frac{1}{2}}_k\log^2(n)}\bigg)\quad\text{for any $\lambda>0$},
\end{equation}
for some $C>0$, together with a simple application of the layer-cake representation. 

\medskip

To see \eqref{ContractivityCOM:Eq2}, we use \eqref{DecomposeFourier} together with Proposition \ref{BersteinCorrelated} to obtain
$$\mathbb{P}\big(\vert\widehat{\mu^n}(k)-\widehat{\rho}(k)\vert>\lambda\big)\lesssim \exp\Big(-\frac{n^2\lambda^2}{nv^2+\|\phi_k\|^2_{\LL^{\infty}}+n\lambda\|\phi_k\|_{\LL^{\infty}}\log^2(n)}\Big),$$
with 
\begin{equation}\label{vsquaredContra}
v^2:=\mathbb{E}[\vert\phi_k(X_1)-\mathbb{E}[\phi_k(X_1)]\vert^2]+2\sum_{i<j}\big\vert\mathbb{E}\big[(\phi_k(X_i)-\mathbb{E}[\phi_k(X_i)])(\phi_k(X_j)-\mathbb{E}[\phi_k(X_j)])\big]\big\vert.
\end{equation}
The estimate \eqref{ContractivityCOM:Eq2} is then a consequence of 
\begin{equation}\label{ContractivityCOM:Eq4}
 \|\phi_k\|_{\LL^{\infty}}\lesssim \lambda^{\frac{1}{2}}_k\quad\text{and}\quad v^2\lesssim \log^{\frac{1}{\eta}}(n)\lambda^{\frac{1}{\log(n)}}_k.
\end{equation}
The first item of \eqref{ContractivityCOM:Eq4} has been treated in \eqref{eq:eigvaleighfunc}. 
%
%
%
%
For the second item of \eqref{ContractivityCOM:Eq4}, we use \eqref{BoundsMomentVectors} and, combined with \eqref{DecayAlphaMixing}, we obtain
\begin{align*}
\sum_{i<j}\big\vert\mathbb{E}\big[(\phi_k(X_i)-\mathbb{E}[\phi_k(X_i)])(\phi_k(X_j)-\mathbb{E}[\phi_k(X_j)])\big]\big\vert\lesssim& \|\phi_k\|^{\frac{1}{\log(n)}}_{\LL^\infty}\sum_{\ell\geq 0}\exp\big(-b\tfrac{\ell^{\eta}}{\log(n)}\big)\\
\lesssim& \log^{\frac{1}{\eta}}(n)\lambda^{\frac{1}{\log(n)}}_k.
\end{align*}
 \end{proof}
 %
%
%
%
%
%
%
%

\subsection{Proof of Theorem \ref{MainResultMatchingClouds}: Approximation of the transport plan}\label{ProofTransport}
We only give the arguments for \eqref{ApproximationTransportPlan}, \eqref{ApproximationTransportPlanBis} is proved the same way using the corresponding results \eqref{Eq22}, \eqref{ContractivityEstiExpMoment} and \eqref{eq:matchingcostHighMoments} in the case $\eta>2$ (where some additional comments are given if necessary along the proof). We split the proof into four steps. In the first step, we display some preliminary estimates useful all along the proof. In the second step, we deal with the approximation error that occurs in the process of regularizing $\rho$ into $\rho_\delta$. In the third step, we estimate the $W_2$-distance for the regularized quantity using the quantitative stability result in \cite[Theorem 3.2]{ambrosio2019optimal}, splitting the estimates in small pieces that we control in the fourth step. We finally comment on the proof of Remark \ref{FlatGeometry} and Theorem \ref{th1}, which are obtained with similar techniques.

\medskip

{\sc Step 1. Preliminary estimates. }

\medskip

\textbf{Heat kernel regularization. }The assumption $\rho\in \mathrm{H}^\varepsilon$ provides 
\begin{equation}\label{ApproximationRho}
\|\rho_s-\rho\|_{\LL^2}\lesssim \|\rho\|_{\mathrm{H}^\varepsilon} \,s^{\varepsilon}\quad\text{for any $s>0$}.
\end{equation}
Indeed, using the definition of the heat-kernel, Minkowski's inequality and the spectral decomposition \eqref{SpectralDecompoHeatKernel}, we have
\begin{align*}
\|\rho_{s}-\rho\|^2_{\LL^2}&=\int_{\mathcal{M}}\dd \m(x)\bigg\vert\int_{\mathcal{M}}\rho(y)\dd\m(y)\int_{0}^{s}\dd\tau\,\partial_{\tau}p_{\tau}(x,y)\bigg\vert^2\\
&\leq\bigg(\int_{0}^{s}\dd\tau\,\bigg(\int_{\mathcal{M}}\dd\m(x)\bigg\vert\sum_{k\geq 1}\widehat{\rho}(k)\, \lambda_k e^{-\tau\lambda_k}\phi_k(x)\bigg\vert^2\bigg)^{\frac{1}{2}}\bigg)^2.
\end{align*}
Noticing that 
$$\int_{\mathcal{M}}\dd\m(x)\bigg\vert\sum_{k\geq 1}\widehat{\rho}(k)\, \lambda_k e^{-\tau\lambda_k}\phi_k(x)\bigg\vert^2=\sum_{k\geq 1}\lambda^{2\varepsilon}_k\vert\widehat{\rho}(k)\vert^2\lambda^{2(1-\varepsilon)}_k e^{-2\tau\lambda_k}\lesssim \|\rho\|^2_{{\rm H}^{\varepsilon}}\tau^{2(\varepsilon-1)},$$
we get that
$$\|\rho_{s}-\rho\|^2_{\LL^2}\lesssim \|\rho\|^2_{{\rm H}^{\varepsilon}}\Big(\int_{0}^{s}\tau^{\varepsilon-1}\,\dd\tau\Big)^2\lesssim_{\varepsilon} \|\rho\|^2_{{\rm H}^{\varepsilon}}s^{2\varepsilon}.$$
\textbf{$\LL^\infty$-estimates. }Let $\kappa_1>0$ and $\upsilon>\max\{(\kappa+2)\kappa_1,\frac{1}{\eta}\}$, where $\kappa$ is given in Proposition \ref{Fluctuation}. For the given choices 
\begin{equation}\label{Choicestdelta}
t:=\frac{\log^{\kappa_2}(n)}{n}\quad\text{and}\quad\delta:=\frac{1}{\log^{\kappa_1}(n)},
\end{equation}
provided in Proposition \ref{Fluctuation}, we define $h^{n,t}_\delta\in \dot{\mathrm{H}}^1$ the weak solution of 
\begin{equation}\label{Equationhdelta}
\nabla\cdot\rho_\delta \nabla h^{n,t}_\delta=\mu^{n,t}-\nu^{m,t}.
\end{equation}
Note that by linearity, one can decompose $h^{n,t}_\delta=h^{(2)n,t}_\delta-h^{(1)n,t}_\delta$ with 
\begin{equation}\label{Decompohnt}
-\nabla\cdot\rho_\delta \nabla h^{(1)n,t}_\delta=\mu^{n,t}-\rho_t\quad\text{and}\quad-\nabla\cdot\rho_\delta \nabla h^{(2)n,t}_\delta=\nu^{m,t}-\rho_t,
\end{equation}
so that, considering $u^{n,t}_\delta$ as in \eqref{Eq3Bis} and likewise $v^{n,t}_\delta$ with $\mu^{n,t}$ replaced by $\nu^{m,t}$ and defining 
\begin{equation}\label{eq:An}
\mathcal{A}_n:=\Big\{\|\mu^{n,t}-\rho_t\|_{\LL^{\infty}}+\|\nu^{m,t}-\rho_t\|_{\LL^{\infty}}\leq \tfrac{1}{\log^{\upsilon}(n)}\Big\}
\end{equation}
as well as
\begin{equation}\label{eq:Bdeltan}
\mathcal{B}_{\delta,n}:=\Big\{\big\|\big(\nabla (u^{n,t}_\delta, v^{n,t}_\delta),\nabla^2 (u^{n,t}_\delta, v^{n,t}_\delta)\big)\big\|_{\LL^\infty}\leq \tfrac{1}{\log^{\upsilon}(n)}\Big\},
\end{equation}
we deduce from Proposition \ref{Fluctuation} and $m=m(n)\underset{n\uparrow\infty}{\sim} qn$ that
\begin{equation}\label{Boundhnt}
\big\|(\nabla h^{n,t}_\delta,\nabla^2 h^{n,t}_\delta)\big\|_{\LL^{\infty}}\lesssim \frac{1}{\log^{\upsilon-(\kappa+2)\kappa_1}(n)}\quad\text{in $\mathcal{A}_n\cap \mathcal{B}_{\delta,n}$.}
\end{equation}
\textbf{$\LL^q$-estimates. }A similar decomposition as in \eqref{Decompohnt} of \eqref{eq:f^nmt} together with \eqref{Eq22Bis} and the choice of $t$ in \eqref{Choicestdelta} yields
\begin{equation}\label{LqEstihnt}
\Big(\int_{\mathcal{M}}\vert \nabla h^{n,t}\vert^{\bar{q}}\,\dd\m\Big)^{\frac{2}{\bar{q}}}\leq \mathcal{C}_n\frac{\log(n)+\log^{\frac{1}{\eta}}(n)}{n},
\end{equation}
where $\mathcal{C}_n$ denotes, all along the proof, a random variable which satisfies \eqref{MomentBoundCn} and may change from line to line. 

\medskip

\textbf{$\LL^2$-regularization error. }Note that from \eqref{Equationhdelta} and \eqref{eq:f^nmt}
\begin{equation}\label{DecompositionDeltaProof}
-\nabla\cdot\rho_\delta\nabla(h^{n,t}_\delta-h^{n,t})=\nabla\cdot (\rho_\delta-\rho)\nabla h^{n,t},
\end{equation}
so that from an energy estimate, Hölder's inequality, \eqref{LqEstihnt} and \eqref{Ellipticity}, we obtain 
\begin{equation}\label{ErrorApproxL2}
\int_{\mathcal{M}}\vert\nabla (h^{n,t}_\delta-h^{n,t})\vert^2\,\dd\m\leq \mathcal{C}_n\|\rho_{\delta}-\rho\|^2_{\LL^{2(\frac{\bar{q}}{2})'}}\frac{\log(n)+\log^{\frac{1}{\eta}}(n)}{n},
\end{equation}
where $\bar{q}$ denotes the Meyers' exponent of the operator $-\nabla\cdot \rho\nabla$, see Theorem \ref{Meyers}.

\medskip

Finally, since $\inf_\pi W^2_2(\pi,\gamma^{n,t})\leq (\text{diam}(\mathcal{M}))^2$ and \eqref{Eq13} holds, we can restrict our analysis in $\mathcal{A}_n\cap\mathcal{B}_{\delta,n}$ that we do for the rest of the proof.

\medskip
{\sc Step 2. Regularization error.}
We show that \eqref{ErrorApproxL2} survives when measuring the $W_2$-distance, namely
\begin{equation}\label{eq:thm1explipW2}
W^2_2(\gamma^{n,t}_\delta,\gamma^{n,t})\leq \mathcal{C}_n \|\rho_\delta-\rho\|^2_{\LL^{2(\frac{\bar{q}}{2})'}}\frac{\log(n)+\log^{\frac{1}{\eta}}(n)}{n}\quad\text{with }\gamma^{n,t}_{\delta}:=\big(\mathrm{Id},\exp(\nabla h^{n,t}_\delta)\big)_\#\mu^{n,t}.
\end{equation}
\medskip

Using the coupling $\big((\mathrm{Id},\exp(\nabla h^{n,t}_\delta)),(\mathrm{Id},\exp(\nabla h^{n,t})\big)_\#\mu^{n,t}$ as a competitor in \eqref{Wasserstein} and the fact that $\|\mu^{n,t}\|_{\LL^{\infty}}\lesssim 1$ in $\mathcal{A}_n$, we have
\begin{equation}\label{eq:thm1explip}
W^2_2(\gamma^{n,t}_\delta,\gamma^{n,t})\leq \int_{\mathcal{M}}\mu^{n,t}\dd^2 \big(\exp(\nabla h^{n,t}_\delta), \exp(\nabla h^{n,t})\big)\,\dd\m\lesssim \int_{\mathcal{M}}\dd^2 \big(\exp(\nabla h^{n,t}_\delta), \exp(\nabla h^{n,t})\big)\,\dd\m.
\end{equation}
We then claim that 
\begin{equation}\label{eq:thm1explip2}
\begin{aligned}
\int_\mathcal{M} \dd^2 \big(\exp(\nabla h^{n,t}_\delta), \exp(\nabla h^{n,t})\big)\,\dd\m \lesssim&\,\| \nabla h^{n,t}_\delta - \nabla h^{n,t} \|_{\LL^2}^2\\
&+ \|\rho_\delta-\rho\|^2_{\LL^{(\frac{\bar{q}}{2})'}}\frac{\log(n)}{n}\frac{\| \nabla h^{n,t}_\delta - \nabla h^{n,t}\|_{\LL^2}^2}{\mathbb{E}[ \| \nabla h^{n,t}_\delta - \nabla h^{n,t}\|^{2}_{\LL^2}]}.
\end{aligned}
\end{equation}
which, combined with \eqref{eq:thm1explip} and \eqref{ErrorApproxL2} yields \eqref{eq:thm1explipW2}.

\medskip

We now justify \eqref{eq:thm1explip2}. The difficulty arises from the fact that $\exp$ is not globally Lipschitz. To overcome this, we define
\begin{align*}
\mathrm{E}_n:=\Big\{\vert\nabla h^{n,t}_\delta-\nabla h^{n,t}\vert &\leq C^{-1}_n\mathbb{E}\big[\|\nabla h^{n,t}_\delta-\nabla h^{n,t}\|^{2}_{\LL^{2}}\big]^{\frac{1}{2}}\Big\}\\
& \text{with }C_n:=\varsigma^{-1}\|\rho_{\delta}-\rho\|_{\LL^{(\frac{\bar{q}}{2})'}}\sqrt{\frac{\log(n)+\log^{\frac{1}{\eta}}(n)}{n}},
\end{align*}
for a given $\varsigma$ fixed later, and we split
\begin{align*}
\lefteqn{\int_\mathcal{M} \dd^2 \big(\exp(\nabla h^{n,t}_\delta), \exp(\nabla h^{n,t})\big)\,\dd \m}\\
&= \int_\mathcal{M} \mathds{1}_{\mathrm{E}_n}\dd^2 \big(\exp(\nabla h^{n,t}_\delta), \exp(\nabla h^{n,t})\big)\,\dd \m+\int_\mathcal{M} \mathds{1}_{\mathrm{E}^c_n}\dd^2 \big(\exp(\nabla h^{n,t}_\delta), \exp(\nabla h^{n,t})\big)\,\dd \m\\
&\leq\int_\mathcal{M}\mathds{1}_{\mathrm{E}_n}\dd^2 \big(\exp(\nabla h^{n,t}_\delta), \exp(\nabla h^{n,t})\big)\,\dd \m+(\text{diam}(\mathcal{M}))^2 \m(\mathrm{E}^c_n).
\end{align*}
For the first right-hand side integral, note that from the choice of $C_n$ and \eqref{ErrorApproxL2}, we can choose $\varsigma\ll 1$ uniformly in $n$ such that in $\mathrm{E}_n$ the quantity $\vert\nabla h^{n,t}_{\delta}-\nabla h^{n,t}\vert$ can be made arbitrary small. Since $\exp$ is Lipschitz-continuous in a neighborhood of the null vector, we deduce 
$$\int_\mathcal{M}\mathds{1}_{\mathrm{E}_n}\dd^2 \big(\exp(\nabla h^{n,t}_\delta), \exp(\nabla h^{n,t})\big)\,\dd\m\lesssim \|\nabla h^{n,t}_\delta-\nabla h^{n,t}\|^2_{\LL^{2}}.$$
For the second right-hand side term, we simply apply Markov's inequality in form of
$$\m(\mathrm{E}^c_n)\leq C^2_n\frac{\|\nabla h^{n,t}_\delta-\nabla h^{n,t}\|^2_{\LL^{2}}}{\mathbb{E}\big[\|\nabla h^{n,t}_\delta-\nabla h^{n,t}\|^{2}_{\LL^{2}}\big]}.$$
The combination of the two previous estimates gives \eqref{eq:thm1explip2}.

\medskip

To prove \eqref{ApproximationTransportPlanBis}, we need to control arbitrary $p$-moments, according to Lemma \ref{momentexp}. The argument above can be easily adapted in this case by considering 
\begin{align*}
\mathrm{E}_n:=\Big\{\vert\nabla h^{n,t}_\delta-\nabla h^{n,t}\vert &\leq C^{-1}_n\mathbb{E}\big[\|\nabla h^{n,t}_\delta-\nabla h^{n,t}\|^{2p}_{\LL^{2}}\big]^{\frac{1}{2p}}\Big\}\\
& \text{with }C_n:=\varsigma\|\rho_{\delta}-\rho\|_{\LL^{(\frac{\bar{q}}{2})'}}\sqrt{\frac{\log^{\frac{1}{\eta}}(n)\log(n)}{n}}.
\end{align*}
We then follow the same argument, choosing $\varsigma^{-1}=O(\sqrt{p})$.

\medskip

{\sc Step 3. Quantitative stability. }We show that 
\begin{equation}\label{QuantitativeEstiPlan}
\begin{aligned}
\inf_\pi W^2_2(\pi,\gamma^{n,t}_\delta)\lesssim&\, W^2_2\big(\nu^{m,t},\exp(\nabla h^{n,t}_\delta)_\#\mu^{n,t}\big)+W_2\big(\nu^{m,t},\exp(\nabla h^{n,t}_\delta)_\#\mu^{n,t}\big)W_2(\mu^{n},\nu^{m})\\
&+W^2_2(\nu^{m,t},\nu^{m})+W^2_2(\mu^{n,t},\mu^{n})+\big(W_2(\nu^{m,t},\nu^m)+W_2(\mu^{n,t},\mu^n)\big)W_2(\mu^n,\nu^m),
\end{aligned}
\end{equation}
where we recall that $\gamma^{n,t}_\delta$ is defined in \eqref{eq:thm1explipW2}.

\medskip

Let $\pi$ be a coupling between $\mu^n$ and $\nu^m$. We introduce a regularization parameter $s<1$ and, smoothing the measure $\mu^n$ into $\mu^{n,s}:=\text{P}_s\mu^n$, the optimal transport plan $\pi^{n,s}$ from $\mu^{n,s}$ to $\nu^m$ is represented by a transport map $T^{n,s}$, according to McCann’s theorem \cite{McCann2001}, that is
$$\pi^{n,s}=(\text{Id},T^{n,s})_\#\mu^{n,s}.$$
We then apply the triangle inequality in form of
\begin{equation}\label{RemarkW2}
\begin{aligned}
W_2(\pi^{n,s},\gamma^{n,t}_\delta)\leq &\, W_2((\mathrm{Id},\exp(\nabla h^{n,t}_\delta)_\#\mu^{n,s},(\mathrm{Id},\exp(\nabla h^{n,t}_\delta)_\#\mu^{n,t})\\
&+W_2(\pi^{n,s},(\mathrm{Id},\exp(\nabla h^{n,t}_\delta)_\#\mu^{n,s}).
\end{aligned}
\end{equation}
First, using \eqref{Boundhnt}, $\nabla h^{n,t}_\delta$ is Lipschitz-continuous and $\|\nabla h^{n,t}_\delta\|_{\LL^{\infty}}$ can be made as small as possible for $n$ large. Since $\exp$ is Lipschitz-continuous in a neighborhood of the null vector, we learn from Lemma \ref{lem:transportcontraction} that
\begin{equation}\label{EstimateFirstTermRemark}
W^{2}_2\big((\mathrm{Id},\exp(\nabla h^{n,t}_\delta))_\#\mu^{n,s},(\mathrm{Id},\exp(\nabla h^{n,t}_\delta))_\#\mu^{n,t}\big)\lesssim W^{2}_2(\mu^{n,t},\mu^{n,s}).
\end{equation}
Second, we build a competitor for the second right-hand side term of \eqref{RemarkW2}: Defining
$$\Gamma:=\big((\text{Id},T^{n,s}),(\text{Id},\exp(\nabla h^{n,t}))\big)_\#\mu^{n,s},$$
we have
\begin{align*}
W^{2}_2(\pi^{n,s},(\mathrm{Id},\exp(\nabla h^{n,t})_\#\mu^{n,s})&\leq \int_{\mathcal{M}\times\mathcal{M}\times\mathcal{M}\times\mathcal{M}}\delta^2\big((x,z),(y,w)\big)\dd \Gamma\big((x,y),(z,w)\big)\\
&=\int_{\mathcal{M}}\delta^2\big((x,T^{n,s}(x)),(x,\exp(\nabla h^{n,t}(x)))\big)\,\mu^{n,s}(x)\,\dd\m(x)\\
&\stackrel{\eqref{DistanceTransportPlan}}{=}\int_{\mathcal{M}}\dd^2\big(T^{n,s},\exp(\nabla h^{n,t}_{\delta})\big)\mu^{n,s}\,\dd\m.
\end{align*}
Using again \eqref{Boundhnt} we can apply, for large $n$, the quantitative stability result of transport maps, Theorem \ref{StabilityResultMap}, to $\mu_1=\nu^m$, $\mu_2=\exp(\nabla h^{n,t}_\delta)_\#\mu^{n,s}$ and $\nu=\mu^{n,s}$ to the effect of 
\begin{equation*}
\begin{aligned}
\int_{\mathcal{M}}\dd^2 \big(T^{n,s}, \exp(\nabla h^{n,t}_\delta)\big)\mu^{n,s}\,\dd\m \lesssim&\,W_2^2 \big(\nu^m,\exp(\nabla h^{n,t}_\delta)_\#\mu^{n,s}\big)\\
&+W_2\big(\nu^m,\exp(\nabla h^{n,t}_\delta)_\#\mu^{n,s}\big) W_2\big(\mu^{n,s},\nu^m\big),
\end{aligned}
\end{equation*}
which turns into, using the triangle inequality,
\begin{equation}\label{EstimateSecondRemark}
\begin{aligned}
&\int_{\mathcal{M}}\dd^2 \big(T^{n,s}, \exp(\nabla h^{n,t}_\delta)\big)\mu^{n,s}\,\dd\m\\
&\lesssim\,W_2^2 \big(\nu^{m,t},\exp(\nabla h^{n,t}_\delta)_\#\mu^{n,t}\big)\\
&+W_2\big(\nu^{m,t},\exp(\nabla h^{n,t}_\delta)_\#\mu^{n,t}\big) W_2\big(\mu^{n,s},\nu^m\big)\\
&+W^2_2(\nu^{m,t},\nu^m)+W^2_2\big((\mathrm{Id},\exp(\nabla h^{n,t}_\delta)_\#\mu^{n,s},(\mathrm{Id},\exp(\nabla h^{n,t}_\delta)_\#\mu^{n,t}\big)\\
&+\big(W_2(\nu^{m,t},\nu^m)+W_2((\mathrm{Id},\exp(\nabla h^{n,t}_\delta)_\#\mu^{n,s},(\mathrm{Id},\exp(\nabla h^{n,t}_\delta)_\#\mu^{n,t})\big)W_2(\mu^{n,s},\nu^m).
\end{aligned}
\end{equation}
The combination of \eqref{RemarkW2}, \eqref{EstimateFirstTermRemark} and \eqref{EstimateSecondRemark} yields
\begin{equation}\label{EstimateRemarkConclusion}
\begin{aligned}
W^2_2(\pi^{n,s},\gamma^{n,t}_\delta)\lesssim&\, W^2_2\big(\nu^{m,t},\exp(\nabla h^{n,t}_\delta)_\#\mu^{n,t}\big)+W_2\big(\nu^{m,t},\exp(\nabla h^{n,t}_\delta)_\#\mu^{n,t}\big)W_2(\mu^{n,s},\nu^{m})\\
&+W^2_2(\nu^{m,t},\nu^{m})+W^2_2(\mu^{n,t},\mu^{n,s})+\big(W_2(\nu^{m,t},\nu^m)+W_2(\mu^{n,t},\mu^{n,s})\big)W_2(\mu^{n,s},\nu^m).
\end{aligned}
\end{equation}
Since $\mu^{n,s}\underset{s\downarrow 0}{\rightharpoonup}\mu^{n}$, and consequently (up to extracting a subsequence) $\pi^{n,s}\underset{s\downarrow 0}{\rightharpoonup}\pi$, for some optimal transport plan $\pi$, according to the qualitative stability result \cite[Theorem 5.20]{OldNewVillani}, we can pass to the limit as $s\downarrow 0$ in \eqref{EstimateRemarkConclusion} which leads to \eqref{QuantitativeEstiPlan}.

\medskip

{\sc Step 4. Proof of \eqref{ApproximationTransportPlan}. }
We now fix $\kappa_1=\frac{1}{\eta}(\tfrac{\bar q}{2})\frac{1}{2\varepsilon}+1$ such that, applying \eqref{ApproximationRho}, the regularization error \eqref{eq:thm1explipW2} turns into, recalling that $\delta$ is given by \eqref{Choicestdelta}, 
\begin{equation}\label{RegularizedErrorGoodScaling}
W^2_2(\gamma^{n,t}_\delta,\gamma^{n,t})\leq \mathcal{C}_n\delta^{\frac{2\varepsilon}{(\frac{\bar q}{2})'}}\frac{\log(n)+\log^{\frac{1}{\eta}}(n)}{n}\leq \mathcal{C}_n\frac{1}{n}.
\end{equation}
It remains to show that 
\begin{equation}\label{MainThConclu:Eq1}
\begin{aligned}
&\inf_\pi W_2(\pi,\gamma^{n,t}_\delta)\\
&\leq \mathcal{C}_n\frac{\log(n)}{n} \Big(\sqrt{\log^{\kappa_2-1}(n)\|\rho_{t+\frac{1}{n}}-\rho_t\|_{\LL^1}}+\|\rho_\delta-\rho\|^2_{\LL^{2(\frac{\bar{q}}{2})'}}+\|\rho_t-\rho\|^2_{\LL^{2(\frac{\bar{q}}{2})'}}+\tfrac{1}{\log^{\upsilon}(n)}+\sqrt{\tfrac{\log\log n}{\log n}}\Big),
\end{aligned}
\end{equation}
which together with \eqref{ApproximationRho} and \eqref{RegularizedErrorGoodScaling} leads to \eqref{ApproximationTransportPlan}. To show \eqref{MainThConclu:Eq1}, we control each terms of \eqref{QuantitativeEstiPlan} separately. 

\medskip

The three last terms are controlled using the contractivity estimate \eqref{ContractivityEstiExpectation} and \eqref{eq:matchingcost} which gives 
\begin{align*}
&W^2_2(\nu^{m,t},\nu^{m})+W^2_2(\mu^{n,t},\mu^{n})+\big(W_2(\nu^{m,t},\nu^m)+W_2(\mu^{n,t},\mu^n)\big)W_2(\mu^n,\nu^m)\\
&\leq \mathcal{C}_n\frac{\log(n)}{n}\bigg(\sqrt{\frac{\log\log(n)}{\log(n)}}+\sqrt{\log^{\kappa-1}(n)\|\rho_{t+\frac{1}{n}}-\rho_t\|_{\LL^1}}\bigg).
\end{align*}
For the first two terms, we argue that 
\begin{equation}\label{MainResultConclu:Eq2}
W^2_2\big(\nu^{m,t},\exp(\nabla h^{n,t}_\delta)_\#\mu^{n,t}\big)\leq \mathcal{C}_n\big(\|\rho_{\delta}-\rho\|^{2}_{\LL^{2(\frac{\bar q}{2})'}}+\|\rho_{t}-\rho\|^{2}_{\LL^{2(\frac{\bar q}{2})'}}+\tfrac{1}{\log^{\upsilon}(n)}\big)\frac{\log(n)}{n},
\end{equation}
which combined with \eqref{eq:matchingcost} leads to \eqref{MainThConclu:Eq1}.
\medskip

Let us define the curve $\eta: s\in [0,1]\mapsto \eta_s:=(1-s)\mu^{n,t}+s\nu^{m,t}$ and note that from \eqref{Equationhdelta} we have 
$$\frac{\dd}{\dd s}\eta_s+\nabla\cdot \Big(\eta_s\frac{\rho_\delta\nabla h^{n,t}_\delta}{\eta_s}\Big)=0.$$
Applying Benamou-Brenier' theorem \cite{benamou2000computational}, we learn that 
$$\nu^{m,t}=\phi(1,\cdot)_\#\mu^{n,t}\quad\text{with $\phi$ is the flow induced by $s\mapsto \frac{\rho_\delta\nabla h^{n,t}_\delta}{\eta_s}$}.$$
Next, using that 
\begin{align*}
\Big\vert \frac{\rho_\delta\nabla h^{n,t}_\delta}{\eta_s}-\nabla h^{n,t}_\delta\Big\vert&\lesssim \Big(\vert\rho_\delta-\rho\vert+\vert\rho_t-\rho\vert+\vert\mu^{n,t}-\rho_t\vert+\vert\nu^{m,t}-\rho_t\vert\Big)\vert\nabla h^{n,t}_\delta\vert\\
&\lesssim \Big(\vert\rho_\delta-\rho\vert+\vert\rho_t-\rho\vert+\frac{1}{\log^{\upsilon}(n)}\Big)\vert\nabla h^{n,t}_\delta\vert,
\end{align*}
and applying \cite[Proposition A.1]{ambrosio2019finer} together with Hölder's inequality yields
\begin{align}
W^2_2\big(\nu^{m,t},\exp(\nabla h^{n,t}_\delta)_\#\mu^{n,t}\big)&=W^2_2\big(\phi(1,\cdot)_\#\mu^{n,t},\exp(\nabla h^{n,t}_\delta)_\#\mu^{n,t}\big)\nonumber\\
&\lesssim \int_{\mathcal{M}}\Big(\vert\rho_\delta-\rho\vert+\vert\rho_t-\rho\vert+\frac{1}{\log^{\upsilon}(n)}\Big)^2\vert\nabla h^{n,t}_\delta\vert^2\nonumber\\
&\leq \big(\|\rho_{\delta}-\rho\|^{2}_{\LL^{2(\frac{\bar q}{2})'}}+\|\rho_{t}-\rho\|^{2}_{\LL^{2(\frac{\bar q}{2})'}}+\tfrac{1}{\log^{\upsilon}(n)}\big)\Big(\int_{\mathcal{M}}\vert\nabla h^{n,t}_\delta\vert^{\bar{q}}\Big)^{\frac{2}{\bar{q}}}.\label{StabilityFlowProof}
\end{align}
Using Meyers' estimate of Proposition \ref{Meyers} to \eqref{DecompositionDeltaProof} together with \eqref{Ellipticity} and \eqref{LqEstihnt} provides 
$$\Big(\int_{\mathcal{M}}\vert\nabla h^{n,t}_\delta\vert^{\bar{q}}\Big)^{\frac{2}{\bar{q}}}\lesssim \Big(\int_{\mathcal{M}}\vert\nabla h^{n,t}\vert^{\bar{q}}\Big)^{\frac{2}{\bar{q}}}\stackrel{\eqref{LqEstihnt}}{\leq} \mathcal{C}_n\frac{\log(n)+\log^{\frac{1}{\eta}}(n)}{n},$$
which, combined with \eqref{StabilityFlowProof}, yields \eqref{MainResultConclu:Eq2}.

\medskip

We finally point out that, in the case $\eta>2$, we use \eqref{ContractivityEstiExpMoment} and \eqref{eq:matchingcostHighMoments} and the same computations lead to \eqref{ApproximationTransportPlanBis}.

\medskip

{\sc Step 5. Proof of Theorem \ref{th1} and Remark \ref{FlatGeometry}. }The proof of Theorem \ref{th1} follows the same strategy with the main difference that Step 3 is now dropped and Theorem \ref{StabilityResultMap} is directly applied with $\mu_1=\mu^n$, $\nu=\rho$ and $\mu_2=\exp(\nabla f^{n,t}_\delta)_\#\rho\,\dd m$ where $f^{n,t}_\delta$ solves 
$$\nabla\cdot \rho_\delta\nabla f^{n,t}=\mu^{n,t}-\rho_t.$$
The improvement of Remark \ref{FlatGeometry} follows from the improved contractivity estimate \eqref{ContractivityEstiExpMoment}: Under the assumption \ref{FlatenessAss}, we have (keeping the notations as in Proposition \ref{Contractivity})
\begin{equation}\label{InproveContract}
W^2_2(\mu^{n,t},\mu^n)\leq \mathcal{D}_n\frac{\log\log(n)}{n},
\end{equation}
i.e. we do not have the loss $\log^\frac{1}{\eta}(n)$ in \eqref{ContractivityEstiExpMoment}. Inspecting the proof of \eqref{ContractivityEstiExpMoment}, the loss $\log^{\frac{1}{\eta}}(n)$ comes from estimating $v^2$ defined in \eqref{vsquaredContra}. We obtain \eqref{InproveContract} by simply using \eqref{FlatenessAss} and \eqref{AlphaMixing} to upgrade the second item of \eqref{ContractivityCOM:Eq4} into
$$v^2\lesssim 1+\sum_{\ell\geq 0}\exp(-b\ell^\eta)\lesssim 1.$$
\appendix
\section{Probabilistic and PDE tools}\label{app:ProbPDE}

This section is devoted to recall some probabilistic and analytical tools needed in the proofs. 
%
%
%
%
%
%
%
%
We first recall some concentration inequalities for sequences of random variable satisfying Assumption \ref{Assumptions}. Originally proved for i.i.d. samples, see for instance \cite[Theorem 3.6 \& 3.7]{chung2006concentration}, the proofs in the correlated case can be found in \cite[Theorem 1]{merlevede2011bernstein} and \cite[Theorem 2]{merlevede2009bernstein}.
\begin{proposition}\label{BersteinCorrelated}
Let $n\in\mathbb{N}$, $M>0$, $\{X_i\}_{i}$ be a family of centred random variables such that $\sup_{i\geq 1}\vert X_i\vert\leq M$ for which \eqref{DecayAlphaMixing} holds.

\medskip

For any $\lambda>0$, it holds for some constants $(C_i)_{i\in\{1,\cdot,5\}}$ depending on $a$, $b$ :
\begin{itemize}
\item[(i)]If $\eta<1$,
\begin{align*}
\mathbb{P}\bigg(\Big\vert\frac{1}{n}\sum_{i=1}^n X_i\Big\vert>\lambda\bigg)\leq& n\exp\bigg(-\frac{1}{C_1}\Big(\frac{n\lambda}{M}\Big)^{\eta}\bigg)+\exp\bigg(-\frac{1}{C_2}\frac{n^2\lambda^2}{M^2+n v^2}\bigg)\\
&+\exp\bigg(-\frac{1}{C_3}\frac{n\lambda}{M^2}\exp\Big(\frac{1}{C_4}\Big(\frac{n\lambda}{M}\Big)^{\eta(1-\eta)}\log^{-1}(\tfrac{n\lambda}{M})\Big)\bigg),
\end{align*}
with
$$v^2:=\sup_{i\geq 1}\Big(\mathbb{E}[X^2_i]+2\sum_{j> i}\vert\mathbb{E}[X_iX_j]\vert\Big).$$
\item[(ii)]If $\eta=1$,
\begin{align*}
\mathbb{P}\bigg(\Big\vert\frac{1}{n}\sum_{i=1}^n X_i\Big\vert>\lambda\bigg)\leq \exp\bigg(-\frac{1}{C_5}\frac{n^2\lambda^2}{nv^2+M^2+n\lambda M(\log(n))^2}\Bigg).
\end{align*}
\end{itemize}
\end{proposition}
We then recall the link between algebraic moments and exponential moments. The proof is a direct consequence of the Taylor expansion of the exponential function. 
\begin{lemma}\label{momentexp}Let $X$ be a non-negative random variable. The following two statements are equivalent: 
\begin{itemize}
\item[(i)]There exists $C_1>0$ such that 
$$\mathbb{E}\big[\exp(\tfrac{1}{C_1}X)\big]\leq 2.$$
\item[(ii)] There exists $C_2>0$ such that
$$\mathbb{E}[X^p]^{\frac{1}{p}}\leq p\,C_2\quad\text{ for any $p<\infty$}.$$
\end{itemize}
\end{lemma}
We conclude this section by recalling the standard Meyers' estimate for elliptic equations in divergence form, see for instance the original paper \cite{meyers1963p}.
\begin{theorem}[Meyers estimate]\label{Meyers}
Let $a : \mathcal{M}\rightarrow \mathbb{R}^{2\times 2}$ be measurable and uniformly elliptic. Consider $u\in \mathrm{H}^1$ the solution of the Neumann boundary problem
$$
 \left\{
    \begin{array}{ll}
        -\nabla\cdot a\nabla u=\nabla\cdot g & \text{in $\mathcal{M}$,} \\
        a\nabla u\cdot n_{\mathcal{M}}=0 & \text{on $\partial\mathcal{M}$,}
    \end{array}
\right.
$$
for some $g\in \LL^q$ with $q>2$. There exists $2<\bar q<q$ such that 
$$\nabla u\in \LL^{\bar q} \quad\text{and}\quad \|\nabla u\|_{\LL^{\bar q}}\lesssim \|g\|_{\LL^{\bar{q}}}.$$
\end{theorem}
\section{Matching cost for point clouds}\label{sec:matchcost}
This section is devoted to recall the upper bounds on the matching cost, results which can be found in \cite[Theorem 2]{borda2021empirical} under mild $\beta$-mixing conditions. The case of Markov chains have been studied in \cite{Riekert, Fournier2015} where sharp upper bounds are obtained. We include a short proof for convenience. 

\begin{proposition}[Matching cost]\label{prop:matchcost}
Let $\rho$ satisfying \eqref{Ellipticity} and $\{\mu^n\}_n$ be defined in \eqref{eq:empmeas} with point clouds satisfying the Assumption \ref{Assumptions} or in the class of Markov chains satisfying the Assumption \ref{ClassMarkov}. There exists a constant $C>0$ such that 
\begin{equation}\label{eq:matchingcost}
W_2^2(\mu^n, \rho \,\dd \m) \le \mathcal{C}_n \frac{\log(n)}{n}\quad\text{with $\sup_{n\geq 1}\mathbb{E}\big[\tfrac{1}{C}\mathcal{C}_n\big]\leq 1$.}
\end{equation}
Furthermore, if \eqref{DecayAlphaMixing} holds with $\eta\geq 1$ then the assumption \eqref{DecatBetaMixing} can be dropped and the stochastic integrability can be improved up to losing a $\log(n)$ factor, namely 
\begin{equation}\label{eq:matchingcostHighMoments}
W_2^2(\mu^n, \rho \,\dd \m) \le \mathcal{D}_n \frac{\log^{\frac{1}{\eta}}(n)\log(n)}{n}\quad\text{with $\sup_{n\geq 1}\mathbb{E}\big[\exp(\tfrac{1}{C}\mathcal{D}_n)\big]\leq 1$.}
\end{equation}
\end{proposition}
\begin{proof}
Note that the proof of \eqref{eq:matchingcost} can be found in \cite[Theorem 2]{borda2021empirical} when the point cloud satisfies the Assumption \ref{Assumptions}. We first show how \eqref{eq:matchingcost} can be extended to point clouds which are sampled from a Markov chain satisfying the Assumption \ref{ClassMarkov}. Second, we show how the stochastic integrability can be improved to \eqref{eq:matchingcostHighMoments} when \eqref{DecayAlphaMixing} holds with $\eta \ge 1$.  

\medskip
{\sc Step 1. Markov chains case. }Recall that a Markov chain satisfying the Assumption \ref{ClassMarkov} admits an absolutely continuous invariant measure of the form $\mu_\infty= \rho\, \dd \m$ with $\rho$ satisfying \eqref{Ellipticity}, that is $\lambda\leq \rho\leq \Lambda$. Recalling that we denote by $\{\lambda_n,\phi_n\}_n$ the set of eigenvalues and normalized eigenfunctions of the Laplace-Beltrami operator $-\Delta$ on $\mathcal{M}$, we have by definition \eqref{eq:empmeas} of $\mu^n$, for any $k\geq 1$
\begin{equation}\label{ed:decomposefouriermark}
\widehat{\mu^n}(k) - \widehat{\rho}(k) = \frac1n \sum_{\ell=1}^n (\phi_k(X_\ell) - \mathbb{E}[\phi_k(X_\ell)]) + \frac1n \sum_{\ell=1}^n (\mathbb{E}[\phi_k(X_\ell)] - \mu_\infty (\phi_k)),
\end{equation}
where we use interchangeably the notation $\widehat{\rho}(k) = \int \phi_k \rho \,\dd \m = \mu_\infty(\phi_k)$. Using the Berry-Esseen smoothing inequality \cite[Theorem 5]{borda2021empirical} together with \eqref{ed:decomposefouriermark}, we get
\begin{equation}\label{eq:proofmarkstep1-2}
\begin{split}
\mathbb{E}[W_2^2(\mu^n, \mu_\infty)] \lesssim & \frac1n + \sum_{k\ge1}\frac{e^{-\frac1n\lambda_k}}{\lambda_k} \mathbb{E}\bigg[\bigg(\frac1n \sum_{\ell=1}^n (\phi_k(X_\ell) - \mathbb{E}[\phi_k(X_\ell)])\bigg)^2 \bigg] \\
& + \sum_{k\ge1}\frac{e^{-\frac1n\lambda_k}}{\lambda_k} \mathbb{E} \bigg[ \bigg( \frac1n \sum_{\ell=1}^n (\mathbb{E}[\phi_k(X_\ell)] - \mu_\infty (\phi_k)) \bigg)^2\bigg].
\end{split}
\end{equation}
We now estimate the last two terms of \eqref{eq:proofmarkstep1-2} separately and we start with the third one. Using \eqref{eq:lemMark} and \eqref{ContractivityCOM:Eq4}, we have 
\[
|\mathbb{E}[\phi_k(X_\ell)] - \mu_\infty(\phi_k)| \stackrel{\eqref{eq:lemMark}}{\lesssim} \exp(-b\ell^{\eta}) \|\phi_k\|_{\LL^\infty} \stackrel{\eqref{ContractivityCOM:Eq4}}{\lesssim} \exp(-b\ell^{\eta})\lambda_k^\frac12,
\]
Thus, using in addition \eqref{eq:traceformula}, we get
\begin{equation}\label{eq:proofmarkstep1-3}
\sum_{k\ge1}\frac{e^{-\frac1n\lambda_k}}{\lambda_k} \mathbb{E}\bigg[\bigg(\frac1n \sum_{\ell=1}^n (\phi_k(X_\ell) - \mathbb{E}[\phi_k(X_\ell)])\bigg)^2 \bigg] \lesssim \frac1{n^2} \sum_{k\ge1} e^{-\frac1n \lambda_k} \stackrel{\eqref{eq:traceformula}}{=} \frac1{n^2}\int_{\mathcal{M}} p_{\frac1n}(x,x)\,\dd\m(x) \lesssim \frac1n.
\end{equation}
We now turn to the second term of \eqref{eq:proofmarkstep1-2}. Expanding the square provides
\begin{equation}\label{eq:proofmatchcost-2}
\begin{split}
\lefteqn{\sum_{k \ge 1} \frac{e^{-\frac1n \lambda_k}}{\lambda_k}\Big(\frac1n \sum_{\ell=1}^n (\phi_k(X_\ell) - \mathbb{E}[\phi_k(X_\ell)])\Big)^2}\\ 
& = \frac1{n^2}\sum_{k\ge1}\sum_{\ell=1}^n \frac{e^{-\frac1n \lambda_k}}{\lambda_k} |\phi_k(X_\ell) - \mathbb{E}[\phi_k(X_\ell)]|^2 \\
& + \frac2{n^2} \sum_{k\ge1}\sum_{1\le \ell < \ell'\le n} \frac{e^{-\frac1n \lambda_k}}{\lambda_k} (\phi_k(X_\ell) - \mathbb{E}[\phi_k(X_\ell)])(\phi_k(X_{\ell'}) - \mathbb{E}[\phi_k(X_{\ell'})]).
\end{split}
\end{equation}
We now estimate the two terms on the right hand side of \eqref{eq:proofmatchcost-2}. For the first term, an easy induction argument combining \eqref{AbsoluteContinuity} and \eqref{LawOfTheChain} show that for any $n\geq 1$, $\mathbb{P}_{X_n}\ll \m$ with $\lambda\leq\frac{\dd \mathbb{P}_{X_n}}{\dd\m}\leq \Lambda$. Therefore, we have 
$$ \mathbb{E}\big[|\phi_k(X_\ell) - \mathbb{E}[\phi_k(X_\ell)]|^2\big]\leq \Lambda,$$
and we deduce
\begin{equation}\label{eq:proofmatchcost-3}
\frac1{n^2}\sum_{k\ge1}\sum_{\ell=1}^n \frac{e^{-\frac1n \lambda_k}}{\lambda_k} \mathbb{E}\big[|\phi_k(X_\ell) - \mathbb{E}[\phi_k(X_\ell)]|^2\big] \lesssim \frac1n \sum_{k \ge 1}  \frac{e^{-\frac1n \lambda_k}}{\lambda_k}.
\end{equation}
For the second term, we use \eqref{eq:markchainbetamixc} to obtain
\begin{align*}
&\frac2{n^2} \mathbb{E}\bigg[ \sum_{k\ge1}\sum_{1\le \ell < \ell'\le n} \frac{e^{-\frac1n \lambda_k}}{\lambda_k} (\phi_k(X_\ell) - \mathbb{E}[\phi_k(X_\ell)])(\phi_k(X_{\ell'}) - \mathbb{E}[\phi_k(X_{\ell'})])\bigg]\\
& \lesssim \frac1{n^2}\sum_{1\le \ell < \ell' \le n} \beta_{\ell \ell'} \sup_{x,y \in \mathcal{M}} \bigg\lvert \sum_{k\ge1} \frac{e^{-\frac1n \lambda_k}}{\lambda_k} \phi_k(x)\phi_k(y)\bigg\rvert \\
&\lesssim \frac1n \sup_{x,y \in \mathcal{M}} \bigg\lvert \sum_{k\ge1} \frac{e^{-\frac1n \lambda_k}}{\lambda_k} \phi_k(x)\phi_k(y)\bigg\rvert.
\end{align*}
Combining the latter with \eqref{eq:proofmarkstep1-2}, \eqref{eq:proofmarkstep1-3}, \eqref{eq:proofmatchcost-2} and \eqref{eq:proofmatchcost-3} yields 
\begin{equation}\label{eq:proofmarkstep1-6}
\mathbb{E}[W_2^2(\mu^n, \mu_\infty)] \lesssim \frac1n + \frac1n \sum_{k\ge1} \frac{e^{-\frac1n\lambda_k}}{\lambda_k} + \frac1n \sup_{x,y \in \mathcal{M}} \bigg\lvert \sum_{k\ge1} \frac{e^{-\frac1n \lambda_k}}{\lambda_k} \phi_k(x)\phi_k(y)\bigg\rvert.
\end{equation}
We finally conclude similarly as for \eqref{ControlExpectationLast}.
%
\medskip

{\sc Step 2. Higher stochastic integrability. }We now prove \eqref{eq:matchingcostHighMoments}. We argue using the moment estimate \eqref{ContractivityCOM:Eq1} which, together with Minkowski's inequality and $\lambda_k^\frac1{\log(n)} e^{-\frac1n \lambda_k} \lessim e^{-\frac1{2n}\lambda_k}$ implies
\[
\mathbb{E}\bigg[\bigg(\sum_{k\ge1} \frac{e^{-\frac1n \lambda_k}}{\lambda_k}|\widehat{\mu^n}(k)-\widehat{\rho}(k)|^2\bigg)^p\bigg]^\frac1p \lesssim p^2 \frac1n \sum_{k\ge 1} \frac{e^{-\frac1{2n} \lambda_k}}{\lambda_k}\bigg( \log^\frac1\eta(n) + \frac{\lambda_k(1+\log^2(n))}{n}  \bigg).
\]
Finally, combining the latter with the Berry-Esseen smoothing inequality \cite[Theorem 5]{borda2021empirical} and arguing similarly as for \eqref{ControlExpectationLast} yields \eqref{eq:matchingcostHighMoments} thanks to Proposition \ref{momentexp}.
\end{proof}

\section{Proof for the class of Markov chains}\label{ClassMarkovAppendix}
We provide in this Section the arguments for extending Theorem \ref{MainResultMatchingClouds} and Theorem \ref{th1} to the class of Markov chains introduced in Section \ref{Examples}. The proof follows the lines of the proof of Theorem \ref{MainResultMatchingClouds}, where the main difference is that we drop the assumption that the point clouds is identically distributed. That affects the proofs of the main ingredients (we recall that the scaling of the cost has already be proven in Proposition \ref{prop:matchcost}), namely the $\LL^q$ estimates in Proposition \ref{prop1}, the fluctuation estimates in Proposition \ref{Fluctuation} and the contractivity estimates in Proposition \ref{Contractivity}. We show in the following how to adapt the proofs for a given Markov chain $\{X_n\}_{n\geq 1}$ satisfying Assumption \ref{ClassMarkov}. In the following, we recall that $\mu_\infty=\rho\,\dd\m$ denotes the unique invariant measure of the chain. We split the proof into three steps.

\medskip

{\sc Step 1. $\LL^q$ estimates. }We have to understand the extra error term coming from the deviation of $\mathbb{E}[\mu^n]$ from $\mu_\infty$. In view of \eqref{Eq51}, it is 
$$\bigg(\int_{\mathcal{M}}\dd\m\Big(\int_{0}^\infty \dd s\,\big((-s\Delta)^{\frac{1}{2}}\text{P}_{s+t}(\mathbb{E}[\mu^n]-\mu_\infty)\big)^2\Big)^2\bigg)^{\frac{1}{2}}.$$
Using the definition \eqref{eq:empmeas} of $\mu^n$, the convergence to equilibrium \eqref{eq:lemMark} applied to $f=(-s\Delta)^{\frac{1}{2}}p_{s+t}(x,\cdot)$ and the heat-kernel estimates \eqref{Eq28}, we have for any $s\geq 0$ and $x\in\mathcal{M}$
\begin{equation}\label{eq:estLqmark}
\begin{aligned}
(-s\Delta)^\frac12 \text{P}_{s+t}\big(\mathbb{E}[\mu^n] - \rho\big)(x)
& = \frac1n \sum_{\ell=1}^n  (\mathbb{E}[(-s\Delta)^\frac12 p_{s+t}(x,X_\ell)] - \mu_\infty((-s\Delta)^\frac12 p_{s+t}(x,\cdot))\\ & \stackrel{\eqref{eq:lemMark}}{\lesssim}  \frac{\|(-s \Delta)^\frac12 p_{s+t}(x,\cdot)\|_{\LL^\infty}}n\stackrel{\eqref{Eq28}}{\lesssim}\frac{s^\frac12(s+t)^{-\frac32}}n,
\end{aligned}
\end{equation}
so that, recalling $t=\frac{\log^{\kappa}(n)}{n}$,
$$\bigg(\int_{\mathcal{M}}\dd\m\Big(\int_{0}^\infty \dd s\,\big((-s\Delta)^{\frac{1}{2}}\text{P}_{s+t}(\mathbb{E}[\mu^n]-\mu_\infty)\big)^2\Big)^2\bigg)^{\frac{1}{2}}\lesssim \frac{1}{n^2}\int_0^\infty s(s+t)^{-3}\lesssim \frac{1}{n\log^{\kappa}(n)}\ll\frac{\log(n)}{n}.$$
{\sc Step 2. Fluctuation estimates. }Here, the distribution of the Markov chain affects the concentration estimate \eqref{Eq59}. We show that, defining 
\begin{equation}\label{ExtensionMarkovFluctu3}
\bar{u}^{n,t}_\delta:=\int_{0}^\infty e^{-s}\text{P}_s\Big(\tfrac{1}{\rho_\delta}(\mu^{n,t}-\mathbb{E}[\mu^{n,t}])\Big)\,\dd s,
\end{equation}
we have
\begin{equation}\label{ExtensionMarkovFluctu}
\mathbb{P}\Big(\vert\partial^2_{ij} u^{n,t}_\delta(x)\vert\geq \tfrac{1}{2\log^{\nu}(n)}\Big)\leq \mathbb{P}\Big(\vert\partial^2_{ij} \bar{u}^{n,t}_\delta(x)\vert\geq \tfrac{1}{4\log^{\nu}(n)}\Big)\quad\text{for any $x\in\mathcal{M}$,}
\end{equation}
where the r.h.s can be estimated following the lines of the proof of \eqref{Eq59}. As before, we investigate the extra term coming from the deviation of $\mathbb{E}[\mu^n]$ from $\mu_\infty$. The estimate \eqref{ExtensionMarkovFluctu} follows from 
\begin{equation}\label{ExtensionMarkovFluctu2}
\|\partial^2_{ij}(u^{n,t}_\delta-\bar{u}^{n,t}_\delta)\|_{\LL^{\infty}}\ll \frac{1}{\log^{\nu}(n)}.
\end{equation}
We argue as in \eqref{Eq68}, decomposing $u^{n,t}_\delta-\bar{u}^{n,t}_\delta$ into a regular-part and a singular part: for any $x\in\mathcal{M}$
\begin{equation}\label{MarkovExtensionFluctu4}
\begin{aligned}
(u^{n,t}_\delta-\bar{u}^{n,t}_\delta)(x)=&\frac1{\rho_\delta(x)}\int_0^\infty e^{-s}\,\big(\mathbb{E}[\mu^{n,t+s}] - \rho_{t+s}\big)(x)\,\dd s\\
&+ \int_0^\infty e^{-s}\,\text{P}_s \Big(\big(\tfrac1{\rho_\delta}-\tfrac1{\rho_\delta(x)}\big)\big(\mathbb{E}[\mu^{n,t}] - \rho_t\big)\Big)(x)\,\dd s.
\end{aligned}
\end{equation}
To estimate the second r.h.s integral of \eqref{MarkovExtensionFluctu4}, we use \eqref{Eq69}. For the first r.h.s integral, that we denote by $\mathcal{J}$, we use the definition \eqref{eq:empmeas} of $\mu^n$, the convergence to equilibrium \eqref{eq:lemMark} and the heat-kernel bounds \eqref{Eq28} to obtain 
\begin{align*}
\vert\partial^2_{ij}\mathcal{J}(x)\vert=&\bigg\vert\frac{1}{n}\sum_{k=1}^n\partial^2_{ij}\bigg(\frac1{\rho_\delta(\cdot)}\int_0^\infty e^{-s}\,\big(\mathbb{E}[p_{t+s}(\cdot,X_k)] - \mu_\infty(p_{t+s}(x,\cdot))\big)\,\dd s\bigg)(x)\bigg\vert\\
\lesssim& \frac{1}{n}\sum_{k=1}^n\bigg(\|\nabla^2\tfrac{1}{\rho_\delta}\|_{\LL^\infty}\int_{0}^\infty \vert\mathbb{E}[p_{s+t}(\cdot,X_k)]-\mu_\infty(p_{t+s}(x,\cdot))\vert\\
&+\int_{0}^\infty \vert\mathbb{E}[\nabla^2 p_{s+t}(\cdot,X_k)]-\mu_\infty(\nabla^2 p_{t+s}(x,\cdot))\vert\\
&+\|\nabla\tfrac{1}{\rho_\delta}\|_{\LL^\infty}\int_{0}^\infty \vert\mathbb{E}[\nabla p_{s+t}(\cdot,X_k)]-\mu_\infty(\nabla p_{t+s}(x,\cdot))\vert\Big)\\
\stackrel{\eqref{eq:lemMark},\eqref{Eq28}\eqref{Eq25}}{\lesssim}&\frac{1}{n}\bigg(\delta^{-2}\int_{0}^\infty \min\{(s+t)^{-1},(s+t)^{-\frac{3}{2}}\}\,\dd s+\int_{0}^\infty (s+t)^{-2}\,\dd s+\delta^{-1}\int_{0}^{\infty} (s+t)^{-\frac{3}{2}}\, \dd s\bigg)\\
\lesssim & \frac{1}{n}(\delta^{-2}t^{-\frac{1}{2}}+t^{-1}+\delta^{-1}t^{-\frac{1}{2}})\ll \frac{1}{\log^{\nu}(n)}.
\end{align*}
{\sc Step 3. Contractivity estimate. }Here, the law of the Markov chain affects the estimate \eqref{ContractExpectation}. The extra error term coming from the deviation of $\mathbb{E}[\mu^n]$ from $\mu_\infty$ reads
$$\sum_{k\geq 1} \frac{e^{-\frac{2}{n}\lambda_k}}{\lambda_k}\big(e^{-t\lambda_k}-1\big)^2\vert\mathbb{E}[\widehat{\mu^n}(k)]-\mu_\infty(\phi_k)\vert^2.$$
Using the definition \eqref{eq:empmeas} of $\mu^n$ and the convergence to equilibrium \eqref{eq:lemMark} applied with $f=\phi_k$ and the bound on the eigenfunctions \eqref{eq:eigvaleighfunc}, we have for any $k\leq n$
\begin{align*}
\vert\mathbb{E}[\widehat{\mu^n}(k)]-\mu_\infty(\phi_k)\vert=&\frac{1}{n}\bigg\vert\sum_{\ell=1}^n(\mathbb{E}[\phi_k(X_{\ell})]-\mu_\infty(\phi_k))\bigg\vert^2\stackrel{\eqref{eq:lemMark}}{\lesssim} \frac{1}{n}\|\phi_k\|_{\LL^{\infty}}\stackrel{\eqref{eq:eigvaleighfunc}}{\lesssim}\frac{1}{n}\lambda^{\frac{1}{2}}_k,
\end{align*}
so that, using the trace formula \eqref{eq:traceformula} and the heat-kernel estimates \eqref{Eq28}, we deduce 
$$\sum_{k\geq 1} \frac{e^{-\frac{2}{n}\lambda_k}}{\lambda_k}\big(e^{-t\lambda_k}-1\big)^2\vert\mathbb{E}[\widehat{\mu^n}(k)]-\mu_\infty(\phi_k)\vert^2\lesssim \frac{1}{n^2}\sum_{k\geq 1}e^{-\frac{2}{n}\lambda_k}(e^{-t\lambda_k}-1)^2\stackrel{\eqref{eq:traceformula},\eqref{Eq28}}{\lesssim} \frac{1}{n} \ll \frac{\log\log(n)}{n}.$$

\section*{Acknowledgments} The authors warmly thank Lorenzo Dello Schiavo, Antonio Agresti and Martin Huesmann for useful discussions and fruitful comments.

\bibliographystyle{plain}
\bibliography{references}

@article{douc2004practical,
  title={PRACTICAL DRIFT CONDITIONS FOR SUBGEOMETRIC RATES OF CONVERGENCE},
  author={Douc, Randal and Fort, Gersende and Moulines, Eric and Soulier, Philippe},
  journal={The Annals of Applied Probability},
  volume={14},
  number={3},
  pages={1353--1377},
  year={2004},
  publisher={Citeseer}
}

@article{liebscher2005towards,
  title={Towards a unified approach for proving geometric ergodicity and mixing properties of nonlinear autoregressive processes},
  author={Liebscher, E.},
  journal={Journal of Time Series Analysis},
  volume={26},
  number={5},
  pages={669--689},
  year={2005},
  publisher={Wiley Online Library}
}

@article{stroock1998upper,
  title={Upper bounds on derivatives of the logarithm of the heat kernel},
  author={Stroock, D. W. and Turetsky, J.},
  journal={Communications in Analysis and Geometry},
  volume={6},
  number={4},
  pages={669--685},
  year={1998},
  publisher={International Press of Boston}
}

@article{toth2002riemannian,
  title={Riemannian manifolds with uniformly bounded eigenfunctions},
  author={Toth, J. A. and Zelditch, S.},
  journal={Duke Mathematical Journal},
  volume={111},
  number={1},
  pages={97--132},
  year={2002},
  publisher={Duke University Press}
}

@article{meyers1963p,
  title={An $\mathrm{L}^p$-estimate for the gradient of solutions of second order elliptic divergence equations},
  author={Meyers, N. G.},
  journal={Annali della Scuola Normale Superiore di Pisa-Classe di Scienze},
  volume={17},
  number={3},
  pages={189--206},
  year={1963}
}

@article{benamou2000computational,
  title={A computational fluid mechanics solution to the {Monge-Kantorovich} mass transfer problem},
  author={Benamou, J.-D. and Brenier, Y.},
  journal={Numerische Mathematik},
  volume={84},
  number={3},
  pages={375--393},
  year={2000},
  publisher={Springer-Verlag}
}

@article{boutet2002almost,
  title={Almost sure convergence of the minimum bipartite matching functional in {Euclidean} space},
  author={Boutet de Monvel, J. H. and Martin, O. C.},
  journal={Combinatorica},
  volume={22},
  number={4},
  pages={523--530},
  year={2002},
  publisher={Springer}
}

@article{caracciolo2015scaling,
  title={Scaling hypothesis for the {Euclidean} bipartite matching problem. {II}. Correlation functions},
  author={Caracciolo, S. and Sicuro, G.},
  journal={Physical Review E},
  volume={91},
  number={6},
  pages={062125},
  year={2015},
  publisher={APS}
}

@incollection{sicuro2017euclidean,
  title={Euclidean Matching Problems},
  author={Sicuro, G.},
  booktitle={The {Euclidean} Matching Problem},
  pages={59--118},
  year={2017},
  publisher={Springer}
}

@inproceedings{talagrand1992ajtai,
  title={The {Ajtai}-{Koml{\'o}s}-{Tusn{\'a}dy} matching theorem for general measures},
  author={Talagrand, M.},
  booktitle={Probability in Banach Spaces, 8: Proceedings of the Eighth International Conference},
  pages={39--54},
  year={1992},
  organization={Springer}
}

@book{steele1997probability,
  title={Probability theory and combinatorial optimization},
  author={Steele, J. M.},
  year={1997},
  publisher={SIAM}
}

@book{lovasz2009matching,
  title={Matching theory},
  author={Lov{\'a}sz, L. and Plummer, M. D.},
  volume={367},
  year={2009},
  publisher={American Mathematical Soc.}
}

@book{nicolaescu2020lectures,
  title={Lectures on the Geometry of Manifolds},
  author={Nicolaescu, L. I.},
  year={2020},
  publisher={World Scientific}
}

@article{Brenier,
  title={Polar Factorization and Monotone Rearrangement of Vector-Valued Functions},
  author={Brenier, Y.},
  journal={Communications on Pure and Applied Mathematics},
  year={1991},
  volume={44},
  pages={375-417}
}

@article{erbar2015equivalence,
  title={{On the equivalence of the entropic curvature-dimension condition and Bochner’s inequality on metric measure spaces}},
  author={Erbar, M. and Kuwada, K. and Sturm, K.-T.},
  journal={Inventiones mathematicae},
  volume={201},
  number={3},
  pages={993--1071},
  year={2015},
  publisher={Springer}
}

@book{aubin1998some,
  title={Some nonlinear problems in Riemannian geometry},
  author={Aubin, T.},
  year={1998},
  publisher={Springer Science \& Business Media}
}

@article{peyre2018comparison,
  title={Comparison between {W2} distance and $\dot{H}$-1 norm, and localization of {Wasserstein} distance},
  author={Peyre, R.},
  journal={ESAIM: Control, Optimisation and Calculus of Variations},
  volume={24},
  number={4},
  pages={1489--1501},
  year={2018},
  publisher={EDP Sciences}
}

@book{chavel1984eigenvalues,
  title={Eigenvalues in Riemannian geometry},
  author={Chavel, I.},
  year={1984},
  publisher={Academic press}
}

@incollection{stein2016topics,
  title={Topics in Harmonic Analysis Related to the {Littlewood-Paley} Theory.(AM-63), Volume 63},
  author={Stein, E. M.},
  booktitle={Topics in Harmonic Analysis Related to the Littlewood-Paley Theory.(AM-63), Volume 63},
  year={2016},
  publisher={Princeton University Press}
}

@article{borda2021empirical,
  title={{Empirical measures and random walks on compact spaces in the quadratic Wasserstein metric}},
  author={Borda, B.},
  journal={Annales de l'Institut Henri Poincare (B) Probabilites et statistiques},
  volume={59},
  number={4},
  pages={2017--2035},
  year={2023},
  organization={Institut Henri Poincar{\'e}}
}

@article{merlevede2011bernstein,
  title={A {Bernstein} type inequality and moderate deviations for weakly dependent sequences},
  author={Merlev{\`e}de, F. and Peligrad, M. and Rio, E.},
  journal={Probability Theory and Related Fields},
  volume={151},
  number={3},
  pages={435--474},
  year={2011},
  publisher={Springer}
}

@incollection{merlevede2009bernstein,
  title={Bernstein inequality and moderate deviations under strong mixing conditions},
  author={Merlev{\`e}de, F. and Peligrad, M. and Rio, E.},
  booktitle={High dimensional probability V: the Luminy volume},
  pages={273--292},
  year={2009},
  publisher={Institute of Mathematical Statistics}
}

@article{chung2006concentration,
  title={Concentration inequalities and martingale inequalities: a survey},
  author={Chung, F. and Lu, L.},
  journal={Internet mathematics},
  volume={3},
  number={1},
  pages={79--127},
  year={2006},
  publisher={Taylor \& Francis}
}

@book{giaquinta2013introduction,
  title={An introduction to the regularity theory for elliptic systems, harmonic maps and minimal graphs},
  author={Giaquinta, M. and Martinazzi, L.},
  year={2013},
  publisher={Springer Science \& Business Media}
}

@article{benedetto2020euclidean,
  title={Euclidean random matching in 2{D} for non-constant densities},
  author={Benedetto, D. and Caglioti, E.},
  journal={Journal of Statistical Physics},
  volume={181},
  number={3},
  pages={854--869},
  year={2020},
  publisher={Springer}
}

@article{caracciolo2014scaling,
  title={Scaling hypothesis for the {Euclidean} bipartite matching problem},
  author={Caracciolo, S. and Lucibello, C. and Parisi, G. and Sicuro, G.},
  journal={Physical Review E},
  volume={90},
  number={1},
  pages={012118},
  year={2014},
  publisher={APS}
}

@book{Santambrogio,
  title={Optimal Transport for Applied Mathematicians: Calculus of Variations, PDEs, and Modeling},
  author={Santambrogio, F.},
  isbn={9783319208282},
  series={Progress in Nonlinear Differential Equations and Their Applications},
  url={https://books.google.it/books?id=UOHHCgAAQBAJ},
  year={2015},
  publisher={Springer International Publishing}
}

@book {Viltop,
    AUTHOR = {Villani, C.},
     TITLE = {Topics in optimal transportation},
    SERIES = {Graduate Studies in Mathematics},
    VOLUME = {58},
 PUBLISHER = {American Mathematical Society, Providence, RI},
      YEAR = {2003},
     PAGES = {xvi+370},
      ISBN = {0-8218-3312-X},
       DOI = {10.1007/b12016},
       URL = {http://dx.doi.org/10.1007/b12016},
}

@book{gilbarg2015elliptic,
  title={Elliptic partial differential equations of second order},
  author={Gilbarg, D. and Trudinger, N. S.},
  volume={224},
  year={2015},
  publisher={springer}
}

@article{ambrosio2021quadratic,
author = {Ambrosio, L. and Goldman, M. and Trevisan, D.},
title = {{On the quadratic random matching problem in two-dimensional domains}},
volume = {27},
journal = {Electronic Journal of Probability},
number = {none},
publisher = {Institute of Mathematical Statistics and Bernoulli Society},
pages = {1 -- 35},
year = {2022},
doi = {10.1214/22-EJP784},
URL = {https://doi.org/10.1214/22-EJP784}
}

@article{ambrosio2019finer,
  title={Finer estimates on the $2 $-dimensional matching problem},
  author={Ambrosio, L. and Glaudo, F.},
  journal={Journal de l’{\'E}cole polytechnique—Math{\'e}matiques},
  volume={6},
  pages={737--765},
  year={2019}
}

@article{ambrosio2019optimal,
title = {On the optimal map in the $ 2 $-dimensional random matching problem},
journal = {Discrete and Continuous Dynamical Systems},
volume = {39},number = {12},pages = {7291-7308},
year = {2019},
issn = {1078-0947},
doi = {10.3934/dcds.2019304},
url = {/article/id/17beca52-6bdd-4a1a-a9c9-4b2145874254},
author = {Ambrosio, L. and Glaudo, F. and Trevisan, D.}
}

@article{ambrosio2019pde,
  title={A {PDE} approach to a 2-dimensional matching problem},
  author={Ambrosio, L. and Stra, F. and Trevisan, D.},
  journal={Probability Theory and Related Fields},
  volume={173},
  number={1},
  pages={433--477},
  year={2019},
  publisher={Springer}
}

@Article{McCann2001,
author={McCann, R. J.},
title={Polar factorization of maps on {R}iemannian manifolds},
journal={Geometric {\&} Functional Analysis GAFA},
year={2001},
month={Aug},
day={01},
volume={11},
number={3},
pages={589-608},
issn={1420-8970},
doi={10.1007/PL00001679},
url={https://doi.org/10.1007/PL00001679}
}

@book{OldNewVillani,
  title={Optimal Transport: Old and New},
  author={Villani, C.},
  isbn={9783540710509},
  lccn={2008932183},
  series={Grundlehren der mathematischen Wissenschaften},
  url={https://books.google.it/books?id=hV8o5R7\_5tkC},
  year={2008},
  publisher={Springer Berlin Heidelberg}
}

@Article{AKT84,
    Author = {{Ajtai}, M. and {Koml\'os}, J. and {Tusn\'ady}, G.},
    Title = {{On optimal matchings.}},
    FJournal = {{Combinatorica}},
    Journal = {{Combinatorica}},
    ISSN = {0209-9683; 1439-6912/e},
    Volume = {4},
    Pages = {259--264},
    Year = {1984},
    Publisher = {Springer, Berlin/Heidelberg; J\'anos Bolyai Mathematical Society, Budapest},
    DOI = {10.1007/BF02579135},
    MSC2010 = {60D05 60G40},
    Zbl = {0562.60012}
}

@ARTICLE{GHO1,
       author = {{Goldman}, M. and {Huesmann}, M. and {Otto}, F.},
        title = "{A large-scale regularity theory for the Monge-Ampere equation with rough data and application to the optimal matching problem}",
      journal = {arXiv:1808.09250},
         year = "2018",
        month = "Aug",
archivePrefix = {arXiv},
       eprint = {1808.09250},
 primaryClass = {math.AP},
       adsurl = {https://ui.adsabs.harvard.edu/\#abs/2018arXiv180809250G}
}

@article{MezPar,
	author = {M\'ezard, M. and Parisi, G.},
	title = {The {Euclidean} matching problem},
	DOI= "10.1051/jphys:0198800490120201900",
	url= "https://doi.org/10.1051/jphys:0198800490120201900",
	journal = {J. Phys. France},
	year = 1988,
	volume = 49,
	number = 12,
	pages = "2019-2025",
}

@Inbook{Barthe2013,
author="Barthe, F.
and Bordenave, C.",
title="Combinatorial Optimization Over Two Random Point Sets",
bookTitle="S{\'e}minaire de Probabilit{\'e}s XLV",
year="2013",
publisher="Springer International Publishing",
address="Heidelberg",
pages="483--535",
isbn="978-3-319-00321-4",
doi="10.1007/978-3-319-00321-4_19",
url="https://doi.org/10.1007/978-3-319-00321-4_19"
}

@article{Lukas,
 title={Geometric linearisation for optimal transport with strongly p-convex cost}, 
  author={Koch, L.},
  journal={arXiv preprint arXiv:2303.10760},
  year={2023}
}

@inbook{bapat, place={Cambridge}, series={Encyclopedia of Mathematics and its Applications}, title={Doubly stochastic matrices}, DOI={10.1017/CBO9780511529979.003}, booktitle={Nonnegative Matrices and Applications}, publisher={Cambridge University Press}, author={Bapat, R. B. and Raghavan, T. E. S.}, year={1997}, pages={59–114}, collection={Encyclopedia of Mathematics and its Applications}}

@Article{Fournier2015,
author={Fournier, N.
and Guillin, A.},
title={On the rate of convergence in {Wasserstein} distance of the empirical measure},
journal={Probability Theory and Related Fields},
year={2015},
month={Aug},
day={01},
volume={162},
number={3},
pages={707-738},
issn={1432-2064},
doi={10.1007/s00440-014-0583-7},
url={https://doi.org/10.1007/s00440-014-0583-7}
}

@Article{Gangbo1996,
author={Gangbo, W.
and McCann, R. J.},
title={The geometry of optimal transportation},
journal={Acta Mathematica},
year={1996},
month={Sep},
day={01},
volume={177},
number={2},
pages={113-161},
issn={1871-2509},
doi={10.1007/BF02392620},
url={https://doi.org/10.1007/BF02392620}
}

@article{GH,
author = {Goldman, M. and Huesmann, M.},
title = {{A fluctuation result for the displacement in the optimal matching problem}},
volume = {50},
journal = {The Annals of Probability},
number = {4},
publisher = {Institute of Mathematical Statistics},
pages = {1446 -- 1477},
year = {2022},
doi = {10.1214/21-AOP1562},
URL = {https://doi.org/10.1214/21-AOP1562}
}

@book{wangbook,
  title={Analysis for diffusion processes on Riemannian manifolds},
  author={Wang, F.-Y.},
  volume={18},
  year={2014},
  publisher={World Scientific}
}

@Article{alsm03,
author={Alsmeyer, G.},
title={{On the Harris Recurrence of Iterated Random {Lipschitz} Functions and Related Convergence Rate Results}},
journal={Journal of Theoretical Probability},
year={2003},
month={Jan},
day={01},
volume={16},
number={1},
pages={217-247},
issn={1572-9230},
doi={10.1023/A:1022290807360},
url={https://doi.org/10.1023/A:1022290807360}
}

@article{HMTfBm,
title = {{Wasserstein asymptotics for the empirical measure of fractional Brownian motion on a flat torus}},
journal = {Stochastic Processes and their Applications},
volume = {155},
pages = {1-26},
year = {2023},
issn = {0304-4149},
doi = {https://doi.org/10.1016/j.spa.2022.09.008},
url = {https://www.sciencedirect.com/science/article/pii/S0304414922001995},
author = {Huesmann, M. and Mattesini, F. and Trevisan, D.}
}

@article{Jonas,
author = {Jalowy, J.},
title = {{The Wasserstein distance to the circular law}},
volume = {59},
journal = {Annales de l'Institut Henri Poincaré, Probabilités et Statistiques},
number = {4},
publisher = {Institut Henri Poincaré},
pages = {2285 -- 2307},
year = {2023},
doi = {10.1214/22-AIHP1317},
URL = {https://doi.org/10.1214/22-AIHP1317}
}

@article{Wang,
author = {Wang, F.-Y.  and Zhu, J.-X.},
title = {{Limit theorems in Wasserstein distance for empirical measures of diffusion processes on Riemannian manifolds}},
volume = {59},
journal = {Annales de l'Institut Henri Poincaré, Probabilités et Statistiques},
number = {1},
publisher = {Institut Henri Poincaré},
pages = {437 -- 475},
year = {2023},
doi = {10.1214/22-AIHP1251},
URL = {https://doi.org/10.1214/22-AIHP1251}
}

@article{Wang2,
  title={{Convergence in Wasserstein distance for empirical measures of semilinear SPDEs}},
  author={Wang, F.-Y.},
  journal={The Annals of Applied Probability},
  volume={33},
  number={1},
  pages={70--84},
  year={2023},
  publisher={Institute of Mathematical Statistics}
}

@article{Wang3,
  title={Wasserstein convergence for empirical measures of subordinated diffusions on Riemannian manifolds},
  author={Wang, F.-Y. and Wu, B.},
  journal={Potential Analysis},
  volume={59},
  number={3},
  pages={933--954},
  year={2023},
  publisher={Springer}
}

@Article{Ledoux,
author={Ledoux, M.},
title={On Optimal Matching of {G}aussian Samples},
journal={Journal of Mathematical Sciences},
year={2019},
volume={238},
number={4},
pages={495-522},
issn={1573-8795},
doi={10.1007/s10958-019-04253-6},
url={https://doi.org/10.1007/s10958-019-04253-6}
}

@misc{Ledoux2,
  title={On optimal matching of {G}aussian samples {II}},
  author={Ledoux, M.},
  year={2019}
}

@article{Ledoux3,
  title={On optimal matching of Gaussian samples {III}},
  author={Ledoux, Michel and Zhu, Jie-Xiang},
  journal={Probability and Mathematical Statistics},
  volume={41},
  year={2021}
}

@article{CagliotiPieroni,
  title={Random matching in 2D with exponent 2 for densities defined on unbounded sets},
  author={Caglioti, E. and Pieroni, F.},
  journal={arXiv preprint arXiv:2302.02602},
  year={2023}
}

@article{CSMatch,
url = {http://dx.doi.org/10.1561/0400000057},
year = {2013},
volume = {8},
journal = {Foundations and Trends® in Theoretical Computer Science},
title = {{Online Matching and Ad Allocation}},
doi = {10.1561/0400000057},
issn = {1551-305X},
number = {4},
pages = {265-368},
author = {Mehta, A.}
}

@article{GalShap,
 ISSN = {00029890, 19300972},
 URL = {http://www.jstor.org/stable/2312726},
 author = {Gale, D. and Shapley, L. S.},
 journal = {The American Mathematical Monthly},
 number = {1},
 pages = {9--15},
 publisher = {Mathematical Association of America},
 title = {College Admissions and the Stability of Marriage},
 urldate = {2023-02-10},
 volume = {69},
 year = {1962}
}

@article{EcMatch,
 ISSN = {00129682, 14680262},
 URL = {http://www.jstor.org/stable/1913320},
 author = {Crawford, V. P. and Knoer, E. M.},
 journal = {Econometrica},
 number = {2},
 pages = {437--450},
 publisher = {[Wiley, Econometric Society]},
 title = {Job Matching with Heterogeneous Firms and Workers},
 urldate = {2023-02-10},
 volume = {49},
 year = {1981}
}

@article{Riekert,
title = {Convergence rates for empirical measures of {M}arkov chains in dual and {Wasserstein} distances},
journal = {Statistics \& Probability Letters},
volume = {189},
pages = {109605},
year = {2022},
issn = {0167-7152},
doi = {https://doi.org/10.1016/j.spl.2022.109605},
url = {https://www.sciencedirect.com/science/article/pii/S0167715222001468},
author = {Riekert, A.}
}

@article{GoldTrev,
author = {Goldman, M. and Trevisan, D.},
year = {2021},
month = {05},
pages = {121-142},
title = {Convergence of asymptotic costs for random {Euclidean} matching problems},
volume = {2},
journal = {Probability and Mathematical Physics},
doi = {10.2140/pmp.2021.2.121}
}

@article{GoldTrevOpt,
author={Goldman, M.
and Trevisan, D.},
title={Optimal transport methods for combinatorial optimization over two random point sets},
journal={Probability Theory and Related Fields},
year={2023},
month={Nov},
day={07},
issn={1432-2064},
doi={10.1007/s00440-023-01245-1},
url={https://doi.org/10.1007/s00440-023-01245-1}
}

@article{TaoVu,
author = {Tao, T. and Vu, V. and Krishnapur, M.},
title = {{Random matrices: Universality of ESDs and the circular law}},
volume = {38},
journal = {The Annals of Probability},
number = {5},
publisher = {Institute of Mathematical Statistics},
pages = {2023 -- 2065},
year = {2010},
doi = {10.1214/10-AOP534},
URL = {https://doi.org/10.1214/10-AOP534}
}

@phdthesis{maxime,
  TITLE = {{Contributions to the optimal transport problem and its regularity}},
  AUTHOR = {Prod'Homme,  M.},
  URL = {https://theses.hal.science/tel-03419872},
  NUMBER = {2021TOU30122},
  SCHOOL = {{Universit{\'e} Paul Sabatier - Toulouse III}},
  YEAR = {2021},
  MONTH = Oct,
  TYPE = {Theses},
  HAL_ID = {tel-03419872},
  HAL_VERSION = {v2},
}

@article{serfaty,
  title={Coulomb gases and Ginzburg-Landau vortices},
  author={Serfaty, S.},
  journal={arXiv preprint arXiv:1403.6860},
  year={2014}
}

@article{Zelada,
author = {D. Garc{\'i}a-Zelada},
title = {{Concentration for Coulomb gases on compact manifolds}},
volume = {24},
journal = {Electronic Communications in Probability},
number = {none},
publisher = {Institute of Mathematical Statistics and Bernoulli Society},
pages = {1 -- 18},
year = {2019},
doi = {10.1214/19-ECP211},
URL = {https://doi.org/10.1214/19-ECP211}
}

@article{Chafai,
title = {{Concentration for Coulomb gases and Coulomb transport inequalities}},
journal = {Journal of Functional Analysis},
volume = {275},
number = {6},
pages = {1447-1483},
year = {2018},
issn = {0022-1236},
doi = {https://doi.org/10.1016/j.jfa.2018.06.004},
url = {https://www.sciencedirect.com/science/article/pii/S0022123618302209},
author = {D. Chafaï and A. Hardy and M. Maïda}
}
\end{document}